\documentclass[11pt,reqno,english,a4]{amsart}
\usepackage{amsmath,amsthm,amsfonts,amssymb,amscd}
\usepackage[latin1]{inputenc}
\usepackage{psfrag}
\usepackage{epsfig}
\usepackage{a4wide}
\usepackage[]{graphicx}
\usepackage{times}
\usepackage{bm}
\usepackage{mathdots}
\usepackage{mathrsfs}
\usepackage[
breaklinks=true,colorlinks=true,
linkcolor=blue,urlcolor=blue,citecolor=blue,
bookmarks=true,bookmarksopenlevel=2]{hyperref}

\input{epsf.tex}
\newtheorem{theorem}{Theorem}[section]
\newtheorem{proposition}[theorem]{Proposition}
\newtheorem{lemma}[theorem]{Lemma}
\newtheorem{corollary}[theorem]{Corollary}
\newtheorem{definition}[theorem]{Definition}
\newtheorem{remark}[theorem]{Remark}
\newtheorem{notation}[theorem]{Notation}

\numberwithin{equation}{section}

\newcommand{\R}{\mathbb{R}}

\newcommand{\Pn}{\mathbb{P}}

\newcommand{\Aff}{\mathbb{A}\mathrm{ff}}

\newcommand{\el}{\mathscr{L}}

\newcommand{\q}{\mathfrak{q}}

\newcommand{\J}{\mathfrak{J}}

\newcommand{\Conf}{\mathrm{Conf}}
\newcommand{\dS}{\mathrm{dS}}
\newcommand{\AdS}{\mathrm{AdS}}
\newcommand{\Iso}{\mathrm{Iso}}
\newcommand{\PSL}{\mathrm{PSL}}
\newcommand{\SL}{\mathrm{SL}}

\newcommand{\Ein}{\mathbb{E}\mathrm{in}}
\newcommand{\Sn}{\mathbb{S}}
\newcommand{\En}{\mathbb{E}}
\newcommand{\Hn}{\mathbb{H}}
\newcommand{\g}{\mathfrak{g}}
\newcommand{\h}{\mathfrak{h}}
\newcommand{\ki}{\mathfrak{k}}
\newcommand{\s}{\mathfrak{sl}}
\newcommand{\aff}{\mathfrak{aff}}
\newcommand{\so}{\mathfrak{so}}

\newcommand{\Y}{\mathcal{Y}}

\newcommand{\Nu}{\mathfrak{N}}
\newcommand{\gi}{\mathsf{g}}
\begin{document}

\title[Cohomogeneity one actions on the three-dimensional Einstein universe]
      {Cohomogeneity one actions on the three-dimensional Einstein universe}
      
\author{M. Hassani}
\author{P. Ahmadi}

\thanks{}



\date{\today}

 \address{
M. Hassani\\
Departmental of mathematics\\
University of Zanjan\\
University blvd.\\
Zanjan\\
Iran}
 \email{masoud.hasani@znu.ac.ir}
 
  \address{
  And\\
Universit\'e d'Avignon, Campus Jean-Henri Fabre\\
301, rue Baruch de Spinoza,  BP 21239 F-84 916\\
AVIGNON Cedex 9.\\
France}
 \email{masoud.hassani@alumni.univ-avignon.fr }

\address{
P. Ahmadi\\
Departmental of mathematics\\
University of Zanjan\\
University blvd.\\
Zanjan\\
Iran}
 \email{p.ahmadi@znu.ac.ir}

 \begin{abstract}
 The aim of this paper is to classify the cohomogeneity one conformal actions on the $3$-dimensional Einstein universe $\Ein^{1,2}$, up to orbit equivalence. In a recent paper \cite{Has}, we studied the unique (up to conjugacy) irreducible action of $\PSL(2,\R)$ on $\Ein^{1,2}$ and we showed that the action is of cohomogeneity one. In the present paper, we   determine all the subgroups of $\Conf (\Ein^{1,2})$, up to conjugacy, acting reducibly and with cohomogeneity one on $\Ein^{1,2}$. We show that any cohomogeneity one reducible action on $\Ein^{1,2}$ induces a fixed point in the $4$-dimensional projective space $\R\Pn^4$. Also, we describe all the codimension one induced orbits by these actions.  
 \end{abstract}

 \maketitle
\tableofcontents
 \medskip
 \medskip

 \thispagestyle{empty}


\section*{Introduction}
Felix Klein is known for his work on the connections between geometry and group theory. By his 1872 Erlangen Program, geometries are classified by their underlying symmetry groups. According to his approach, a geometry is a $G$-space $M$ , where $G$ is a group of transformations of $M$. This makes a link between geometry and algebra. The most natural case occurs when the group $G$ acts on $M$ transitively. In this case $M$ is called a \textit{homogeneous} $G$-space. For instance Euclidean, affine and projective geometries are homogeneous spaces. 

One special case of non-transitive actions of transformation groups on manifolds is when the action has an orbit of codimension one, the so called \textit{cohomogeneity one} action. The concept of a cohomogeneity one action on a manifold $M$ was introduced by P.S. Mostert in his 1956 paper \cite{Mo}. The key hypothesis was the compactness of the acting Lie group in the paper. He assumed that the acting Lie group $G$ is compact and determined the orbit space up to homeomorphism. More precisely, he proved that by the cohomogeneity one action of a compact Lie group $G$ on a manifold $M$ the orbit space $M/G$ is homeomorphic to $\R$, $\Sn^1$, $[0,1]$, or $[0,1)$. In the general case, in \cite{B} B. Bergery showed that if a Lie group acts on a manifold properly and with cohomogeneity one, then the orbit space $M/G$ is homeomorphic to one of the above spaces. 

 A result by D. Alekseevsky in \cite{Alek} says that, for an arbitrary Lie group $G$, there is a complete $G$-invariant Riemannian metric on $M$ if and only if the action of $G$ on $M$ is proper. This theorem provides a link between proper actions and Riemannian $G$-manifolds.

Cohomogeneity one Riemannian manifolds have been studied by many mathematicians (see, e.g., \cite{AA, B, GWZ, GZ1, GZ2, MK, Mo, N, P, PS, S, V1, V2}). The subject is still an active one. The common hypothesis in the theory is that the acting group is closed in the full isometry group of the Riemannian manifold and the action is isometrically. When the metric is indefinite, this assumption in general does not imply that the action is proper, so the study becomes much more complicated. Also, some of the results and  techniques of definite metrics fail for indefinite metrics. 

The natural way to study a cohomogeneity one semi-Riemannian manifold $M$ is to determine the acting group in the full isometry group $\Iso(M)$, up to conjugacy, since the actions of two conjugate subgroups in $\Iso(M)$ admit almost the same orbits in $M$. This has been done for space forms in some special cases (see \cite{A1}, \cite{A2} and \cite{AhA}). This way is the one that we pursue in this paper.

Let $\R^{2,n+1}$ be the $(n+3)$-dimensional real vector space endowed with a quadratic form $\mathfrak{q}$ of signature $(2,n+1)$. The nullcone $\Nu^{2,n+1}$ of $\R^{2,n+1}$ is the set of non-zero vectors $v\in \R^{2,n+1}$ with $\mathfrak{q}(v)=0$. The nullcone $\Nu^{2,n+1}$ is a degenerate hypersurface of $\R^{2,n+1}$. The $(n+1)$-dimensional Einstein universe $\Ein^{1,n}$ is the projectivization of the nullcone $\Nu^{2,n+1}$ via $\Pn:\R^{2,n+1}\setminus \{0\}\rightarrow \R\Pn^{n+2}$. It is a compact submanifold of $\R\Pn^{n+2}$, and the degenerate metric on $\Nu^{2,n+1}$ induces a Lorentzian conformal structure on $\Ein^{1,n}$. The group $O(2,n+1)$, group of the linear isometries of $\R^{2,n+1}$, acts on Einstein universe conformally. Indeed, the full conformal group $\Conf(\Ein^{1,n})$ is $PO(2,n+1)$. There is a Lorentzian version of Liouville's Theorem..


\begin{theorem}(Liouville's Theorem \cite[Theorem 4.4]{Fr})\label{thm.lio}
Let $U,V\subset \Ein^{1,n}$ be non-empty connected open subsets, and $f:U\rightarrow V$ be a conformal map. Then $f$ extends to a unique global conformal map on $\Ein^{1,n}$.
\end{theorem}

Therefore, every locally conformally flat Lorentzian manifold of dimension $n+1\geq 3$ admits a $(\Conf(\Ein^{1,n}),\Ein^{1,n})$-structure. In particular, the Lorentzian model spaces of constant sectional curvatures $c$, namely, the Minkowski space $\En^{1,n}$ for $c=0$, the Anti de-Sitter space $\AdS^{1,n}$ for $c=-1$, and the de-Sitter space $\dS^{1,n}$ for $c=1$, all are conformally equivalent to some specific open dense subsets of $\Ein^{1,n}$. 

In this paper, we propose to study conformal actions of cohomogeneity one on the three-dimensional Einstein universe $\Ein^{1,2}$. Our strategy  is to determine the representation of the acting group in the Lie group $\Conf(\Ein^{1,2})$, up to conjugacy. Also, we describe the causal character of the codimension one orbits induced by cohomogeneity one actions on $\Ein^{1,2}$. For more details on the topology and causal character of all the orbits, we refer to \cite{Has0}.

The group of conformal transformations of $\Ein^{1,2}$ is isomorphic to $O(2,3)$. The identity component of $O(2,3)$ is $SO_\circ(2,3)$ which acts on $\Ein^{1,2}$ transitively. Hence, a cohomogeneity one action on $\Ein^{1,2}$ comes from a proper subgroup of  $\Conf(\Ein^{1,2})$. 

The proper actions are significant and they deserve to be studied first. Let $G$ be a connected Lie subgroup of $\Conf(\Ein^{1,2})$. In Section \ref{proper}, we show that up to conjugacy there are exactly two Lie groups acting properly and with cohomogeneity one on $\Ein^{1,2}$, namely, $SO(3)$ and $SO(2)\times SO(2)$. The action of $SO(3)$ on Einstein universe $\Ein^{1,2}$ admits a codimension one foliation on which each leaf is a \textit{spacelike hypersphere}. On the other hand, $SO(2)\times SO(2)$ admits a unique $1$-dimensional orbit, and on the complement admits a codimension one foliation on which every leaf is conformally equivalent to the $2$-dimensional Einstein universe $\Ein^{1,1}$.

Rolling out the proper case already discussed above, in the non-proper case, we deal with subgroups of some known geometric groups such as: $(\R^*\times O(1,2))\ltimes \R^{1,2}$ group of conformal transformations on the $3$-dimensional Minkowski space $\En^{1,2}$, $O(2,2)$ group of isometries of $3$-dimensional Anti de-Sitter space $\AdS^{1,2}$, and $O(1,3)$ group of isometries of the $3$-dimensional de-Sitter space $\dS^{1,2}$. In particular, we consider various representations of \textit{M}$\ddot{o}$\textit{bius group} $\PSL(2,\R)$ and its subgroups in $O(2,3)$. In most of the cases, we determine the Lie subgroups acting with cohomogeneity one on $\Ein^{1,2}$ considering their corresponding subalgebras in the Lie algebra $\so(2,3)=Lie(O(2,3))$. 

In order to study the non-proper cohomogeneity one actions we have the following definition.
\begin{definition}\label{def}
The action of a Lie subgroup $G\subset \Conf(\Ein^{1,2})$ is called \textit{irreducible} if $G$ preserves no non-trivial linear subspace of $\R^{2,3}$, and it is called \textit{reducible} if it is not irreducible.
\end{definition}
\begin{theorem}\label{1}\cite{Di}
Let $G$ be a connected Lie subgroup of $SO_\circ(2,n)$ which acts on $\R^{2,n}$ irreducibly. Then $G$ is conjugate to one of the following Lie groups:
\begin{itemize}
\item[i:] for arbitrary $n\geq 1$: $SO_\circ(2,n)$.
\item[ii:] for $n=2p$ even: $U(1,p)$, $SU(1,p)$ or $S^1.SO_\circ(1,p)$ if $p>1$,
\item[iii:] for $n=3$: $SO_\circ(1,2)\subset SO_\circ(2,3)$.
\end{itemize}
\end{theorem}

The following result on irreducible cohomogeneity one actions on $\Ein^{1,2}$ is proved in \cite{Has}.
\begin{theorem}\label{thm10}
Let $G\subset \Conf(\Ein^{1,2})$ be a connected Lie subgroup which acts on $\Ein^{1,2}$ irreducibly and with cohomogeneity one. Then $G$ is conjugate to $SO_\circ(1,2)\simeq\PSL(2,\R)$.
\end{theorem}


According to Theorem \ref{1}, one way to classify the cohomogeneity one reducible actions is to consider the stabilizer of various (in dimension and signature) linear subspaces of $\R^{2,3}$ by the action of $SO_\circ(2,3)$. The following theorem shows that considering the actions preserving a $1$-dimensional linear subspace of $\R^{2,3}$ is enough (even the proper actions are included in the following).

\begin{theorem}\label{thm.1}
Let $G$ be a connected Lie subgroup of $\Conf(\Ein^{1,2})$ which acts on $\Ein^{1,2}$ reducibly and with cohomogeneity one. Then $G$ fixes a point in the projective space $\R\Pn^4$.
\end{theorem}
\begin{proof}
It follows from Proposition \ref{prop.all}, Proposition \ref{cor.proper}, and Theorem \ref{thm.photon}.
\end{proof}
By Theorem \ref{thm.1}, every reducible cohomogeneity one action of $G$ on $\Ein^{1,2}$ preserves a line $\ell\leq \R^{2,3}$. The line $\ell$ can be lightlike, spacelike or timelike. 

The case when $G$ preserves a lightlike line $\ell$ is the richest one, i.e., actions fixing a point in Einstein universe, which are fully studied in Section \ref{fix}. By the action of $O(2,3)$, the stabilizer of a point in Einstein universe is isomorphic to the group of conformal transformations of the Minkowski space. More precisely, if $G$ fixes a point $p\in \Ein^{1,2}$, then it preserves the \textit{lightcone} $L(p)$ and its corresponding \textit{Minkowski patch}. Hence, the action of $G$ on the Minkowski patch is equivalent to the action of a Lie subgroup of $\Conf(\En^{1,2})=(\R^*\times O(1,2))\ltimes \R^{1,2}$ on Minkowski space $\En^{1,2}$. We apply the adjoint action of $\Conf(\En^{1,2})$ on its Lie algebra $(\R\oplus\so(1,2))\oplus_\theta \R^{1,2}$, and then, we determine all the Lie subalgebras $\g$ of $(\R\oplus\so(1,2))\oplus_\theta \R^{1,2}$ with $\dim \g\geq 2$, up to conjugacy. This leads to the classification of the connected Lie subgroups of $\Conf(\En^{1,2})$ with dimension greater than or equal to $2$, up to conjugacy (Theorem \ref{thm.3.2.1}). Actually, there are infinitely many subgroups of $\Conf(\En^{1,2})$ with $\dim \geq 2$, up to conjugacy. All those subgroups act on Einstein universe $\Ein^{1,2}$ with cohomogeneity one $\R^{1,2}$ and $\R^*_+\ltimes \R^{1,2}$ (Theorem \ref{thm.fix}).

In Section \ref{de}, we study the actions preserving a non-degenerate line $\ell\leq \R^{2,3}$. Indeed, in this chapter we only consider the cohomogeneity one actions on $\Ein^{1,2}$ with no fixed point in $\Ein^{1,2}$. 
\begin{itemize}
\item If $G$ preserves a spacelike line $\ell$, then it preserves an \textit{Einstein hypersphere} $\Ein^{1,1}\subset \Ein^{1,2}$, and its complement $\Ein^{1,2}\setminus \Ein^{1,1}$ which is conformally equivalent to the \textit{anti de-Sitter space} $\AdS^{1,2}$.  In this case, the action of $G$ on $\Ein^{1,2}$ is conformally equivalent to the natural action of a Lie subgroup of the group of isometries of $\AdS^{1,2}$, namely, $O(2,2)$. We show that, up to orbit equivalency there are exactly seven connected Lie subgroups which preserve an Einstein hypersphere $\Ein^{1,1}$ in $\Ein^{1,2}$, act on $\Ein^{1,2}$ non-properly, with cohomogeneity one, and with no fixed point in $\Ein^{1,2}$.

\item If $G$ preserves a timelike line, then it preserves a spacelike hypersphere $\Sn^2\subset \Ein^{1,2}$, and its complement $\Ein^{1,2}\setminus \Sn^2$ which is conformally equivalent to the \textit{de-Sitter space} $\dS^{1,2}$. The action of $G$ is conformally equivalent to the action of a Lie subgroup of the group of isometries of $\dS^{1,2}$, namely, $O(1,3)$. It is shown that, up to conjugacy, there is only one connected Lie subgroup which preserves a spacelike hypersphere in $\Ein^{1,2}$, acts non-properly and with cohomogeneity one, and fixes no point in $\Ein^{1,2}$.
\end{itemize}

Actions preserving a photon in $\Ein^{1,2}$ are very interesting. In Section \ref{photon}, we prove a theorem which states that a cohomogeneity one action on $\Ein^{1,2}$ preserving a photon, fixes a point in $\R\Pn^4$. This theorem will play a key role in proof of Theorem \ref{thm.1} .

\section{Preliminary}
We fix the following notation and definitions for the rest of this work. 

Let $\R^{m+n}$ be the $(m+n)$-dimensional real vector space. Assume that $\mathfrak q_{m,n}$ is a quadratic form on $\R^{m+n}$ of signature $(m,n)$, i.e., in a suitable basis for $\R^{m+n}$ we have: 
\[\q_{m,n}(v)=-\underset{i=1}{\overset{m}{\sum}} v_i^2+\underset{j=m+1}{\overset{m+n}{\sum}} v_j^2,\;\;\;\;\;\;\;\;\;v=(v_1,\cdots,v_{m+n})\in \R^{m+n}.\]
 We denote by $\R^{m,n}$ the vector space $\R^{m+n}$ equipped with $\mathfrak q_{m,n}$. Also, we denote by $O(m,n)$ the Lie group of linear isometries of $\R^{m,n}$, and by $SO_\circ(m,n)$ its identity component.
 \begin{definition}
 A linear subspace $V\leq \R^{m,n}$ is said to be of signature $(p,q,r)$ if the restriction of $\mathfrak{q}_{m,n}$ on $V$ is of signature $(p,q,r)$, meaning that the maximal totally isotropic subspace has dimension $r$, and that the maximal negative and positive definite subspaces have dimensions $p$ and $q$, respectively. If $V$ is nondegenerate (i.e., $r=0$), we omit $r$ and simply denote its signature by $(p,q)$. Also, we call $V$
 \begin{itemize}
 \item spacelike, if $q\neq 0$ and $p=r=0$.
 \item timelike, if $p=1$ and $q=r=0$.
 \item lightlike, if $r=1$ and $p=q=0$.
 \item Lorentzian, if $p=1$, $q\neq 0$, and $r=0$.
 \item degenerate, if $p+q,r\neq 0$.
 \end{itemize}
 \end{definition}
 A non-zero vector $v\in \R^{m,n}$  is called spacelike (resp. timelike, lightlike) if the value $\mathfrak{q}_{m,n}(v)$ is positive (resp. negative, zero).
 \begin{definition}
 Let $M$ be a smooth manifold equipped with a semi-Riemannian metric $\gi$ (resp. conformal structure $[\gi]$) and $S\subset M$ an immersed submanifold. Then, $S$ is said to be of signature $(p,q,r)$ if for all $x\in S$ the restriction of the ambient metric $\gi_x$ (resp. the conformal structure $[\gi]_x$) on the tangent space $T_xS$ is of signature $(p,q,r)$.
 \end{definition}

 The $(n+1)$-dimensional \textit{Minkowski space} $\En^{1,n}$ is an affine space whose underlying vector space is the Lorentzian scalar product space $\R^{1,n}$. It is characterized up to isometry as a simply connected, geodesically complete, flat Lorentzian manifold. The isometry group of the Minkowski space $\En^{1,n}$ is isomorphic to the semi-direct product $O(1,n)\ltimes \R^{1,n}$ where the Lorentz group $O(1,n)$ acts naturally on $\R^{1,n}$ by $(A,v)\mapsto A(v)$.

 Geodesics in the Minkowski space $\En^{1,n}$ are affine lines, i.e., they have the form 
\[\gamma:\R\rightarrow \En^{1,n},\;\;\;\;\;\;t\mapsto p+tv,\]
where $p\in \En^{1,n}$ is a point and $v\in \R^{1,n}$ is a vector. 

A homothety on $\En^{1,n}$ (centered at $x_0\in \En^{1,n}$) is any map conjugate by a translation to a scalar multiplication:
\[\En^{1,n}\rightarrow \En^{1,n},\;\;\;\;\;\; x\mapsto x_0+r(x-x_0).\]
A conformal map on the Minkowski space $\En^{1,n}$ is a composition of an isometry of $\En^{1,n}$ with a homothety:
\[f:x\rightarrow rA(x)+v.\]
where $A\in O(1,n)$, $r\neq 0$, and $v\in \R^{1,n}$. We denote the group of conformal transformations of $\En^{1,n}$ by $\Conf(\En^{1,n})$. It is evidently isomorphic to the semi-direct product $(\R^*\times O(1,n))\ltimes \R^{1,n}$.

\begin{definition}
The de-Sitter hypersphere in the Minkowski space $\En^{1,n}$ of radius $r$ centered at $p\in \En^{1,n}$ is 
\[S_r(p)=\{p+v\in \En^{1,n}:\q_{1,n}(v)=r^2\}.\]
\end{definition}
\subsection{Projective special linear group}\label{proj}
The projective special linear group $\PSL(2,\R)$ is the quotient  of $\SL(2,\R)$ by $\mathbb{Z}_2=\{\pm Id\}$. The group $\PSL(2,\R)$ acts on the Poincaré half-plane model of hyperbolic plane with conformal boundary in infinity $\overline{\Hn}^2=\Hn^2\cup \partial\Hn^2$ by \textit{m}$\ddot{o}$\textit{bius transformation}
\[x\mapsto \dfrac{ax+b}{cx+d}, \;\;\;\;a,b,c,d\in \R,\;\;ad-bc=1.\]
This action preserves the hyperbolic plane $\Hn^2$ and its boundary $\partial\Hn^2\simeq \R\Pn^1$. Indeed, $\PSL(2,\R)$ is the group of orientation-preserving isometries of hyperbolic plane, and acts on it transitively. Furthermore it acts on the conformal boundary of hyperbolic plane conformaly and transitively. We refer to \cite{Ra} for more details.
 \begin{definition}
 Let $A\in \PSL(2,\R)$ be a non-trivial element. We call $A$
 \begin{itemize}
 \item an elliptic element, if it fixes a point in the hyperbolic plane $\Hn^2$.
 \item a parabolic element, if it fixes a unique point on the boundary $\partial\Hn^2$.
 \item a hyperbolic element, if it fixes exactly two distinct points on the boundary $\partial\Hn^2$.
 \end{itemize}
 \end{definition}
Every non-trivial element in $\PSL(2,\R)$ is either elliptic, parabolic, or hyperbolic. Actually, all the non-trivial elements in a $1$-parameter subgroup $G\subset \PSL(2,\R)$ are elliptic (resp. parabolic, hyperbolic) if $G$ contains an elliptic (resp. parabolic, hyperbolic) element. So, we can use the terminology elliptic, parabolic and hyperbolic for the $1$-parameter subgroups of $\PSL(2,\R)$. It is well-known that $\PSL(2,\R)$ has three distinct $1$-parameter subgroup, up to conjugacy. Also, it has a unique (up to conjugacy) $2$-dimensional connected Lie subgroup, which is solvable and non-abelian.
\begin{notation}
By the action of $\PSL(2,\R)$ on $\overline{\Hn}^2=\Hn^2\cup \R\Pn^1$, we denote by $Y_E$ the stabilizer of the point $i\in \Hn^2$ which is a $1$-parameter elliptic subgroup. Also, we denote by $\Aff$ the stabilizer of $\infty\in \R\Pn^1\approx\partial\Hn^2$, which is a $2$-dimensional connected Lie subgroup. Furthermore, we denote by $Y_P$ the derived subgroup of $[\Aff,\Aff]$ which is a $1$-parameter parabolic subgroup. Finally, we denote by $Y_H$ the $1$-parameter subgroup of $\Aff$ fixing $\{0,\infty\}\subset \R\Pn^1\approx\partial\Hn^2$.
\end{notation}
%

%
The Lie algebra of $\PSL(2,\R)$ is isomorphic to the simple Lie algebra $\s(2,\R)$, the Lie algebra of $2\times 2$ traceless matrices. The determinant function $\det:\s(2,\R)\rightarrow \R$ is a quadratic form of signature $(2,1)$. Thus, $(\s(2,\R),-\det)$ is a model for the $3$-dimensional Minkowski space. The group $\SL(2,\R)$ acts on $(\s(2,\R),-\det)$ linearly and isometrically by 
\[(A,X)\mapsto AXA^{-1}.\]
The kernel of this action is $\{\pm Id\}$. This implies that $\PSL(2,\R)$ is isomorphic to $SO_\circ(1,2)$.

 The action of $\Conf_\circ(\s(2,\R),-\det)=(\R_+^*\times \PSL(2,\R))\ltimes \s(2,\R)$ on the Minkowski space $(\s(2,\R),-\det)$ is as follows:
\begin{align*}
(r,[A],V).X:=rAXA^{-1}+V,\;\;\;\;\;(r,[A],V)\in (\R_+^*\times \PSL(2,\R))\ltimes \s(2,\R),\;\;\;X\in \s(2,\R),
\end{align*}
where $A$ is an arbitrary representative of $[A]\in \PSL(2,\R)=\SL(2,\R)/\{\pm Id\}$.

The set of following matrices is a basis for $\s(2,\R)$ as a vector space,
\begin{align*}
\Y_E=\begin{bmatrix}
0 & 1\\
-1 & 0
\end{bmatrix},\;\;\;\;\;
\Y_H=\begin{bmatrix}
1 & 0\\
0 & -1
\end{bmatrix},\;\;\;\; \Y_P=\begin{bmatrix}
0 & 1\\
0 & 0
\end{bmatrix},
\end{align*}
and we have $[\Y_E,\Y_H]=2\Y_E-4\Y_P$, $[\Y_E,\Y_P]=\Y_H$ and $[\Y_H,\Y_P]=2\Y_P$. In fact, $\Y_E$ (resp. $\Y_P$, $\Y_H$) is the generator of the Lie algebra of $Y_E$ (resp. $Y_P$, $Y_H$). Also, $\{\Y_P,\Y_H\}$ is a basis for the affine algebra $\aff$, the Lie algebra of the affine group $\Aff=Y_H\ltimes Y_P$.

The following two lemmas are needed in the sequel, which may be obtained easily (see \cite{Has0}).
\begin{lemma}\label{lem.1.4.10}
The identity component of the group of automorphisms of the affine group $\Aff$,  and that of the group of automorphisms of affine algebra $\aff$ are isomorphic to $\Aff$.
 \end{lemma}
 
 \begin{lemma}\label{lem.1.4.7}
Let $G$ be a Lie group with $1\leq \dim G\leq 2$, and $\varphi:\PSL(2,\R) \rightarrow G$ be a Lie group morphism. Then $\varphi$ is the trivial map.
\end{lemma}
\begin{remark}\label{rem.SL}
The map $\SL(2,\R)\rightarrow \PSL(2,\R)$ sending $A$ to $[A]$ is a Lie group double covering. We call an element $A\in \SL(2,\R)$ elliptic, parabolic, or hyperbolic if the corresponding element $[A]\in \PSL(2,\R)$ is elliptic, parabolic, or hyperbolic, respectively. Also, one can see that the group $\SL(2,\R)$ has three distinct $1$-parameter subgroup up to conjugacy, and it has a unique $2$-dimensional connected Lie subgroup isomorphic to $\Aff$ up to conjugacy. Hence, we may use the same terminology and notations for the elements and subgroup of $\SL(2,\R)$ as we used for those of $\PSL(2,\R)$ when there is no ambiguity. We do this for the Lie groups $SO_\circ(1,2)$ and $SO_\circ(2,1)$ as well, since they are isomorphic to $\PSL(2,\R)$.
\end{remark}


Here, we remind some facts about the elements, connected subgroups, and the Lie algebra of $SO_\circ(1,2)$. It is easy to see that $1$ is the common eigenvalue of all the elements in $SO_\circ(1,2)$.
\begin{definition}
A non-trivial element $A\in SO_\circ(1,2)$ is called
\begin{itemize}
\item elliptic, if it has two distinct non real eigenvalues.
\item parabolic, if $1$ is its only eigenvalue.
\item hyperbolic, if it has two distinct real eigenvalues.
\end{itemize}
\end{definition}
It can be easily seen that the eigenspace of an elliptic element in $SO_\circ(1,2)$ is a timelike line. Also, the eigenspace of a parabolic element is a lightlike line. Furthermore, the eigenspaces of a hyperbolic element are two distinct lightlike lines and a spacelike line. A $1$-parameter subgroup of $SO_\circ(1,2)$ is called elliptic (res. parabolic, hyperbolic), if it contains an elliptic (resp. parabolic, hyperbolic) element.

Let $\{e_1,e_2,e_3\}$ be an orthonormal basis for $\R^{1,2}$ on which $e_1$ is a timelike vector. The set of the following matrices is a basis for the Lie algebra $\so(1,2)=Lie (SO_\circ(1,2))$ as a vector space
\begin{align*}
\Y_E=\begin{bmatrix}
0 & 0 & 0\\
0 & 0 & 1\\
0 & -1 &0
\end{bmatrix},\;\;\;\;\Y_H=\begin{bmatrix}
0 & 1 & 0\\
1 & 0 & 0\\
0 & 0 & 0
\end{bmatrix},\;\;\;\;
\Y_P=\begin{bmatrix}
0 & 0 & 1\\
0 & 0 & 1\\
1 & -1 & 0
\end{bmatrix},
\end{align*}
and we have, $[\Y_E,\Y_H]=\Y_E-\Y_P$, $[\Y_E,\Y_P]=\Y_H$ and $[\Y_H,\Y_P]=\Y_P$. The affine subalgebra $\aff$ is generated by $\{\Y_P,\Y_H\}$. We denote by $Y_E$, $Y_P$, and $Y_H$ the $1$-parameter subgroups of $SO_\circ(1,2)$ generated by the vector $\Y_E,\Y_P,\Y_H\subset\so(1,2)$, respectively.
The group $SO_\circ(1,2)$ has a unique $2$-dimensional connected Lie subgroup up to conjugacy. It is the Lie group associate to the Lie subalgebra $\aff\leq \so(1,2)$ and we denote it by $\Aff$.
\subsection{Einstein universe}\label{sec.Ein}
In this section we give the definition and some properties of Einstein universe $\Ein^{1,n}$. For more details, we refer to \cite{Fr}, and \cite{Bar}.

Consider the scalar product space $\R^{2,n+1}=(\R^{n+3},\q_{2,n+1})$. The nullcone of $\R^{2,n+1}$ is the set of all non-zero null vectors in $\R^{2,n+1}$, and we denote it by $\Nu^{2,n+1}$.
\[\Nu^{2,n+1}=\{v\in \R^{2,n+1}\setminus \{0\}: \q_{2,n+1}(v)=0\}.\]
The nullcone is a degenerate hypersurface in $\R^{2,n+1}$ invariant by $O(2,n+1)$.
\begin{definition}\label{def.1.7.1}
The $(n+1)$-dimensional Einstein universe $\Ein^{1,n}$ is the image of $\Nu^{2,n+1}$
under the projectivization:
\[\Pn:\R^{2,n+1}\setminus\{0\}\longrightarrow \R\Pn^{n+2}.\]
\end{definition}
 In the sequel, for notational convenience, we will denote $\Pn$ as a map from $\R^{2,n+1}$
implicitly assuming that the origin $0$ is removed from any subset of $\R^{2,n+1}$ on which we apply $\Pn$.

The degenerate metric on the nullcone $\mathfrak{N}^{2,n+1}$ induces an $O(2,n+1)$-invariant conformal Lorentzian structure on $\Ein^{1,n}$. Indeed, the group of conformal transformations of Einstein universe $\Ein^{1,n}$ is $PO(2,n+1)$ (see \cite{Fr}).
\begin{definition}\label{def.1.7.2}
The double covering space $\widehat{\Ein}^{1,n}$ of Einstein universe is the quotient space of the null-cone $\Nu^{2,n+1}$ by the action of positive scalar multiplication.
\end{definition}
For many purposes the double covering may be more useful than $\Ein^{1,n}$, itself. The double covering space $\widehat{\Ein}^{1,n}$ is conformally equivalent to $(\Sn^1\times \Sn^n,-dt^2+ds^2)$, where $dt^2$ and $ds^2$ are the usual round metrics on the spheres $\Sn^1$ and $\Sn^n$ of radius one (see \cite{Bar}).

 \begin{remark}
It is remarkable that in a conformal structure $(M,[\gi])$, if $\gamma$ is a geodesic respect to a metric $\gi\in [\gi]$, it is not a geodesic respect to all the metrics in $[\gi]$, necessarily. However, if $\gi$ is not a Riemannian metric (i.e. it is not positive or negative definite), then any lightlike geodesic respect to $\gi$ is a (unparameterized) lightlike geodesic respect to all the metrics in class $[\gi]$.
\end{remark}
 
 \begin{definition}\label{def.1.7.3}
  A photon in Einstein universe $\Ein^{1,n}$ is the projectivization of a totally isotropic $2$-plane in $\R^{2,n+2}$. We denote the set of all photons by $\mathrm{Pho}^{1,n}$.
 \end{definition}
 In fact, a photon is a lightlike geodesic in the Einstein universe and  it is naturally homeomorphic to the real projective line $\R\Pn^1$.
  \begin{definition}\label{def.1.7.4}
Given any point $[p]\in \Ein^{1,n}$ the lightcone $L([p])$ with vertex $[p]$ is the union of all photons containing $[p]$: 
 \[L([p]):=\{\phi\in \mathrm{Pho^{1,n}:[p]\in \phi}\}.\]
 \end{definition}
 A lightcone $L([p])$ can be equivalently defined as the projectivization of the orthogonal complement $p^\perp\cap \Nu^{2,n+1}$. Furthermore, $L([p])$ is a singular hypersurface with  the only singular point $[p]$. We denote the lightcone with removed vertex by $L([\hat{p}])$ and call it a \textit{vertex-less lightcone}. It is easy to see that a vertex-less lightcone is  homeomorphic to $\Sn^{n-1}\times \R$. We also denote a photon $\phi\subset L([p])$ with removed the vertex $[p]$ by $\hat{\phi}$ and call it a \textit{vertex-less photon}.
 
 \begin{definition}\label{def.1.7.5}
 Given any point $[p]\in \Ein^{1,n}$, the Minkowski patch $Mink([p])$ is the complement of the lightcone $L([p])$ in $\Ein^{1,n}$. 
 \end{definition}
 \begin{proposition}\label{prop.1.7.6}\cite[Lemma 4.11]{Fr}
 For an arbitrary point $[p]\in \Ein^{1,n}$, the Minkowski patch $Mink([p])$ is conformally equivalent to the Minkowski space $\En^{1,n}$.
 \end{proposition}

\begin{definition}\label{def.1.7.7}
Let $[p]\in \Ein^{1,n}$ and $\gamma(t)=x+tv$ be a lightlike geodesic in the Minkowski patch $Mink([p])\approx\En^{1,n}$. We call a point $[q]\in L([p])$ the limit point of $\gamma$ in $L([p])$, if $ \lim_{t\rightarrow \pm\infty}\gamma(t)=[q]$. Similarly, for an affine degenerate hyperplane $\Pi\subset Mink([p])$, we call a point $[q]\in L([p])$ the limit point of $\Pi$ in $L([p])$ if for all lightlike geodesics $\gamma(t)\subset \Pi$, $\lim_{t\rightarrow \pm\infty}\gamma(t)=[q]$.
\end{definition}

\begin{proposition}\cite[Lemma 4.12]{Fr}\label{prop.1.7.10}
Let $[q]\in L([\hat{p}])$ be an arbitrary point. Then there exists a unique affine degenerate hyperplane $\Pi$ in $Mink([p])$, such that $[q]$ is the limit point of $\Pi$. Furthermore, two points $[q_1],[q_2]\in L([\hat{p}])$ lie in the same photon in $L([p])$ if and only if their corresponding degenerate affine hyperplanes in $Mink([p])$ are parallel. 
\end{proposition}
%
\begin{lemma}\cite[p.p. 56]{Fr}
By the action of $O(2,n)$ on Einstein universe $\Ein^{1,n}$, the stabilizer of a point in $\Ein^{1,n}$ is isomorphic to $\Conf(\En^{1,n})\simeq (\R^*\times O(1,n))\ltimes \R^{1,n}$.
\end{lemma}

Let $\phi$ be a photon in Einstein universe $\Ein^{1,n}$. The complement of $\phi$ in $\Ein^{1,n}$ is an open dense subset and we denote it by $\Ein^{1,n}_\phi$. It is diffeomorphic to $\Sn^1\times \R^n$. There is a natural codimension $1$ foliation $\mathcal{F}_\phi$ on $\Ein^{1,n}_\phi$ on which the leaves are degenerate hypersurfaces diffeomorphic to $\R^n$. Fixing a point $[p]\in \phi$, the complement of $\phi$ in the lightcone $L([p])$ is a leaf of $\mathcal{F}_{\phi}$. The other leaves are the degenerate affine hyperplanes in $Mink([p])$ with limit point in $\phi$. The group of conformal transformations of $\Ein^{1,n}_\phi$ acts on $\Ein^{1,n}_\phi$ transitively and the action preserves the foliation $\mathcal{F}_\phi$. In fact, $\Conf(\Ein^{1,n}_\phi)$ is the stabilizer of $\phi$ by the action of $PO(2,n)$ (see \cite{Fr}).
\begin{lemma}\cite[Lemma 4.15]{Fr}\label{lem.Hiz}
By the action of $PO(2,n)$ on Einstein universe $\Ein^{1,n}$, the stabilizer of a photon is isomorphic to $(\R^*\times \SL(2,\R)\times O(n-2))\ltimes H(2n-3)$, where $H(2n-3)$ is the $(2n-3)$-dimensional Heisenberg group.
\end{lemma}

\begin{definition}\label{def.1.7.14}
A spacelike hypersphere in $\Ein^{1,n}$ is the one-point compactification of a spacelike affine hyperplane of a Minkowski patch in $\Ein^{1,n}$.
\end{definition}
 Equivalently, spacelike hyperspheres are projectivizations of $v^\perp\cap \Nu^{2,n+1}$ for timelike vectors $v\in \R^{2,n+1}$. It is easy to see that a spacelike hypersphere is conformally equivalent to the round sphere $\Sn^n$.
 \begin{definition}\label{def.1.7.17}
An Einstein hypersphere in $\Ein^{1,n}$ is the closure of a de-Sitter hypersphere $S_r([q])\subset Mink([p])$ in $\Ein^{1,n}$, for some $[p]\in \Ein^{1,n}$ and $[q]\in Mink([p])$.
\end{definition}
The Einstein hyperspheres are projectivizations of $v^\perp\cap\Nu^{2,n+1}$ for spacelike vectors $v\in \R^{2,n+1}$. For $n\geq 2$, every Einstein hypersphere in $\Ein^{1,n}$ has a natural structure of $\Ein^{1,n-1}$.
\begin{remark}\label{rem.ein}
For an Einstein hypersphere $\Ein^{1,n-1}\subset \Ein^{1,n}$ and an arbitrary point $x_0\in \Ein^{1,n-1}$, the intersection of $\Ein^{1,n-1}$ with the Minkowski patch $Mink(x_0)\subset \Ein^{1,n}$ (the Minkowski patch of $x_0$ as a point in $\Ein^{1,n}$) is an affine Lorentzian hyperplane. Thus, an Einstein hypersphere in $\Ein^{1,n}$ is the conformal compactification of a Lorentzian affine hyperplane in a Minkowski patch (it is the Lorentzian analog of Definition (\ref{def.1.7.14})). 
\end{remark}
 \begin{definition}\label{def.1.7.15}
A de-Sitter component in $\Ein^{1,n}$ is the complement of a spacelike hypersphere $\Sn^{n}\subset\Ein^{1,n}$. It is homeomorphic to $\Sn^n \times \R$ and evidently, its conformal boundary is $\Sn^n$.
\end{definition}
\begin{definition}\label{def.1.7.18}
For $n\geq 2$, an Anti de-Sitter component in $\Ein^{1,n}$ is the complement of an Einstein hypersphere $\Ein^{1,n-1}\subset  \Ein^{1,n}$. It is homeomorphic to $\Sn^1 \times \R^n$ and evidently, its boundary is $\Ein^{1,n-1}$.
\end{definition}
\begin{lemma}\label{lem.1.7.16}\cite[p.p. 58]{Fr}
A de-Sitter (resp. Anti de-Sitter) component in $\Ein^{1,n}$ is conformally equivalent to the model de-Sitter space $dS^{1,n}=\q_{1,n+1}^{-1}(1)\subset \R^{1,n+1}$ (resp. the model Anti de-Sitter space $AdS^{1,n}=\q_{2,n}^{-1}(-1)\subset \R^{2,n}$).
\end{lemma}

\subsubsection{\underline{Two dimensional Einstein universe}}\label{subsec.1.7.3}
The $2$-dimensional Einstein universe $\Ein^{1,1}$ is diffeomorphic to a $2$-torus. 

Here is a useful model. Let $(M(2,\R),-\det)$ denote the vector space consisting the $2\times 2$ real matrices endowed with the $(2,2)$-bilinear form associated to the negative determinant. In this model, the nullcone $\Nu^{2,2}$ is the set of nonzero singular matrices and the $2$-dimensional Einstein universe $\Ein^{1,1}$ is the quotient of the nullcone by the action of nonzero scalar multiplication. Observe that the direct product group $\SL(2,\R)\times \SL(2,\R)$ acts on $(M(2,\R),-\det)$ by 
\[((A,B),X)\mapsto AXB^{-1},\;\;\;\;\;\;\;X\in M(2,\R),\;\;A,B\in \SL(2,\R).\]
Obviously, this action is linear and  preserves the quadratic form $-\det$. In particular, its kernel is $\{\pm(Id,Id)\}\simeq \mathbb{Z}_2$. This implies that $SO_\circ(2,2)$ is isomorphic to $(\SL(2,\R)\times \SL(2,\R))/\mathbb{Z}_2$. Moreover, by the action of $\SL(2,\R)\times \SL(2,\R)$ on $2$-dimensional Einstein universe $\Ein^{1,1}=\Pn(\Nu^{2,2})$, the kernel is $\{(\pm Id,\pm Id)\}$. Hence, $\Conf_\circ(\Ein^{1,1})$ is isomorphic to $\PSL(2,\R)\times \PSL(2,\R)$.

Every nonzero singular element $X\in M(2,\R)$ determines two lines in $\R^2$: its kernel and its image. Let $(A,B)\in \PSL(2,\R)\times \PSL(2,\R)$ be an arbitrary element. Then $A$ preserves $\ker X$ and $B$ preserves $\mathrm{Im}\; X$. Therefore, there is a canonical $\PSL(2,\R)\times \PSL(2,\R)$-invariant identification of $\Ein^{1,1}$ with $\R\Pn^1\times\R\Pn^1$ by
 \[\Ein^{1,1}\rightarrow \R\Pn^1\times \R\Pn^1,\;\;\;\;\;\; [X]\mapsto ( [\ker X],[\mathrm{Im}\; X]).\]
By this identification, the left factor of $\PSL(2,\R)\times \PSL(2,\R)$ acts trivially on the left factor $\R\Pn^1$, and the right factor acts trivially on the right factor $ \R\Pn^1$.

The following corollary will be useful in sequel.
\begin{corollary}\label{cor.Anti}
The $3$-dimensional Anti de-Sitter space $\AdS^{1,2}$ is isometric to $\SL(2,\R)$ endowed with the metric induced from $(M(2,\R),-\det)$.
\end{corollary}
\subsubsection{Three dimensional Einstein universe}\label{subsec.1.7.4}

 \begin{definition}\label{def.1.7.20}
 Let $[p],[q]\in \Ein^{1.2}$ be two distinct points and they do not lie on a common photon. Then, the intersection of the lightcones $L([p])$ and $L([q])$ is called an ideal circle.
 \end{definition}
 \begin{lemma}\label{lem.1.7.21}
 An ideal circle is the projectivized nullcone of a linear subspace of $\R^{2,3}$ of signature $(1,2)$.
 \end{lemma}
 \begin{proof}
 Observe that, the intersection of two degenerate hyperplanes $p^\perp,q^\perp\leq \R^{2,3}$ is a linear subspace of signature $(1,2)$. Now, the lemma follows easily.
 \end{proof}
\begin{definition}\label{def.1.7.22}
A timelike circle in $\Ein^{1,2}$ is the projectivized nullcone of a linear subspace of $\R^{2,3}$ of signature $(2,1)$.
\end{definition}
\begin{lemma}\label{lem.1.7.23}
The complement of a timelike circle in $\Ein^{1,2}$ is conformally equivalent (up to double cover) to $(\AdS^{1,1}\times \Sn^1,d\sigma^2+d\theta^2)$ where $d\sigma^2$ (resp. $d\theta^2$) is the usual Lorentzian metric on $\AdS^{1,1}$ of constant sectional curvature $-1$ (resp. positive definite metric on $\Sn^1$ of radius $1$). Furthermore, by the action of $O(2,3)$, the identity component of the stabilizer of a timelike circle is isomorphic to the direct product $SO_\circ(2,1)\times SO(2)$.
\end{lemma}
\begin{proof}
Let $\mathfrak{C}\subset \Ein^{1,2}$ be a timelike circle and $V\leq \R^{2,3}$ the linear subspace corresponding to $\mathfrak{C}$. The orthogonal complement $V^\perp$ is of signature $(0,2)$. Choose an orthonormal basis $\{e_1,e_2,e_3\}$ for $V$ where $e_3$ is spacelike, and an orthonormal basis $\{e_4,e_5\}$ for $V^\perp$. Then $B=\{e_1,\cdots,e_5\}$ is an orthonormal basis of $\R^{2,3}$. Suppose that $[p]\in\Ein^{1,2}\setminus \mathfrak{C}$ is an arbitrary point, and $p=(p_1,\cdots,p_5)$ is the representative of $p$ with respect to $B$. We have 
\[-p_1^2-p_2^2+p_3^2=-p_4^2-p_5^2.\]
Observe that $(p_4,p_5)\in V^\perp$ is a non-zero vector. Therefore, dividing $p$ by the positive number $\sqrt{p_4^2+p_5^2}$, we may assume
\[-p_1^2-p_2^2+p_3^2=-p_4^2-p_5^2=-1,\]
which describes the product $\AdS^{1,1}\times \Sn^1\subset \widehat{\Ein}^{1,2}$. 

 Obviously, the identity component of the stabilizer of $V$ is isomorphic to the direct product $SO_\circ(2,1)\times SO(2)$. This completes the proof.
\end{proof}
%
\begin{remark}
In the rest of the paper, we denote a point $[p]\in \Ein^{1,n}$ simply by $p$ when there is no ambiguity between the point $[p]$ and a representative $p\in[p]$ which is a lightlike vector $p\in \R^{2,n}$.
\end{remark}
\subsection{Cohomogeneity one actions}\label{sec.1.8}

\begin{definition}\label{def.1.8.1}
Let $G$ be a Lie group acting on a manifold $M$ smoothly. The action of $G$ is called of cohomogeneity one if it admits a codimension one orbit in $M$.
\end{definition}

\begin{definition}\label{def.1.3.3}
Let $(M,[\gi])$ be a semi-Riemmanian manifold and $G,H\subset\Conf(M)$ be Lie subgroups. Then the actions of $G$ and $H$ on $M$ are said to be orbitally-equivalent if there exists a conformal map $\varphi$ on $M$ such that for all $p\in M$, $\varphi(G(p))=H(\varphi(p))$.
\end{definition}
 The following lemma gives a powerful tool to distinguish the orbitally-equivalent actions (see \cite{Has0}).
\begin{lemma}\label{lem.1.3.4}
Let $G$ and $H$ be connected Lie subgroups of $\Conf(M)$ and $\varphi$ be a conformal map on $M$. Then the following statements are equivalent:
\begin{itemize}
\item[(i)] $G$ and $H$ are orbitally-equivalent via $\varphi$.
\item[(ii)] for all $p\in M$,
\begin{align*}
d\varphi_p(T_pG(p))=T_{\varphi(p)}H(\varphi(p)).
\end{align*}
\end{itemize}
\end{lemma}
%
\begin{definition}
Let $G$ be a Lie group and $H,K$ be Lie subgroups. Then, $H$ and $K$ are called transversal to each other if the dimension of $HK$, the Lie subgroup of  $G$ generated by the elements in $H$ and $K$, is strictly greater than $\max\{\dim H,\dim K\}$.
\end{definition}
\begin{lemma}\label{lemma.1.50}
Let $G\subset \Conf(\Ein^{1,2})$ preserves a linear subspace $V\leq \R^{2,3}$ of signature $(0,2)$ and acts on $\Ein^{1,2}$ with cohomogeneity one. Then $G$ preserves a $1$-dimensional linear subspace of $\R^{2,3}$.
\end{lemma}
\begin{proof}
By Lemma \ref{lem.1.7.23}, $G$ is a subgroup of $SO_\circ(2,1)\times SO(2)$ up to conjugacy. The following natural projections are group morphisms
\begin{align*}
P_1:SO_\circ(2,1)\times SO(2)\longrightarrow SO_\circ(2,1),\;\;\;\;\;P_2:SO_\circ(2,1)\times SO(2)\longrightarrow SO(2).
\end{align*}
The group $G$ is a subgroup of $P_1(G)\times P_2(G)$. We show that $P_1(G)\times P_2(G)$ preserves a line in $\R^{2,3}$. Observe that $SO_\circ(2,1)\times S(2)$ preserves a timelike circle and acts on its complement $\AdS^{1,1}\times \Sn^1$ in $\Ein^{1,2}$ transitively. Hence, $G$ is a proper subgroup of $SO_\circ(2,1)\times SO(2)$. In the one hand, if $P_2(G)=\{Id\}$, then $G$ acts on $V$ trivially. On the other hand, $P_1(G)\neq \{Id\}$, since $\dim G\geq 2$. Therefore, if $P_2(G)\neq \{Id\}$, up to conjugacy, $G$ is a subgroup of $Y_E\times SO(2)$ or $\Aff\times SO(2)$. 

 The $1$-parameter subgroup $Y_E$ preserves a unique $1$-dimensional spacelike linear subspace $\ell\leq V^\perp$. Since the action of $SO(2)$-factor on $V^\perp$ is trivial, $\ell$ is invariant by $Y_E\times SO(2)$.
 
 The affine group $\Aff$ preserves a unique lightlike line $\ell$ in $V^\perp$. Since the action of $SO(2)$-factor is trivial on $V^\perp$, $\ell$ is invariant by $\Aff\times SO(2)$. 
\end{proof}

\begin{proposition}\label{prop.all}
Let $G\subset \Conf(\Ein^{1,2})$ be a connected Lie subgroup which acts on $\Ein^{1,2}$ reducibly and with cohomogeneity one. Then, either $G$ preserves a line in $\R^{2,3}$ or it is a subgroup of $SO(2)\times SO(3)$ (up to conjugacy), or it preserves a photon.
\end{proposition}
\begin{proof}
Suppose that $G$ preserves a proper linear subspace $V\leq \R^{2,3}$. Denote by $sgn(V)$ the signature of the restriction of the metric from $\R^{2,3}$ on $V$. We consider all the possible signatures for $V$. 

If $\dim V=1$, then obviously, $G$ fixes a point in the projective space $\R\Pn^4$, namely $\Pn(V)$. Assume that $V$ is $2$-dimensional.
\begin{itemize}
\item[(I)] If $sgn(V)=(1,1)$. then, $V$ contains exactly two distinct lightlike lines. Hence, the intersection of $\Pn(V)$ with $\Ein^{1,2}$ consists of two points. Since $G$ is connected, it fixes both the points.
\item[(II)] If $sgn(V)=(2,0)$, $G$ preserves the orthogonal complement $V^\perp$ ($sgn(V^\perp)=(0,3)$) as well, and so, it is a subgroup of $SO(2)\times SO(3)$ up to conjugacy.
\item[(III)] If $sgn(V)=(0,1,1)$ or $(1,0,1)$, then, $V$ contains a unique lightlike line $\ell$. Since the action of $G$ on $\R^{2,3}$ is isometric, it preserves $\ell$. Hence, $G$ fixes $\Pn(\ell)\in \Ein^{1,2}$.
\item[(IV)] If $sgn(V)=(0,2)$, by Lemma \ref{lemma.1.50}, $G$ fixes a point in $\R\Pn^4$.
\item[(V)] If $sgn(V)=(0,0,2)$, then $G$ preserves a photon.
\end{itemize}
Now, suppose that $\dim V>2$. Since $G$ preserves $V^\perp$ and $\dim V^\perp \leq 2$, the proposition follows easily.
\end{proof}

\section{Proper actions}\label{proper}
In this section, we describe the cohomogeneity one proper actions on the $3$-dimensional Einstein universe $\Ein^{1,2}$.

The following theorem is the main result of this section.
\begin{theorem}\label{thm.2.0.1}
Let $G\subset \Conf_\circ(\Ein^{1,2})=SO_\circ(2,3)$ be a connected Lie group which acts on $\Ein^{1,2}$ properly and with cohomogeneity one. Then $G$ is conjugate to either $SO(3)$ or $SO(2)\times SO(2)$. Furthermore, the action of $SO(3)$ on $\Ein^{1,2}$ admits a codimension $1$ foliation on witch each leaf is a spacelike hypersphere. Moreover, by the action of $SO(2)\times SO(2)$ every $2$-dimensional orbit is conformally equivalent to $2$-dimensional Einstein universe $\Ein^{1,1}$ and one of them is an Einstein hypersphere.
\end{theorem}
 A Lie group $G$ acts on Einstein universe $\Ein^{1,2}$ properly if and only if $G$ is compact, since $\Ein^{1,2}$ itself is compact.
It is well-known that every maximal compact subgroup of $SO_\circ(2,3)$ is conjugate to $SO(2)\times SO(3)$ (see \cite[p.p. 275]{Hel}). 

\begin{lemma}\label{lem.2.0.5}
Let $G$ be a proper connected Lie subgroup of $SO(2)\times SO(3)$ with $\dim G\geq 2$. Then, either $G=\{Id\}\times SO(3)\simeq SO(3)$ or $G$ is conjugate to $SO(2)\times SO(2)$.
\end{lemma}
\begin{proof}
Let $P_1$ and $P_2$ denote the projection morphisms from $SO(2)\times SO(3)$ to $SO(2)$ and $SO(3)$, respectively. The group $G$ is a subgroup of $P_1(G)\times P_2(G)$.  Since $SO(3)$ has no two dimensional subgroup,  $\dim P_2(G)\in \{1,3\}$ (see \cite[Proposition 2.2]{Has0}). If $\dim P_2(G)=1$, then $P_2(G)=SO(2)$ up to conjugacy and $P_1(G)=SO(2)$. Therefore, $G=SO(2)\times SO(2)$ up to conjugacy.

Now, suppose that $\dim P_2(G)=3$. Since $G$ is a proper subgroup, $\dim G=3$. Hence, the differential map $dP_2$ at the identity element of $G$ is a Lie algebra isomorphism from $\g=Lie(G)$ to $Lie(SO(3))=\mathfrak{so}(3)$. So, $f=dP_1\circ (dP_2)^{-1}:\mathfrak{so}(3)\rightarrow dP_1(\g)$ is a surjective Lie algebra morphism. If $P_1(G)=SO(2)$, then $\ker f$ is a $2$-dimensional ideal of $\mathfrak{so}(3)$. This contradicts the simplicity of $\mathfrak{so}(3)$. Hence $P_1(G)=\{Id\}$. Therefore, $G=\{Id\}\times SO(3)\simeq SO(3)$.
\end{proof}
\textbf{Proof of Theorem \ref{thm.2.0.1}.} By the cohomogeneity one assumption $\dim G\geqslant 2$, so by Lemma \ref{lem.2.0.5}, $G$ is conjugate to either $SO(3)$ or $SO(2)\times SO(2)$. There exist a unique $SO(2)\times SO(3)$-invariant decomposition for $\R^{2,3}=V\oplus V^\perp$ where $V$ is of signature $(2,0)$ (and hence $V^\perp$ is of signature $(0,3)$). 

Observe that, by the action of $SO(3)=I\times SO(3)$ on the double cover space $\widehat{\Ein}^{1,2}$, the induced orbit at each point $(x,y)\in \Sn^1\times \Sn^2$ is $\{x\}\times \Sn^2$. This orbit is the projectivization of $x^\perp\cap \Nu^{2,3}\subset \R^{2,3}$ on $\widehat{\Ein}^{1,2}$. It is clear that, the action of $\mathbb{Z}_2=\{\pm (Id_{\Sn^1},Id_{\Sn^2})\}$ on $\Sn^1\times \Sn^2$ maps the orbits $\{x\}\times \Sn^2$ and $\{-x\}\times \Sn^2$ to $\{[x]\}\times\Sn^2$, for all $x\in \Sn^1$. Hence, each orbit induced by $SO(3)$ in $\Ein^{1,2}$ is a spacelike hypersphere.

The action of $SO(2)$ on $\Sn^2$ admits two antipodal fixed points $\{x_0,-x_0\}\in \Sn^2$ and acts on $\Sn^2\setminus \{x_0,-x_0\}$ freely. Hence, $SO(2)\times SO(2)$ preserves a timelike circle
$$\mathfrak{C}=(\Sn^1\times \{x_0,-x_0\})/\mathbb{Z}_2=\Pn(V\oplus \R x_0)\cap \Ein^{1,2},$$
and acts on it transitively. It is clear that, for arbitrary $y\in \Sn^1$ and $x\in \Sn^2\setminus\{x_0,-x_0\}$, the orbit induced by $SO(2)\times SO(2)$ at $(y,x)\in \Sn^1\times\Sn^2$ is conformally equivalent to $\widehat{\Ein}^{1,1}$. Hence, the orbit induced at $[x:y]\in \Ein^{1,2}$ is conformally equivalent to $\widehat{\Ein}^{1,1}/\mathbb{Z}_2=\Ein^{1,1}$. Since $SO(2)\times SO(2)$ preserves the spacelike line $\R x_0$, it preserves the orthogonal complement subspace $x_0^\perp \leq \R^{2,3}$. Hence the orbit induced at $p\in \Pn( x_0^\perp )\cap \Ein^{1,2}$ is an Einstein hypersphere. This completes the proof.\hfill$\square$

\begin{proposition}\label{cor.proper}
Let $G\subset SO_\circ(2,3)$ be a connected Lie subgroup which acts on $\Ein^{1,2}$ properly and with cohomogeneity one. Then $G$ admits a fixed point in the projective space $\R\Pn^4$.
\end{proposition}
\begin{proof}
It follows immediately from the proof of Theorem \ref{thm.2.0.1}.
\end{proof}

\section{Actions on Minkowski patch and lightcone}\label{fix}
In this section, we study the actions which admit a fixed point in Einstein universe $\Ein^{1,2}$. Observe that, for an arbitrary point $p\in \Ein^{1,2}$, the stabilizer $Stab_{SO_\circ(2,3)}(p)$ preserves the lightcone $L(p)$, since it preserves the linear subspace $p^\perp\leq \R^{2,3}$. Also, it preserves the Minkowski patch $Mink(p)$. Hence, a cohomogeneity one action of a subgroup of $Stab_{SO_\circ(2,3)}(p)$ admits a $2$-dimensional orbit in $L(p)$ or $Mink(p)\approx \En^{1,2}$.

Recall from Section \ref{sec.Ein} that, by the action of $\Conf_\circ(\Ein^{1,2})=SO_\circ(2,3)$ on Einstein universe $\En^{1,2}$, the identity component of the stabilizer of a point $p\in \Ein^{1,2}$ is isomorphic to the semidirect product
\begin{align}\label{fix.eq.1}
 \Conf_\circ(\En^{1,2})=(\R_+^*\times SO_\circ(1,2))\ltimes \R^{1,2}.
\end{align} 
By the action of $\Conf(\En^{1,2})$ on Minkowski space $\En^{1,2}$, the identity component of the stabilizer of a point is conjugate to $\R_+^*\times SO_\circ(1,2)$. In fact, $\R_+^*\times SO_\circ(1,2)$ fixes a unique point $o\in \En^{1,2}$ and the splitting Eq. \ref{fix.eq.1} depends strongly on $o$.

\begin{theorem}\label{thm.fix}
Every connected Lie subgroup of $\Conf(\En^{1,2})$ with $\dim \geq 2$ acts on Einstein universe $\Ein^{1,2}=Mink(p)\cup L(p)$ with cohomogeneity one, except $\R^{1,2}$ and $\R_+^*\ltimes \R^{1,2}$.
\end{theorem}
\begin{proof}
It follows immediately from Theorem \ref{thm.3.0.2}.
\end{proof}
\begin{remark}
Choosing a point $o$ as the origin of Minkowski space $\En^{1,2}$, it becomes naturally a vector space. Moreover, the Lorentzian quadratic form $\q_{1,2}$ on $\R^{1,2}$ induces a Lorentzian quadratic form $\q$ on $(\En^{1,2},o)$, making it a Lorentzian scalar product space.
\end{remark}
There are three natural group morphisms
\begin{align*}
&P_{l}:(\R_+^*\times SO_\circ(1,2))\ltimes \R^{1,2}\longrightarrow \R_+^*\times SO_\circ(1,2), &(\lambda,A,v)\mapsto (\lambda,A).\\
&P_{li}:(\R_+^*\times SO_\circ(1,2))\ltimes \R^{1,2}\longrightarrow SO_\circ(1,2),&(\lambda,A,v)\mapsto A,\\
&P_{h}:(\R_+^*\times SO_\circ(1,2))\ltimes \R^{1,2}\longrightarrow \R_+^*,&(\lambda,A,v)\mapsto \lambda.
\end{align*}

\begin{definition}\label{def.3.0.1}
Given a Lie subgroup $G\subset \Conf_\circ(\En^{1,2})$, the identity component of the kernel $\ker P_l|_G$ is called the translation part of $G$ and is denoted by $T(G)$. Also, the images of $G$ under $P_l$, $P_{li}$ and $P_{h}$ are called the linear projection, the linear isometry projection and the homothety projection of $G$, respectively.
\end{definition}

\begin{theorem}\label{thm.3.0.2}
Let $G$ be a Lie subgroup of $\Conf_\circ(\En^{1,2})$. Then $G$ acts on $\Ein^{1,2}=Mink(p)\cup L(p)$ with cohomogeneity one if and only if $\dim G\geq 2$ and it satisfies one of the following conditions.
\begin{itemize}
\item[-] The linear isometry projection $P_{li}(G)$ is non-trivial.
\item[-] The linear isometry projection $P_{li}(G)$ is trivial and the translation part $T(G)$ has dimension less that or equal to $2$.
\end{itemize}
\end{theorem}

\subsection{Actions on lightcone}\label{sec.3.1}

According to Proposition \ref{prop.1.7.10}, for an arbitrary point $q\in L(\hat{p})=L(p)\setminus \{p\}$, there exists a unique affine degenerate plane $\widetilde{\Pi}$ in $\En^{1,2}\approx Mink(p)$ such that $q$ is the limit point of $\widetilde{\Pi}$. This induces a one-to-one correspondence between the set of photons in $L(p)$ and the set of degenerate linear $2$-planes in $\R^{1,2}$: Choosing a photon $\phi\subset L(p)$ there exists a unique degenerate plane $\Pi\leq \R^{1,2}$ such that the corresponding affine degenerate planes with limit points in $\phi$ are parallel to $\Pi$. On the other hand, choosing a degenerate plane $\Pi$ in $\R^{1,2}$ there exists a unique photon $\phi$ on which the limit points of the leaves of the foliation induced by $\Pi$ in $\En^{1,2}$ lie in $\phi$. From now on, for a photon $\phi$ in the lightcone $L(p)$, we denote the corresponding degenerate plane in $\R^{1,2}$ by $\Pi_\phi$.

The subgroup $\R_+^*\times SO_\circ(1,2)$ fixes the point $o$ in the Minkowski patch $Mink(p)\approx\En^{1,2}$. Thus, every element of $\R_+^*\times SO_\circ(1,2)$ maps each affine degenerate plane through $o$ to an affine degenerate plane through $o$. Hence, this group preserves the ideal circle $S_\infty\simeq \R\Pn^1$ which is the union of the intersections of the lightcones $L(p)$ and $L(o)$.

A homothety $\lambda\in \R_+^*$ (centered at $o$) on the Minkowski space $\En^{1,2}$ preserves every degenerate affine plane through $o$. It follows that, the homothety factor $\R_+^*$ acts on the ideal circle $S_\infty$ trivially. Let $\phi\subset L(p)$ be a photon and $\Pi_{\phi,q}\subset \En^{1,2}$ be the corresponding affine degenerate plane with limit point $q=S_\infty\cap \phi$. For an arbitrary point $u\in \hat{\phi}\setminus\{q\}$ with corresponding affine degenerate plane $\Pi_{\phi,u}$, the homothety $\lambda$ maps $\Pi_{\phi,u}$ to a parallel  plane $\Pi_{\phi,q}$. Hence, every homothety preserves each photon in the lightcone $L(p)$. Also, the homothety factor $\R_+^*$ act  on the set of parallel planes to $\Pi_{\phi,q}$ transitively. Therefore, $\R_+^*$ acts on the both connected components of $\hat{\phi}\setminus q$, transitively.

Let $g$ be a non-trivial element in $SO_\circ(1,2)$. 
\begin{itemize}
\item If $g$ is an elliptic element, then it preserves no degenerate plane in $\R^{1,2}$, and so, $g$ preserves no photon in $L(p)$. This implies that every elliptic $1$-parameter subgroup of $SO_\circ(1,2)$ acts on $L(\hat{p})$ freely.
\item If $g$ is a parabolic element, then it preserves a unique degenerate plane in $\R^{1,2}$. Consequently, $g$ preserves a unique photon $\phi$ in $L(p)$. Therefore, every parabolic $1$-parameter subgroup of $SO_\circ(1,2)$ acts on $L(p)\setminus \phi$ freely.
\item If $g$ is a hyperbolic element, then it preserves exactly two degenerate planes in $\R^{1,2}$. Hence, $g$ preserves two photons $\phi$ and $\psi$ in the lightcone $L(p)$ and admits exactly two fixed points in the ideal circle $S_\infty$. Henceforth, every hyperbolic $1$-parameter subgroup of $SO_\circ(1,2)$ acts on $L(p)\setminus (\phi\cup \psi)$ freely.
\end{itemize}

Now, let $v\in \R^{1,2}$ be a translation on Minkowski space $\En^{1,2}$. Assume that $\Pi\leq \R^{1,2}$ is a degenerate linear plane and $\widetilde{\Pi}$ is an affine degenerate plane in $\En^{1,2}$ parallel to $\Pi$. Then, $v$ maps $\widetilde{\Pi}$ to the parallel plane $v+\widetilde{\Pi}$. Hence, every translation preserves each photon in $L(p)$.

\begin{itemize}
\item If $v$ is a timelike vector, then it preserves no affine degenerate plane in $\En^{1,2}$. Hence, a timelike vector admits no fixed point in the vertex-less lightcone $L(\hat{p})$.
\item If $v$ is a lightlike vector, then it preserves all the affine degenerate planes in $\En^{1,2}$ parallel to $v^\perp\leq \R^{1,2}$. This implies that the set of points in $L(p)$ fixed by a lightlike element is a unique photon.
\item If $v$ is a spacelike vector, then it preserves the two degenerate planes $\Pi,\Pi'\leq \R^{1,2}$ directed by the two distinct lightlike directions in the Lorentzian $2$-plane $v^\perp\leq \R^{1,2}$. Obviously, $v$ preserves all the affine degenerate planes in $\En^{1,2}$ which are parallel to $\Pi$ or $\Pi'$. Hence the set of points in $L(p)$ fixed by a spacelike element is the union of two distinct photons.
\end{itemize}

Let $G$ be a connected Lie subgroup of $\R_+^*\times SO_\circ(1,2))\ltimes \R^{1,2}$. The translation part $T(G)$ is a normal subgroup of $G$, hence $G$ acts on $T(G)$ by conjugation. Therefore, the natural action of the linear isometry projection $P_{li}(G)$ on $\R^{1,2}$ preserves the translation part $T(G)\leq \R^{1,2}$. Furthermore, assume that $T(G)$ fixes a photon $\phi$ in the lightcone $L(p)$ pointwisely. Thus, $T(G)$ preserves the leaves of the foliation induced by ${\Pi_\phi}$ in $\En^{1,2}$. Henceforth, $T(G)$ is a linear subspace of $\Pi_\phi\leq \R^{1,2}$. This implies that, $T(G)$ is either a degenerate subspace or it is a spacelike line. In the first, obviously $P_{li}(G)$ preserves $\Pi_\phi$. In the later, it is easy to see that $P_{li}(G)$ is a hyperbolic $1$-parameter subgroup. Hence, $P_{li}(G)$ preserves the two degenerate planes generated by $T(G)$ and each of the null directions in the Lorentzian $2$-plane $T(G)^\perp$. We conclude that, if $T(G)$ fixes a photon $\phi\subset L(p)$ pointwisely, then $\phi$ is invariant by $P_{li}(G)$ and consequently, by $G$.

\begin{proposition}\label{pro.3.1.1}
Let $G$ be a connected Lie subgroup of $\R_+^*\ltimes \R^{1,2}$. Then $G$ is conjugate to the semidirect product $P_h(G)\ltimes T(G)$.
\end{proposition}
\begin{proof}
First assume that $P_h(G)$ is trivial. Then, obviously $G=T(G)$. Now, suppose that $P_h(G)\neq \{1\}$.
Let $L$ be a $1$-parameter subgroup of $G$ transversal to $T(G)$. Considering the Lie algebra of $L$, one can see, there exists a unique vector $v\in \R^{1,2}$, such that $L=\left\lbrace\big(e^t,(e^t-1)v\big):t\in \R\right\rbrace\subset \R_+^*\ltimes \R^{1,2}$. Observe that $L$ is conjugate to $\R_+^*$ via $(1,v)$ and $T(G)$ is invariant by this conjugation. Therefore, $G$ is conjugate to $\R_+^*\ltimes T(G)$.
\end{proof}
\textbf{Proof of Theorem \ref{thm.3.0.2}:} Assume that $G$ admits a $2$-dimensional orbit at $q\in \Ein^{1,2}$. Obviously $\dim G\geq 2$. If $q\in L(p)$, then $P_{li}(G)$ is nontrivial, since the subgroup $\R^*_+\ltimes \R^{1,2}$ admits no open orbit in $L(p)$. If $q$ belongs to the Minkowski patch $Mink(p)\approx \En^{1,2}$, then $\dim T(G)\leq 2$, since the subgroup $\R^{1,2}$ acts on $\En^{1,2}$ transitively.

Now, we prove the reverse direction. Suppose that $\dim G\geq 2$. First, assume that the linear isometry projection $P_{li}(G)$ is non-trivial. There are some cases:

\textbf{Case I:} \textit{The translation part is non-trivial.} In this case, the photons which are fixed pointwisely by $T(G)$ (if there exists any) are exactly those which $P_{li}(G)$ preserves them. Assume that $H\subset G$ is a $1$-parameter subgroup with non-trivial linear isometry projection $P_{li}(H)$. There exist a photon $\phi\subset L(p)$ which is not $H$-invariant. For an arbitrary point $q\in \hat{\phi}$, the vector tangent to $G(q)$ at $q$ induced by $H$ is spacelike. On the other hand, $T(G)$ acts on the vertex-less photon $\hat{\phi}$ transitively. So, there is a lightlike vector tangent to $G(q)$ at $q$. This implies that, the tangent space $T_qG(q)$ is $2$-dimensional.

\textbf{ Case II:} \textit{The translation part is trivial:} There are two subcases.
\begin{itemize}
\item[$\blacktriangleright$] \textit{$G$ contains the homothety factor $\R_+^*$ as a subgroup.} The subgroup $\R_+^*$ preserves the ideal circle $S_\infty$ and acts on $L(p)\setminus (S_\infty\cup \{p\})$ freely. Let $L$ be a $1$-parameter subgroup of $G$ transversal to $\R_+^*$. Obviously, $P_{li}(L)$ is non-trivial. There exists a photon $\phi\subset L(p)$ which is not invariant by $L$. For an arbitrary point $q\in \phi\setminus (S_\infty\cup \{p\})$, the vectors tangent to the orbit $G(q)$ at $q$ induced by $\R_+^*$ and $L$ are lightlike and spacelike, respectively. Hence, $G$ admits an open orbit in $L(p)$.

\item[$\blacktriangleright$] \textit{The homothety factor $\R_+^*$ is not a subgroup of $G$.} In this case, we have $\dim G=\dim P_{li}(G)\in \{2,3\}$.
\begin{itemize}
\item If $\dim G=3$, then $P_{li}(G)=SO_\circ(1,2)$ (in fact $G$ is isomorphic to $SO_\circ(1,2)$). The group $G$ is a Levi factor of $(\R_+^*\times SO_\circ(1,2))\ltimes \R^{1,2}$. By the uniqueness of Levi factor we have $G=SO_\circ(1,2)$, up to conjugacy (see \cite[p.p. 93]{Jac}). Obviously, $SO_\circ(1,2)$ admits a $2$-dimensional orbit in $\En^{1,2}$.
\item If $\dim G=2$, then $P_{li}(G)=\Aff$, up to conjugacy. Hence, $G$ preserves a unique photon $\phi\subset L(p)$. Let $\mathcal{P}$ be a $1$-parameter subgroup of $G$ which its linear isometry projection $P_{li}(\mathcal{P})$ is parabolic. Furthermore, assume that $\mathcal{H}$ is a $1$-parameter subgroup of $G$ which is transversal to $\mathcal{P}$. Obviously, $\mathcal{H}$ has hyperbolic linear isometry projection and it also preserves a photon $\psi\subset L(p)$ different from $\phi$. Denote by $\Pi_\psi$ the degenerate plane in $\R^{1,2}$ correspond to $\psi$. Observe that $\mathcal{P}$ acts on $L(p)\setminus \phi$ freely, since it preserves only $\phi$. Assume that $\mathcal{H}$ induces an open arc (orbit) $I\subset \psi$. Then $\mathcal{P}$ maps $I$ to other photons, and therefore, $G$ admits an open orbit in $L(p)$. If $\mathcal{H}$ fixes $\psi$ pointwisely, then it preserves the leaves of the foliation $\mathcal{F}_{\Pi_\psi}$ induced by $\Pi_\psi$ in $\En^{1,2}$. 
Observe that $\mathcal{H}$ admits a $1$-dimensional orbit at some point $q\in \En^{1,2}$ included in the leaf $\mathcal{F}_{\Pi_\psi}(q)$, since the action of $\Conf(\En^{1,2})$ is faithful. The vector tangent to $G(q)$ at $q$ induced by $\mathcal{P}$ does not lie in $\Pi_\psi$, since $\mathcal{P}$ does not preserve $\Pi_\psi$. Hence, $G$ admits a $2$-dimensional orbit at $q$.  
\end{itemize}
\end{itemize}

Finally, assume that the linear isometry projection $P_{li}(G)$ is trivial, then $G$ is a subgroup of $\R_+^*\ltimes \R^{1,2}$. By Proposition \ref{pro.3.1.1}, $G$ is conjugate to $P_h(G)\ltimes T(G)$. Observe that $T(G)$ admits a $\dim T(G)$-dimensional foliation $\mathcal{F}_{T(G)}$ in $\En^{1,2}$. If the homothety projection $P_h(G)$ is trivial, then $G$ is a linear $2$-plane in $\R^{1,2}$ and so, all the orbits in $\En^{1,2}$ are $2$-dimensional. If $P_h(G)=\R_+^*$, then $G$ preserves a unique leaf of $\mathcal{F}_{T(G)}$, namely the leaf containing $o$. Now, the result follows easily.\hfill $\square$

\subsection{Lie subgroups of $\Conf(\En^{1,2})$}
According to Theorem \ref{thm.fix}, every connected Lie subgroup of $\Conf(\En^{1,2})$ with $\dim \geq 2$ acts on Einstein universe with cohomogeneity one, except $\R^{1,2}$ and $\R_+^*\ltimes \R^{1,2}$. Here we give a complete list of connected Lie subgroups of $\Conf(\En^{1,2})$ with $\dim \geq 2$ up to conjugacy. 

Let $\{e_1,e_2,e_3\}$ be an orthonormal basis for $\R^{1,2}$, where $e_1$ is timelike vector, and $(x,y,z)$ be the corresponding coordinate on the Minkowski space $Mink(p)\approx\En^{1,2}$ with $o=(0,0,0)\in \En^{1,2}$ as the origin (here $o$ is the unique point fixed by $\R_+^*\times SO_\circ(1,2)$). The Lie algebra  of $\Conf(\En^{1,2})$ is isomorphic to the semi-direct sum $(\R\oplus\mathfrak{so}(1,2))\oplus_\theta \R^{1,2}$, where $\theta$ is the natural representation of $\R \oplus \mathfrak{so}(1,2)$ into $\mathfrak{gl}(\R^{1,2})$. Recall form Section \ref{proj}, the simple group $SO_\circ(1,2)\simeq \PSL(2,\R)$ has exactly four distinct non-trivial proper subgroup. The set $\{\Y_E,\Y_P,\Y_H\}$ is a basis for the Lie algebra $\mathfrak{so}(1,2)$ as a vector space. For arbitrary elements $\lambda\in \R$, $X\in \mathfrak{so}(1,2)$, and $v\in \R^{1,2}$, we denote the corresponding element in $(\R\oplus\mathfrak{so}(1,2))\oplus_\theta \R^{1,2}$ simply by $\lambda+X+v$, when there is no ambiguity. Furthermore, we denote by $\R(\lambda+X+v)$ the linear subspace of $(\R\oplus\mathfrak{so}(1,2))\oplus_\theta \R^{1,2}$ generated by the vector $\lambda+X+v$. Also, for a Lie subalgebra $\g\leq (\R\oplus\mathfrak{so}(1,2))\oplus_\theta \R^{1,2}$, we denote by $\exp(\g)$ the corresponding connected Lie subgroup of $(\R_+^*\times SO_\circ(1,2))\ltimes\R^{1,2}$.
The following theorem, characterizes all the connected Lie subgroups of $\Conf(\En^{1,2})$ with $\dim \geq 2$, up to conjugacy.

\begin{theorem}\label{thm.3.2.1}
Let $G\subset \Conf_\circ(\Ein^{1,2})$ be a connected Lie subgroup with $\dim G\geq 2$. Then $G$ is conjugate to one of the subgroups indicated in Tables \ref{table1}-\ref{table8}.
\end{theorem}

\begin{table}[h!]
\centering
\begin{tabular}{|c| c| c|c| } 
 \hline
  \multicolumn{4}{|c|}{\textbf{Subgroups with full translation part}} \\
  \hline
 $\R^*_+\ltimes \R^{1,2}$ & $(\R_+^*\times SO_\circ(1,2))\ltimes \R^{1,2}$ & $SO_\circ(1,2)\ltimes \R^{1,2}$ &  $\exp\big(\R(a+\Y_E))\ltimes \R^{1,2}$\\
  \hline
 $\R^{1,2}$ & $(\R_+^*\times \Aff)\ltimes \R^{1,2}$ & $\Aff\ltimes \R^{1,2}$ & $\exp\big(\R(a+\Y_P))\ltimes \R^{1,2}$\\
  \hline
$Y_H\ltimes \R^{1,2}$  &  $(\R_+\times Y_H)\ltimes \R^{1,2}$ &  $Y_P\ltimes \R^{1,2}$  & $\exp\big(\R(a+\Y_H)+\R\Y_P\big)\ltimes \R^{1,2}$\\
  \hline
$Y_E\ltimes \R^{1,2}$ &  $(\R_+\times Y_P)\ltimes \R^{1,2}$ & $(\R_+\times Y_E)\ltimes \R^{1,2}$ & $\exp\big(\R(a+\Y_H))\ltimes \R^{1,2}$ \\
\hline
\end{tabular}
\caption{Here $a\in \R^*$ is a fixed number.}
\label{table1}
\end{table}

\begin{table}[h!]
\centering
\begin{tabular}{|c| c| } 
 \hline
 \multicolumn{2}{|c|}{\textbf{Subgroups with a Lorentzian plane as the translation part}} \\
  \hline
  $(\R_+^*\times Y_H)\ltimes (\R e_1\oplus \R e_2)$ & $Y_H\ltimes (\R e_1\oplus \R e_2)$\\
\hline  
 $\exp\big(\R(a+\Y_H)\big)\ltimes (\R e_1\oplus \R e_2)\big)$ &  $\exp\big(\R(\Y_H+e_3)\big)\ltimes (\R e_1\oplus \R e_2)$\\
  \hline
    $\R_+^*\ltimes(\R e_1\oplus\R e_2) $ & $\R e_1\oplus \R e_2$\\
  \hline
\end{tabular}
\caption{Here $a\in \R^*$ is a fixed number.}
\label{table2}
\end{table}

\begin{table}[h!]
\centering
\begin{tabular}{|c| c| } 
\hline
 \multicolumn{2}{|c|}{\textbf{Subgroups with a spacelike plane as the translation part}} \\
  \hline
    $(\R_+^*\times Y_E)\ltimes (\R e_2\oplus \R e_3)$ & $Y_E\ltimes (\R e_2\oplus \R e_3)$\\
\hline  
 $\exp\big(\R(a+\Y_E)\big)\ltimes (\R e_2\oplus \R e_3)\big)$ &  $\exp\big(\R(\Y_E+e_1)\big)\ltimes(\R e_2\oplus \R e_3)$\\
  \hline
  $\R_+^*\ltimes(\R e_2\oplus\R e_3) $ & $\R e_2\oplus \R e_3$\\
  \hline
\end{tabular}
\caption{Here $a\in \R^*$ is a fixed number.}
\label{table3}
\end{table}

\begin{table}[h!]
\centering
\begin{tabular}{|c| c| } 
   \hline
  \multicolumn{2}{|c|}{\textbf{Subgroups with a degenerate plane as the translation part}} \\
  \hline
$\Aff\ltimes \Pi_\phi$ & $\exp\big(\R(a+\Y_H)+\R\Y_P\big)\ltimes \Pi_\phi$\\
\hline
$(\R_+^*\times\Aff)\ltimes \Pi_\phi$ & $\exp\big(\R(a+\Y_P)\big)\ltimes \Pi_\phi$ \\
\hline
$Y_P\ltimes \Pi_\phi$ & $\exp\big(\R(a+\Y_H)\big)\ltimes \Pi_\phi$ \\
\hline 
$(\R_+^*\times Y_P)\ltimes \Pi_\phi$ & $\exp\big(\R(1+\Y_H +e_1)+\R\Y_P\big)\ltimes\Pi_\phi$\\
\hline 
$Y_H\ltimes \Pi_\phi$ & $\exp\big(\R(\Y_P+e_1)\big)\ltimes\Pi_\phi$\\
\hline 
$(\R_+^*\times Y_H)\ltimes \Pi_\phi$ & $\exp\big(\R(1+\Y_H +e_1)\big)\ltimes \Pi_\phi$\\
\hline 
$\Pi_\phi$ &  $\exp\big(\R(2+\Y_H)+\R(\Y_P+e_1)\big)\ltimes \Pi_\phi$ \\
 \hline
 $\R_+^*\ltimes \Pi_\phi$ & \\
 \hline
\end{tabular}
\caption{Here $\Pi_\phi$ denotes the degenerate plane $\R(e_1+e_2)\oplus \R e_3\leq \R^{1,2}$, and $a\in \R^*$ is a fixed number.}
\label{table4}
\end{table}

\begin{table}[h!]
\centering
\begin{tabular}{|c| c| c|c| } 
\hline
 \multicolumn{4}{|c|}{\textbf{Subgroups with a timelike line as the translation part}} \\
  \hline
$\R_+^*\ltimes \R e_1$ & $(\R_+^*\times Y_E)\ltimes \R e_1$ & $Y_E\times \R e_1$ & $\exp\big(\R(a+\Y_E)\big)\ltimes \R e_1$ \\
  \hline
\end{tabular}
\caption{ Here $a\in \R^*$ is a fixed number.}
\label{table5}
\end{table}

\begin{table}[h!]
\centering
\begin{tabular}{|c| c| c| } 
\hline
  \multicolumn{3}{|c|}{\textbf{Subgroups with a spacelike line as the translation part}} \\
  \hline
  $\R_+^*\ltimes \R e_3$ & $(\R_+^*\times Y_H)\ltimes \R e_3$ & $Y_H\times \R e_3$ \\
  \hline
   $\exp\big(\R(a+\Y_H))\ltimes \R e_3$ & $\exp\big(\R(1+\Y_H+e_1) \big)\ltimes \R e_3$ & $\exp\big(\R(-1+\Y_H+e_1) \big)\ltimes \R e_3$ \\
   \hline
\end{tabular}
\caption{Here $a\in \R^*$ is a fixed number.}
\label{table6}
\end{table}

\begin{table}[h!]
\centering
\begin{tabular}{|c| c| } 
 \hline
 \multicolumn{2}{|c|}{\textbf{Subgroups with a lightlike line as the translation part}} \\
  \hline
$(\R_+^*\times \Aff)\ltimes \el$ & $\exp\big(\R(a+\Y_H)+\R\Y_P\big)\ltimes \el$  \\
\hline
$Y_H\ltimes \el$ &$\exp\big(\R(a+\Y_P)\big)\ltimes \el$ \\
\hline
$\Aff\ltimes \el$  & $\exp\big(\R(\Y_H+e_3)+\R\Y_P\big)\ltimes  \el$ \\
  \hline
   $(\R_+^*\times Y_P)\ltimes \el$ &  $\exp\big(\R(\Y_P+e_1)\big)\ltimes \el$\\
  \hline
$Y_P\times \el$  &   $\exp\big(\R(2+\Y_H)+\R(\Y_P+e_1)\big)\ltimes \el$ \\
  \hline
$(\R_+^*\times Y_H)\ltimes \el$ & $\exp\big(\R(1+\Y_H+e_1)\big)\ltimes \el$ \\
\hline
 $\exp\big(\R(a+\Y_H)\big)\ltimes \el$ & $\exp\big(\R(\Y_H+e_3)\big)\ltimes\el$ \\
   \hline

\end{tabular}
\caption{ Here $\el$ denotes the lightlike line $\R(e_1+e_2)\leq \R^{1,2}$ and $a\in \R^*$ is a fixed number.}
\label{table7}
\end{table}

\begin{table}[h!]
\centering
\begin{tabular}{|c| c|  } 
  \hline
 \multicolumn{2}{|c|}{\textbf{Subgroups with trivial translation part}} \\
  \hline
 $SO_\circ(1,2)$  & $\R_+^*\times SO_\circ(1,2)$  \\
  \hline 
  $\Aff$  & $\R_+^*\times \Aff$\\
  \hline
  $\R_+^*\times Y_E$ & $\exp\big(\R(a+\Y_H)+\R\Y_P\big)$ \\
  \hline
  $\R_+^*\times Y_P$ &  $\exp\big(\R(2+\Y_H)+\R(\Y_P+e_1-e_2)\big)$\\
  \hline 
  $\R_+^*\times Y_H$ &  $\exp\big(\R(-1+\Y_H+e_1+e_2)+\R\Y_P\big)$\\
  \hline
\end{tabular}
\caption{Here $a\in [-1,1]$ is a fixed number.}
\label{table8}
\end{table}

\textbf{Sketch of proof:}
The proof is based on considering different possible cases for the dimension and causal character of the translation part of $G$.
Here, we explain the idea of proof by an example. The other subgroups can be obtained in a similar way. A complete version of proof is given in \cite{Has0}.

 Let $\dim T(G)=2$. Let the linear isometry projection $P_{li}(G)$ be a $1$-parameter parabolic subgroup of $SO_\circ(1,2)$ and $\dim G=3$. Set $H=(\R_+^*\times SO_\circ(1,2))\ltimes \R^{1,2}$ and $\h=(\R\oplus \so(1,2))\oplus_\theta\R^{1,2}$. We determine $G$ up to conjugacy by considering the adjoint action of $H$ on $\h$. The linear isometry projection $P_{li}(G)$ is conjugate to $Y_P\subset SO_\circ(1,2)$. The degenerate $2$-plane $\Pi_\phi=\R(e_1+e_2)\oplus \R e_3$ is the unique $2$-dimensional $Y_P$-invariant subspace of $\R^{1,2}$. Hence, $T(G)=\Pi_\phi$. The subgroups with a degenerate plane as the translation part have been indicated in Table \ref{table4}.

The Lie bracket rule on $\h$ is
\[[a+V+v,b+W+w]=[V,W]+V(w)+aw-W(v)-bv.\]
 The adjoint action of  $H$ on its Lie algebra $\h$ is as follows: for an arbitrary element $(r,A,v)\in H$, we have
\[Ad_{(r,A,v)}:\h\rightarrow \h,\;\;\;\; a+W+w\mapsto a+AWA^{-1}+rA(w)-av-AWA^{-1}(v).\]

Denote by $\g$ the Lie algebra of $G$. There is a constant $a\in \R$ and a vector $w\in\R^{1,2}$ such that $\{a+\Y_P+w,e_1+e_2,e_3\}$  is a basis for $\g$ as a vector space. The aim is to find a vector $x\in \R^{1,2}$ such that 
\[ Ad_{(1,Id,x)}(a+\Y_P+w)=a+\Y_P.\]
However, this is not always possible, but we try to simplify the vector $w=(w_1,w_2,w_3)$ as much as we can. Note that the translation part $T(\g)$ is always invariant by the adjoint action of $\R^*\ltimes\R^{1,2}$. 

%
\begin{itemize}
\item If $a\neq 0$, setting
\[x=\left(\frac{(a^2+1)w_1-aw_3-w_2}{a^3},\frac{(a^2-1)w_2-aw_3+w_1}{a^3},\frac{aw_3-w_1+w_2}{a^2}\right),\]
we have
\[Ad_{(1,Id,x)}(a+\Y_P+v)=a+\Y_P,\;\;\;Ad_{(1,Id,x)}(e_1+e_2)= e_1+e_2, \;\;\;\;Ad(1,Id,x)(e_3)=e_3.\]
Thus $a+\Y_P\in \g'=Ad_{(1,Id,x)}(\h)$, and so, $\{a+\Y_P,e_1+e_2,e_3\}$ is a basis for $\g'$. Therefore $\g$ is conjugate to the semi-direct sum 
\[\R(a+\Y_P)\oplus_\theta (\R (e_1+e_2)\oplus \R e_3).\]

\item If $a=0$, setting  $x=(0,-w_3,v_2)$ we have
\[Ad_{(1,Id,x)}(\Y_P+w)=\Y_P+(w_1-w_2,0,0),\;\;\;\;\;Ad_{(1,Id,x)}(e_1+e_2)=e_1+e_2,\;\;\;\;\;Ad_{(1,Id,x)}(e_3)=e_3.\]
Therefore, if $w_1=w_2$, the Lie algebra $\g$ is conjugate to 
$$\R\Y_P\oplus_\theta(\R(e_1+ e_2)\oplus \R e_3).$$
 Otherwise, $\g$ is conjugate to the following Lie algebra via $Ad_{(1/(w_1-w_2),Id,x)}$
$$\R(\Y_P+e_1)\oplus_\theta (\R(e_1+e_2)\oplus \R e_3).$$ 
\end{itemize}
\hfill$\square$

In Section \ref{orbit}, we describe the codimension one orbits induced by the connected Lie subgroups of $\Conf(\En^{1,2})$ obtained in Theorem \ref{thm.3.2.1}.
\section{Actions on Anti de-Sitter and de-Sitter components and their boundaries}\label{de}

In this section, we study the cohomogeneity one actions on Einstein universe $\Ein^{1,2}$ preserving a  timelike or a spacelike direction in $\R^{2,3}$. The first corresponds to the actions preserving a de-Sitter component $\dS^{1,2}$ and its conformal boundary $\partial \dS^{1,2}=\Sn^2$. The second corresponds to the actions preserving an Anti de-Sitter component $\AdS^{1,2}$ and its conformal boundary $\partial \AdS^{1,2}=\Ein^{1,1}$.
\subsection{Actions on de-Sitter component and its boundary}\label{sec.4.2}
Here, we consider cohomogeneity one actions of connected subgroups of $SO_\circ(2,3)$ on Einstein universe $\Ein^{1,2}$ preserving a $1$-dimensional timelike linear subspace of $\R^{2,3}$.

Let $G$ be a connected Lie subgroup of $SO_\circ(2,3)$ admitting a $2$-dimensional orbit at $p\in \Ein^{1,2}$ and preserving a timelike line $\ell\leq \R^{2,3}$. Then $G$ preserves the orthogonal complement subspace $\ell^\perp\approx \R^{1,3}$ as well. Therefore, $G$ is a subgroup of $Stab_{SO_\circ(2,3)}(\ell)\simeq SO_\circ(1,3)$. It follows that, $G$ preserves a spacelike hypersphere $\Sn^2\subset \Ein^{1,2}$ (Definition \ref{def.1.7.14}), which is the projection by $\Pn$ of the nullcone $\Nu^{1,3}$ of $\ell^\perp\approx\R^{1,3}$. Also, $G$ preserves the complement of $\Sn^2$ in $\Ein^{1,2}$ which by Lemma \ref{lem.1.7.16} is conformally equivalent to the $3$-dimensional de-Sitter space $\dS^{1,2}$. Hence, in this case, the problem reduces to the consideration of conformal actions with an open orbit in $\partial\dS^{1,2}\approx\Sn^2$ or the isometric actions with a $2$-dimensional orbit in $\dS^{1,2}$.
\begin{theorem}\label{thm.4.2.1}
Let $G$ be a connected Lie subgroup of $\Iso_\circ(\dS^{1,2})\simeq SO_\circ(1,3)$ which acts on Einstein universe $\Ein^{1,2}=\dS^{1,2}\cup \partial \dS^{1,2}$ with cohomogeneity one. Then either $G$ fixes a point in $\Ein^{1,2}$ or it is conjugate to $SO(3)$ or it is $\Iso_\circ(\dS^{1,2})$. 
\end{theorem}
The actions admitting a fixed point in Einstein universe $\Ein^{1,2}$ has been studied in Section \ref{fix}. In fact, if a connected Lie subgroup $G\subset SO_\circ(1,3)$ fixes a point in the boundary $\partial\dS^{1,2}\approx\Sn^2$, then it is a subgroup of $(\R_+^*\times Y_E)\ltimes (\R e_2\oplus \R e_3)$ up to conjugacy (see Remark \ref{Space}). On the other hand, by the action of $\Iso_\circ(\dS^{1,2})$ on $\dS^{1,2}$ the stabilizer of a point is a $3$-dimensional subgroup isomorphic to $SO_\circ(1,2)$. Actually, it is conjugate to the Levi factor of $\Conf_\circ(\En^{1,2})$ which is considered in Section \ref{subsec.3.2.8}. Hence, we only discuss on the subgroups which admit no fixed point $\Ein^{1,2}$.

The following theorem will play a key rule in the proof of Theorem \ref{thm.4.2.1}.
\begin{theorem}\label{thm.4.2.2}
\cite[Theorem 1.1]{Di2}. Let $G$ be a connected (non necessarily closed) Lie subgroup of $SO(1,n)$ and assume that the action of $G$ on the Lorentzian space $\R^{1,n}$ is irreducible. Then $G = SO_\circ(1,n)$.
\end{theorem}

Let $G\subset \Iso(\dS^{1,3})\simeq SO_\circ(1,3)$ be a connected subgroup which preserves a proper linear subspace $V$ of $\R^{1,3}$. Observe that since $G$ preserves the orthogonal space $V^\perp$ as well, it is suffix to consider the case $\dim V\leq 2$. Furthermore, if $V$ (resp. $V^\perp$) contains a lightlike vector, then the intersection of $\Pn(V)$ (resp. $\Pn(V^\perp)$) with $\Pn(\Nu^{1,3})$ is a discrete subset $A$ consisting of one or two points. Hence, by connectedness, $G$ acts on $A$ trivially. If $V$ is a $1$-dimensional timelike subspace, then $G$ is a subgroup of $SO(3)$ -the maximal compact subgroup- up to conjugacy.


\textbf{Proof of Theorem \ref{thm.4.2.1}.} If the action of $G$ on $\R^{1,3}$ is irreducible, then by Theorem \ref{thm.4.2.2} $G=SO_\circ(1,3)$. Suppose that $G$ preserves a proper linear subspace $V\leq \R^{1,3}$. 
\begin{itemize}
\item If $V$ or $V^\perp$ contains a lightlike vector, then $G$ fixes a point in the spacelike hypersphere $\Sn^2$. 
\item If $V$ is a timelike line, then $G\subset SO(3)$ up to conjugacy. Since $G$ admits a $2$-dimensional orbit in $\Ein^{1.2}$, $\dim G\geq 2$. By the fact that $SO(3)$ has no $2$-dimensional Lie subgroup, $G=SO(3)$ up to conjugacy.
\item If $V$ is a spacelike line, then $G$ is a subgroup of $SO_\circ(1,2)$ up to conjugacy. Obviously, $SO_\circ(1,2)$ admits a fixed point in the de-Sitter component $\dS^{1,2}$.
\end{itemize}
This completes the proof.  \hfill $\square$

The action of $SO(3)$ is described in Section \ref{proper}. Thus, Theorem \ref{thm.4.2.1}, gives only one new cohomogeneity one action, namely $\Iso_\circ(\dS^{1,2})$. Obviously, $\Iso_\circ(\dS^{1,2})\simeq SO_\circ(1,3)$ acts on the de-Sitter component and its conformal boundary $\partial \dS^{1,2}\approx \Sn^2$ transitively.

\subsection{Actions on Anti de-Sitter component and its boundary}\label{sec.4.1}
Now, we study cohomogeneity one actions of connected subgroups of $\Conf(\Ein^{1,2})$ on Einstein universe $\Ein^{1,2}$ preserving a $1$-dimensional spacelike linear subspace of $\R^{2,3}$.

Let $G$ be a connected Lie subgroup of $SO_\circ(2,3)$ admitting a $2$-dimensional orbit in $\Ein^{1,2}$ and preserving a spacelike line $\ell\leq \R^{2,3}$. Then $G$ preserves the orthogonal complement subspace $\ell^\perp\approx \R^{2,2}$ as well, and so, it is a subgroup of $Stab_{SO_\circ(2,3)}(\ell)=SO_\circ(2,2)$. Hence, $G$ preserves an Einstein hypersphere (Definition \ref{def.1.7.17}), which is a copy of $2$-dimensional Einstein universe $\Ein^{1,1}\subset \Ein^{1,2}$. Also, $G$ preserves the complement of $\Ein^{1,1}$ in $\Ein^{1,2}$ which by Lemma \ref{lem.1.7.16} is conformally equivalent to Anti de-Sitter space $\AdS^{1,2}$. Hence, in this case, the problem reduces to the consideration of the conformal actions with an open orbit in $\partial\AdS^{1,2}\approx\Ein^{1,1}$ or the isometric actions with a $2$-dimensional orbit in $\AdS^{1,2}$.

Recall from Section \ref{subsec.1.7.3} that, the $3$-dimensional Anti de-Sitter space $\AdS^{1,2}$ is isometric to $\SL(2,\R)\subset M(2,\R)$ endowed with the metric induced from $(M(2,\R),-\det)$. Also, the identity component of $\Iso(\AdS^{1,2})$ is isomorphic to $(\SL(2,\R)\times \SL(2,\R))/\mathbb{Z}_2$. Moreover, recall from Remark \ref{rem.SL}, the Lie group $\SL(2,\R)$ has three $1$-parameter subgroups $Y_E$, $Y_P$ and $Y_H$ and a unique $2$-dimensional connected subgroup $\Aff$ up to conjugacy.
\begin{notation}\label{not.4.1.2}
We denote by $\mathcal{G}_\lambda$ the following $2$-dimensional subgroup of $Y_E\times\Aff$
\begin{align*}
\left\lbrace \left(\begin{bmatrix}
\cos(\lambda t) & \sin(\lambda t)\\
-\sin(\lambda t) & \cos(\lambda t)
\end{bmatrix},\begin{bmatrix}
e^t & s\\
0 & e^{-t}
\end{bmatrix}\right):t,s\in \R\right\rbrace,
\end{align*}
where $\lambda\in \R_+^*$ is a constant number.
\end{notation}
\begin{theorem}\label{thm.4.1.2}
Let $G$ be a connected Lie subgroup of $\Iso_\circ(\AdS^{1,2})\simeq (\SL(2,\R)\times \SL(2,\R))/\mathbb{Z}_2$ which acts on $\Ein^{1,2}=\AdS^{1,2}\cup\partial\AdS^{1,2}$ with cohomogeneity one. Then either $G$ fixes a point in $\Ein^{1,2}$ or it is compact or its action is orbitally-equivalent to one of the following groups
\begin{align*}
&Y_E\times Y_P,\;\;\;\;\;\;\;\;\;\;\;\;\;\;\;\;Y_E\times Y_H,\;\;\;\;\;\;\;\;\;\;\;\;\;\;\;\;\mathcal{G}_1,\;\;\;\;\;\;\;\;\;\;\;\;\;\;\;\;SO_\circ(2,1),\\
&\SL(2,\R)\times Y_H,\;\;\;\;\;\;\;\;\;\;\;\;\;\;\;\; \SL(2,\R)\times Y_P,\;\;\;\;\;\;\;\;\;\;\;\;\;\;\;\;\Iso_\circ(\AdS^{1,2}),
\end{align*}
where $SO_\circ(2,1)$ is the stabilizer of a spacelike direction in $\R^{2,2}$ by the action of $SO_\circ(2,2)$.
\end{theorem}

The cohomogeneity one isometric actions on $3$-dimensional Anti de-Sitter space has been studied by Ahmadi in \cite{AhA}. We will study some of the actions directly. 

The actions which fix a point in Einstein universe $\Ein^{1,2}$ were studied in Section \ref{fix}. Indeed, a subgroup of $SO_\circ(2,2)$ which admits a fixed point in the boundary $\partial \AdS^{1,2}=\Ein^{1,1}$ is a subgroup of $(\R_+^*\times Y_H)\ltimes (\R e_1\oplus \R e_2)$ up to conjugacy (see Remark \ref{Time}). On the other hand, by the action of $SO_\circ(2,2)$ on $\AdS^{1,2}$, the stabilizer of a point is a $3$-dimensional Lie subgroup isomorphic to $SO_\circ(1,2)$. In fact, it is conjugate to the Levi factor of $\Conf_\circ(\En^{1,2})$ which has been considered in Section \ref{subsec.3.2.8}. Hence, we only discuss on the subgroups on which fix no point in $\Ein^{1,2}$.


The group of conformal transformations on $2$-dimensional Einstein universe $\Ein^{1,1}$ is $PO(2,2)$ and its identity component is isomorphic to $\PSL(2,\R)\times \PSL(2,\R)$. As we mentioned earlier in Section \ref{subsec.1.7.3}, there is a canonical $\PSL(2,\R)\times \PSL(2,\R)$-invariant identification of $\Ein^{1,1}$ with $\R\Pn^1\times \R\Pn^1$. According to that, the left factor (resp. the right factor) $\PSL(2,\R)$ acts on every photon $\{*\}\times \R\Pn^1$ (resp. $\R\Pn^1\times \{*\}$) trivially. 

Note that, the classification of connected Lie subgroups of $(\SL(2,\R)\times \SL(2,\R))/\mathbb{Z}_2$ is equivalent to the classification of connected Lie subgroups of $\SL(2,\R)\times \SL(2,\R)$. Moreover, a connected Lie subgroup of $(\SL(2,\R)\times \SL(2,\R))/\mathbb{Z}_2$ and its corresponding connected Lie subgroup in $\SL(2,\R)\times \SL(2,\R)$ admit the same orbits in Einstein universe $\Ein^{1,2}=\AdS^{1,2}\cup \Ein^{1,1}$. Therefore, we may consider the cohomogeneity one actions of subgroups of $\SL(2,\R)\times\SL(2,\R)$ instead those of $(\SL(2,\R)\times\SL(2,\R))/\mathbb{Z}_2$.

Consider the following natural group morphisms 
\begin{align*}
&P_1:\SL(2,\R)\times \SL(2,\R)\longrightarrow \SL(2,\R), \;\;\;\;\;\;\;(g,h)\mapsto g\\
&P_2:\SL(2,\R)\times \SL(2,\R)\longrightarrow \SL(2,\R), \;\;\;\;\;\;\;(g,h)\mapsto h
\end{align*}
By the action of $\SL(2,\R)\times \SL(2,\R)$, the identity component of the stabilizer of a point in $\Ein^{1,1}\approx\R\Pn^1\times \R\Pn^1$ is conjugate to $\Aff\times \Aff$. Therefore, a subgroup $G\subset \Conf_\circ(\Ein^{1,1})$ admits no fixed point in $\Ein^{1,1}$ if and only if the first or the second projection of $\widehat{G}$ contains an elliptic element.

\begin{remark}\label{rem.ads}
By the action of $\SL(2,\R)\times \SL(2,\R)$ on $\Ein^{1,1}$ the identity component of the stabilizer of a photon is conjugate to $\SL(2,\R)\times \Aff$ which stabilizes $\phi=\{\infty\}\times \R\Pn^1$. This group acts on $\AdS^{1,2}\approx \SL(2,\R)$ transitively, since its Levi factor $\SL(2,\R)$ does. Moreover, the Levi factor $\SL(2,\R)$ preserves every photon $\{*\}\times\R\Pn^1$.
\end{remark}
\begin{remark}\label{rem.4.1.3}
The inversion map $\mathfrak i$ on $\SL(2,\R)$ sending $A$ to $A^{-1}$ is an isometry respect to the metric induced by $-\det$. Hence, by Theorem \ref{thm.lio}, $\mathfrak{i}$ extends to a unique global conformal transformation $\tilde{\mathfrak i}$ on $\Ein^{1,2}$. The restriction of $\tilde{\mathfrak i}$ on the boundary $\partial\AdS^{1,2}=\Ein^{1,1}=\R\Pn^1\times\R\Pn^1$ sends $([x],[y])$ to $([y],[x])$. On the other hand, the map $\J$ on $\SL(2,\R)\times \SL(2,\R)$ sending an element $(A,B)$ to $(B,A)$ is an Lie group isomorphism. It follows that, the action of a connected Lie subgroup $G\subset \SL(2,\R)\times \SL(2,\R)$ on $\Ein^{1,2}=\SL(2,\R)\cup \Ein^{1,1}$ is orbitally-equivalent to the action of $\J(G)$ via $\tilde{\mathfrak i}$.
\end{remark}
By the action of $\SL(2,\R)\times \SL(2,\R)$ on Anti de-Sitter space $\AdS^{1,2}\approx\SL(2,\R)$, the stabilizer of a point is conjugate to the graph of the identity map $Id_{\SL(2,\R)}:\SL(2,\R)\rightarrow\SL(2,\R)$ which is the stabilizer of the identity element $Id\in \SL(2,\R)$
\[Stab_{\SL(2,\R)\times \SL(2,\R)}(Id)=graph(Id_{\SL(2,\R)})=diag(\SL(2,\R),\SL(2,\R))\simeq \SL(2,\R).\] 
Indeed, an automorphism $\varphi:\SL(2,\R)\rightarrow\SL(2,\R)$ is a conjugation if and only if $graph(\varphi)\subset \SL(2,\R)\times\SL(2,\R)$ fixes a point in Anti de-Sitter space $\AdS^{1,2}\approx\SL(2,\R)$.
%
%

Theorem \ref{thm.4.1.2} follows from the following Proposition.

\begin{proposition}\label{prop.4.1.4}
Let $G\subset \SL(2,\R)\times \SL(2,\R)$ be a connected Lie  subgroup with $\dim G\geq 2$. Let $G$ fixes no point in $\Ein^{1,2}$ and $P_1(G),P_2(G)\neq \{Id\}$. Then $G$ or $\J(G)$ is conjugate to one of the following subgroups
\begin{align*}
&Y_E\times Y_P,\;\;\;\;\;Y_E\times Y_H,\;\;\;\;\;Y_E\times Y_E,\;\;\;\;\; \mathcal{G}_\lambda,\;\;\;\;\;Y_E\times \Aff,\;\;\;\;\;Y_E\times \SL(2,\R),\;\;\;\;\;graph(\varphi),\\
&\SL(2,\R)\times Y_H,\;\;\;\;\;\;\;\;\;\; \SL(2,\R)\times Y_P,\;\;\;\;\;\;\;\;\;\;\SL(2,\R)\times \Aff,\;\;\;\;\;\;\;\;\;\;\SL(2,\R)\times\SL(2,\R),
\end{align*}
where $\lambda\in \R^*$, and $\varphi:\SL(2,\R)\rightarrow\SL(2,\R)$ is an isomorphism which is not a conjugation.
\end{proposition}
\begin{proof}
Suppose that $G$ fixes no point in $\Ein^{1,1}$. Therefore, $G$ contains an element $(g,h)$, such that either $g$ or $h$ is elliptic. Without loosing generality, we may restrict ourselves to the case $g$ is elliptic. Hence, $P_1(G)$ is either conjugate to $Y_E$ or it is $\SL(2,\R)$. Denote by $\g$ the Lie algebra of $G$ and by $p_i$ the differential of $P_i$ at the identity element for $i=1,2$.

\textbf{Case I:} $P_1(G)=Y_E$. In this case $G$ is a subgroup of $Y_E\times \SL(2,\R)$ up to conjugacy. 
\begin{itemize}
\item If $\dim P_2(G)=1$, then $G$ is a $2$-dimensional subgroup of $Y_E\times P_2(G)$ up to conjugacy. Therefore, $G$ is conjugate to $Y_E\times Y_P$, $Y_E\times Y_H$, or $Y_E\times Y_E$.

\item If $\dim P_2(G)=2$, then $G$ is a subgroup of $Y_E\times \Aff$ up to conjugacy. 
\begin{itemize}
\item If $\dim G=3$, then $G=Y_E\times \Aff$ up to conjugacy.

\item If $\dim G=2$, the second projection $p_2$ from $\g$ to $\aff$ is a Lie algebra isomorphism. Hence, $f=:p_1\circ p_2^{-1}:\aff \rightarrow p_1(\h)=\R\Y_E$ is a surjective Lie algebra morphism. The kernel of $f$ is a $1$-dimensional ideal of $\aff$, hence $\ker f=\R\Y_P$. This induces an isomorphism from $\aff/\R \Y_P\simeq  \R\Y_H$ to $\R\Y_E$. Thus, there exists a real nonzero number $\lambda$, such that $f(t\Y_H+s\Y_P)=\lambda t\Y_E$, for all $t,s\in \R$. This implies that $G$ is conjugate to $\mathcal{G}_\lambda$.
\end{itemize}
\item If $\dim P_2(G)=3$, then $P_2(G)=\SL(2,\R)$.  Assume that $\dim G=3$. Then the map $p_2:\g\rightarrow \s(2,\R)$ is a Lie algebra isomorphism. Hence, $f:=p_1\circ p_2^{-1}:\s(2,\R)\rightarrow p_1(\g)$ is a surjective Lie algebra morphism. But, this contradicts the simplicity of $\s(2,\R)$, since $\ker f$ is a $2$-dimensional ideal of $\s(2,\R)$. Therefore, $\dim G=4$ and so, $G=Y_E\times \SL(2,\R)$ up to conjugacy.
\end{itemize}
\textbf{Case II: $P_1(G)=\SL(2,\R)$.}
\begin{itemize}
\item $\dim P_2(G)=1$. We claim that $\dim G\neq3$. On the contrary, assume that $\dim G=3$. Then $p_1:\g\rightarrow \s(2,\R)$ is an Lie algebra isomorphism. Thus, $f=p_2\circ p_1^{-1}:\s(2,\R)\rightarrow p_2 (\g)$ is a surjective Lie algebra morphism. But, this contradicts the simplicity of $\s(2,\R)$, since $\ker f$ is a $2$-dimensional ideal of $\s(2,\R)$. Thus $\dim G=4$ and $G$ is conjugate to $\SL(2,\R)\times Y_E$, $\SL(2,\R)\times Y_H$, or $\SL(2,\R)\times Y_P$.

\item $\dim P_2(G)=2$. In this case, $G$ is a subgroup of $\SL(2,\R)\times \Aff$ up to conjugacy. We show that $\dim G=5$. If $\dim G=3$, then using the same argument as the previous case, $f=p_2\circ p_1^{-1}:\s(2,\R)\rightarrow \aff$ is a surjective Lie algebra morphism. But, this contradicts the simplicity of $\s(2,\R)$, since $\ker f$ is a $1$-dimensional ideal of $\s(2,\R)$. If $\dim G=4$, then the kernel of $p_1:\g\rightarrow \s(2,\R)$ is a $1$-dimensional ideal of $\{0\}\oplus\aff$, hence $\ker p_1=\{0\}\oplus\R\Y_P$. In the one hand, $\g/\ker p_1\simeq \s(2,\R)$, thus it is a simple Lie algebra. On the other hand, the map
\[\g/\ker p_1\rightarrow \aff/ \R\Y_P,\;\;\;\;\;(X,a\Y_H)+\ker p_1\mapsto p_2(X,a\Y_H)+\R\Y_P=a\Y_H+\R\Y_P,\]
 is a surjective Lie algebra morphism, where $(X,aY_H)\in \g\leq \s(2,\R)\oplus\aff$. But, this contradicts the simplicity of $\g/\ker p_1$. Thus, $\dim G=5$, and so, $G=\SL(2,\R)\times \Aff$ up to conjugacy.

\item $\dim P_2(G)=3$. In this case $P_2(G)=\SL(2,\R)$, and $\dim G\in \{3,4,5,6\}$. We claim that $\dim G\notin \{4,5\}$, otherwise, $\ker P_1$ is a nontrivial proper normal Lie subgroup of $\{Id\}\times \SL(2,\R)$, which contradicts the simplicity of $\SL(2,\R)$. Thus $\dim G\in \{3,6\}$.

If $\dim G=3$, then $P_1,P_2:G\rightarrow\SL(2,\R)$ are isomorphisms and $G=graph(\varphi)\simeq \SL(2,\R)$ where $\varphi=P_1\circ P_2^{-1}$. In the one hand, $G$ fixes no point in the boundary $\partial\AdS^{1,2}=\Ein^{1,1}$, since the stabilizer of a point in $\Ein^{1,1}$ is a solvable group conjugate to $\Aff\times \Aff$. On the other hand, $\varphi$ is a conjugation if and only if $G$ is conjugate to $graph(id_{\SL(2,\R)})$ if and only if it fixes a point in $\AdS^{1,2}\approx\SL(2,\R)$ (which should be excluded by assumption).

If $\dim G=6$, then $G=\SL(2,\R)\times \SL(2,\R)$.
\end{itemize}
\end{proof}
\begin{lemma}\label{lem.4.50}
Let $\varphi:\SL(2,\R)\rightarrow \SL(2,\R)$ be an isomorphism which is not a conjugation, and let $G=graph(\varphi)$ admit a $2$-dimensional orbit in $\Ein^{1,2}$. Then the action of $graph(\varphi)$ on $(M(2,\R),-\det)$ preserves a unique spacelike line. Therefore, its action on $\R^{2,3}$ is conjugate to the action of $SO_\circ(2,1)$ which acts on a $2$-dimensional positive definite linear subspace (i.e. of signature $(0,2)$) of $\R^{2,3}$ trivially.
\end{lemma}
\begin{proof}
Consider the action of $G=graph(\varphi)$ on $(M(2,\R),-\det)$. By Theorem \ref{1}, $G$ preserves a non-trivial linear subspace $V$ of $(M(2,\R),-\det)$. Since it also preserves the orthogonal space $V^\perp$, it is suffix to focus on $\dim V\leq 2$. We show that $V$ has signature $(0,1)$. Assume the contrary:

\textit{Case I: $\dim V=1$}
\begin{itemize}
\item If $V$ is a timelike line, then obviously $G$ fixes two points on $\AdS^{1,2}\approx\SL(2,\R)$, namely, $\SL(2,\R)\cap V$, which contradicts the fact that $\varphi$ is not a conjugation. 
\item If $V$ is a lightlike line, then $G$ fixes a point in the boundary $\Ein^{1,1}$. This is a contradiction, since the stabilizer of a point in $\Ein^{1,1}$ is a solvable group isomorphic to $\Aff\times \Aff$.
\end{itemize}
\textit{Case II: $\dim V=2$}
\begin{itemize}
\item Assume that the restriction of $-\det$ on $V$ is definite (positive or negative). Hence, there is a representation $\rho:\SL(2,\R)\rightarrow SO(2)\times SO(2)$, since the identity component of the groups of linear isometries on $V$ and $V^\perp$ are isomorphic to $SO(2)$. By Lemma \ref{lem.1.4.7}, $\rho$ is trivial, a contradiction.
\item Suppose that the restriction of $-\det$ on $V$ has signature $(1,1)$, $(1,0,1)$, or $(0,1,1)$. In all the cases, $G$ preserves a lightlike line. Hence, it admits a fixed point in $\Ein^{1,1}$. Once more, a contradiction.
\item If $V$ has signature $(0,0,2)$, then it preserves a photon on $\Ein^{1,1}$. Thus, $G$ is conjugate to the Levi factor of $Stab_{\SL(2,\R)\times \SL(2,\R)}(\phi)=\SL(2,\R)\times \Aff$ which is $\SL(2,\R)$. By Remark \ref{rem.ads}, the Levi factor does not admit a $2$-dimensional orbit in $\Ein^{1,2}$. Again, a contradiction.
\end{itemize}
Henceforth, $V$ is a spacelike line in $M(2,\R)$, and $G$ acts on it trivially. We have assumed in the beginning of this section that $G$ preserves a spacelike line in $\R^{2,3}$ which is orthogonal to $M(2,\R)$. Thus $G$ acts on a positive definite $2$-dimensional linear subspace $\Pi$, generated by $V$ and $\ell$, of $\R^{2,3}$ trivially. This induces a surjective faithful representation of $G$ in $SO_\circ(2,1)$, the group of linear isometries of $\Pi^\perp$. This completes the proof.
\end{proof}
\subsubsection{\underline{Orbits}}\label{subsec.4.1.2}
Now, we consider the orbits induced in $\Ein^{1,2}$ by the subgroups obtained in Proposition \ref{prop.4.1.4}.  Note that the subgroup $Y_E\times Y_E$ is compact and its orbits has been described in Theorem \ref{thm.2.0.1}. Observe that, every subgroup containing $\SL(2,\R)$ as the left or right factor acts on the Anti de-Sitter component $\AdS^{1,2}\approx\SL(2,\R)$ transitively. In order to determine the orbits in $\AdS^{1,2}$, let 
\[p=\begin{bmatrix}
p_{11} & p_{12}\\
p_{21} & p_{22}
\end{bmatrix},\;\;\;\;\det(p)=1\]
be an arbitrary point in $\AdS^{1,2}\approx\SL(2,\R)$.

\begin{itemize}
\item $G$ is one of the groups $\SL(2,\R)\times \SL(2,\R)$ or $\SL(2,\R)\times Y_E$. It is obvious that $G$ acts on the Anti de-Sitter component $\AdS^{1,2}$ and its conformal boundary $\En^{1,1}$ transitively.

\item $G$ is one of the groups $\SL(2,\R)\times \Aff$ or $\SL(2,\R)\times Y_P$. Since $G$ contains $\SL(2,\R)$ as a subgroup, it acts on the Anti de-Sitter component transitively. In the both cases, $G$ preserves the photon $\phi=\{\infty\}\times\R\Pn^1$ and acts on $\Ein^{1,1}\setminus \phi$ transitively.

\item $G=\SL(2,\R)\times Y_H$. This group acts on Anti de-Sitter component transitively, since it contains $\SL(2,\R)$ as a subgroup. Also, it preserves the photons $\phi=\{\infty\}\times\R\Pn^1$ and $\psi=\{0\}\times \R\Pn^1$, and acts on the both connected components of $\Ein^{1,1}\setminus (\phi\cup\psi)$ transitively.

\item $G=Y_E\times \Aff$. This group preserves the photon $\phi=\{\infty\}\times\R\Pn^1$ and acts on $\Ein^{1,1}\setminus \phi$ transitively. Also, it acts on Anti de-Sitter component $\AdS^{1,2}$ freely. So, $G$ admits the same orbits as $\SL(2,\R)\times Y_P$.

\item $G=\mathcal{G}_\lambda$, $\lambda\in \R_+^*$. This group preserves the orbit induced by $Y_E\times \Aff$. It is not hard to see that $\mathcal{G}_\lambda$ acts on $\phi$ and $\Ein^{1,1}$ transitively. On the other hand, $\mathcal{G}_\lambda$ acts on $\AdS^{1,2}$ freely, since $Y_E\times \Aff$ does. Hence, all the $G$-orbits in $\AdS^{1,2}$ are $2$-dimensional. Denote by $H_\lambda$ the following $1$-parameter subgroup of $\mathcal{G}_\lambda$
\begin{align*}
\left\lbrace \left(\begin{bmatrix}
\cos(\lambda t) & \sin(\lambda t)\\
-\sin(\lambda t) & \cos(\lambda t)
\end{bmatrix},\begin{bmatrix}
e^t &0 \\
0 & e^{-t}
\end{bmatrix}\right):t\in \R\right\rbrace.
\end{align*}
 For an arbitrary point $p\in \AdS^{1,2}$, the vectors tangent to the orbit $\mathcal{G}_\lambda(p)$ at $p$ induced by $H_\lambda$ and $\{Id\}\times Y_P$ are:
\begin{align*}
v_\lambda=\lambda\begin{bmatrix}
-p_{21} & p_{22}\\
p_{11} & -p_{12}
\end{bmatrix},\;\;\;\;\;v_P=\begin{bmatrix}
0 & -p_{11}\\
0 & -p_{21}
\end{bmatrix},
\end{align*}
respectively. Observe that, the orthogonal space of the null vector $v_P$ in the tangent space $T_pG(p)$ is $\R v_P$. Hence, all the orbits induced by $\mathcal{G}_\lambda$ are Lorentzian. Note that, for all $\lambda\in \R_+^*$, the orbits induced by $\mathcal{G}_\lambda$ in $\Ein^{1,2}$ are exactly the same as the orbits induced by $\mathcal{G}_1$. In other words, $\mathcal{G}_\lambda$ is orbitally-equivalent to $\mathcal{G}_1$ via the identity map on $\Ein^{1,2}$ (Lemma \ref{lem.1.3.4}).

\item $G=Y_E\times Y_P$. It is not hard to see that $G$ acts on the photon $\phi=\{\infty\}\times \R\Pn^1\subset \Ein^{1,1}$ and on $\Ein^{1,1}\setminus \phi$ transitively. On the other hand, $G$ acts on $\AdS^{1,2}$ freely, since $Y_E\times \Aff$ does. For an arbitrary point $p\in \AdS^{1,2}$, the vector tangent to the orbit $G(p)$ at $p$ induced by the $1$-parameter subgroup $Y_E\times\{Id\}$ is
\begin{align*}
v=\begin{bmatrix}
p_{21} & p_{22}\\
-p_{11} & -p_{12}
\end{bmatrix}.
\end{align*}
Observe that $v$ is a timelike vector. Hence, $G$ admits a codimension $1$ foliation on $\AdS^{1,2}$ where every leaf is Lorentzian.

\item $G=Y_E\times Y_H$. It can be easily seen that $G$ acts on the photons $\phi=\{\infty\}\times \R\Pn^1, \psi=\{0\}\times \R\Pn^1\subset \Ein^{1,1}$, and on the two connected components of $\Ein^{1,1}\setminus (\phi\cup \psi)$ transitively. On the other hand $G$ acts on $\AdS^{1,2}$ freely, since $Y_E\times \Aff$ does. For an arbitrary point $p\in \AdS^{1,2}$, the vector tangent to the orbit $G(p)$ at $p$ induced by the $1$-parameter subgroup $Y_E\times\{Id\}$ is
\begin{align*}
v=\begin{bmatrix}
p_{21} & p_{22}\\
-p_{11} & -p_{12}
\end{bmatrix}.
\end{align*}
The tangent vector $v$ is timelike. Hence, $G$ admits a codimension $1$ foliation on $\AdS^{1,2}$ where every leaf is Lorentzian and $G$ acts freely.
 
 \item $G=SO_\circ(1,2)$. By Lemma \ref{lem.4.50}, $G$ preserves a linear subspace $V\leq \R^{2,3}$ of signature $(0,2)$ and its orthogonal complement $V^\perp$ (of signature $(2,1)$). Hence, $G$ preserves a timelike circle $\mathfrak{C}$ in Einstein universe (Definition \ref{def.1.7.22}). By Lemma \ref{lem.1.7.23} the complement of $\mathfrak{C}$ in $\Ein^{1,2}$ (up to double cover) is conformally equivalent to the product $\AdS^{1,1}\times \Sn^1$. The group $G$ preserves every fiber $\AdS^{1,1}\times\{*\}$ and acts on it transitively.
 \end{itemize}
\textbf{Proof of Theorem \ref{thm.4.1.2}.}  Assume that the first projection $P_1(G)$ is trivial. Then $G$ is either conjugate to $\Aff$ or it is $\SL(2,\R)$. Observe that, affine group admits a fixed point in $\Ein^{1,1}$. Also, $\SL(2,\R)$ acts on $\AdS^{1,2}$ transitively and preserves every photon $\{*\}\times\R\Pn^1$. Thus, it does not admit a $2$-dimensional orbit in $\Ein^{1,2}$. The same happens for the case $P_2(G)=\{Id\}$.

Now, suppose that $P_1(G),P_2(G)\neq \{Id\}$. Then, $G$ (or $\J(G)$) is conjugate to one of the subgroups mentioned in Proposition \ref{prop.4.1.4}. The theorem follows from the above consideration on orbits.\hfill$\square$

\section{Actions preserving a photon}\label{photon}
In this section, we consider the cohomogeneity one actions on Einstein universe $\Ein^{1,2}$ preserving a photon.
\begin{theorem}\label{thm.photon}
Let $\phi$ be a photon in $\Ein^{1,2}$ and $G$ a connected Lie subgroup of $\Conf(\Ein^{1,2})$. If $G$ preserves $\phi$ and admits a $2$-dimensional orbit in $\Ein^{1,2}$, then it fixes a point in the projective space $\R\Pn^4$. Therefore, $G$ preserves in Einstein universe $\En^{1,2}$ a Minkowski patch or a de-Sitter component or an Anti de-Sitter component.
\end{theorem}

Recall from Section \ref{sec.Ein} that, the complement of $\phi$ in $\Ein^{1,2}$ is an open homogeneous subset diffeomorphic to $\Sn^1\times \R^2$. The group of conformal transformations on $\Ein^{1,2}_\phi=\Ein^{1,2}\setminus \phi$ is $Stab_{O(2,3)}(\phi)$. Furthermore, $\Ein^{1,2}_\phi$ admits a codimension $1$ foliation $\mathcal{F}_\phi$ invariant by $\Conf(\Ein^{1,2}_\phi)$, for which each leaf is a degenerate surface diffeomorphic to $\R^2$. More precisely, choosing a point $x_0\in \phi$, one of the leaves is $L(x_0)\setminus \phi$ and the other leaves are the degenerate affine $2$-planes in the Minkowski patch $Mink(x_0)$ with limit point in $\phi$. Therefore, we may determine a leaf of $\mathcal F_\phi$ by its limit point $x\in\phi$ and denote it by $\mathcal{F}_\phi(x)$. In other words, for an arbitrary point $x\in \phi$, the leaf $\mathcal{F}_\phi(x)$ is the degenerate surface $L(x)\setminus \phi$.

By Lemma \ref{lem.Hiz}, the stabilizer of a photon in $\Ein^{1,2}$ is isomorphic to $(\R^*\times \SL(2,\R))\ltimes H(3)$, where $H(3)$ is the $3$-dimensional Heisenberg group. The action of $\Conf(\Ein^{1,2}_\phi)$ on $\phi\approx \R\Pn^1$ admits a surjective representation
$$\pi:\Conf_\circ(\Ein^{1,2}_\phi)\longrightarrow \PSL(2,\R)\simeq \Conf_\circ(\R\Pn^1).$$ 
The kernel $\mathcal{K}=\ker\pi$ is a four-dimensional Lie subgroup with two connected components. Fix an arbitrary point $x_0\in \phi$ and consider the Minkowski space $Mink(x_0)\approx\En^{1,2}$ with underlying Lorentzian vector space $(\s(2,\R),-\det)$ (described in Section \ref{proj}). Now, we have 
\begin{align}\label{ker}
\mathcal{K}=\left\lbrace\left( e^{2t}, \begin{bmatrix}
\varepsilon e^t & s \\
0 & \varepsilon e^{-t}
\end{bmatrix}
, \begin{bmatrix}
u & v\\
0& -u
\end{bmatrix}\right): \;t,s,u,v\in \R,\;\varepsilon =\pm 1\right\rbrace\subset (\R^*\times\PSL(2,\R))\ltimes \s(2,\R).
\end{align}
 The identity component $\mathcal{K}_\circ$ of the kernel is  conjugate to the Lie subgroup 
\begin{align*}
\mathsf{K}=\exp\left(\R(1+\Y_H)+\R \Y_P\right)\ltimes (\R(e_1+e_2)\oplus \R e_3),
\end{align*}
described in Remark \ref{rem.ker}.

Setting $t=0$ in Eq. \ref{ker}, we obtain a Lie subgroup of $\mathcal{K}$ which is the maximal unipotent subgroup isomorphic to the $3$-dimensional Heisenberg group $H(3)$. Indeed, $H(3)\subset \mathcal{K}$ coincides with the semi-direct product $Y_P\ltimes \Pi$ where $Y_P$ is the $1$-parameter parabolic subgroup of $\PSL(2,\R)$ and $\Pi$ is the unique degenerate $Y_P$-invariant $2$-plane in $\R^{1,2}=(\s(2,\R),-\det)$. The center of Heisenberg group $H(3)= Y_P\ltimes \Pi$ is a $1$-dimensional Lie subgroup $\el$ isomorphic to $\R$. Observe that, $\el$ is the set of lightlike (elements in $\Pi$) translations in the Minkowski patch $Mink(x_0)$. 

The following proposition gives a powerful tool to prove Theorem \ref{thm.photon}.
\begin{proposition}\label{pro.5.1.2}
Let $G$ be a connected Lie subgeoup of $\Conf(\Ein^{1,2}_\phi)$ which acts transitively on $\phi$ and admits a $2$-dimensional orbit at $p\in \Ein^{1,2}_\phi$. Then the connected component of $G\cap H(3)$ is a subgroup of $\el$. 
\end{proposition}

\begin{definition}\label{def.5.1.3}
Let $x\in \phi$ be an arbitrary point. A non-trivial element $g\in H(3)$ is called:
\begin{itemize}
\item a lightlike transformation on $Mink(x)$, if it fixes no point in $\Ein^{1,2}_\phi$.
\item a spacelike transformation on $Mink(x)$, if the set of its fixed points in $\Ein^{1,2}_\phi$ is a unique lightlike geodesic included in the leaf $\mathcal{F}_\phi(x)\subset L(x)$.
\item a parabolic transformation on $Mink(x)$, if the set of its fixed points in $\Ein^{1,2}_\phi$ is a unique lightlike geodesic in the Minkowski patch $Mink(x)$.
\end{itemize}
\end{definition}
Observe that, if $g\in H(3)$ is a lightlike transformation on a Minkowski patch $Mink(x_0)$ (for some $x_0\in \phi$), then for all $y\in \phi$, it  is a lightlike transformation on the Minkowski patch $Mink(y)$. 
Thus, we may talk about a lightlike transformation without mentioning a Minkowski patch. Indeed, an element $g\in H(3)$ is a lightlike transformation if and only if $g$ belongs to the center $\el$.

Assume that $g\in H(3)$ is a spacelike transformation on a Minkowski patch $Mink(x_0)$ (for some $x_0\in \phi$). 
Actually, $g$ is a spacelike translation in the Minkowski patch $Mink(x_0)$, and so, the $g$-invariant subsets of $\mathcal{F}_\phi(x_0)=L(x_0)\setminus \phi$ are included in the vertex-less photons in the lightcone $L(x_0)$. More precisely, $g$ preserve no spacelike curve in the degenerate surface $\mathcal{F}_\phi(x_0)$.

Furthermore, assume that $g\in H(3)$ is a parabolic transformation on a Minkowski patch $Mink(x)$ (for some $x\in \phi$). Denote by $\gamma$ the lightlike geodesic in $Mink(x)$ which is fixed pointwisely by $g$. By continuity, $g$ fixes the limit point of $\gamma$ in the lightcone $L(x)$. Hence, the limit point of $\gamma$ is contained in $\phi$, since $g$ fixes no point in the leaf $\mathcal{F}_\phi(x)=L(x)\setminus \phi$. This shows that $\gamma$ is contained in a leaf of $\mathcal{F}_\phi$. One can see that the $1$-parameter subgroup of $H(3)$ generated by $g$ acts on the leaf $\mathcal{F}_{\phi}(x)$ freely and every orbit is a spacelike curve (i.e. of signature $(0,1)$).
\begin{lemma}\label{lem.5.1.4}
Let $g$ be a non-trivial element in $H(3)$. Then there exists a unique point $x_0\in \phi$ such that $g$ is a spacelike transformation on $Mink(x_0)$ if and only if for all $x\in \phi\setminus \{x_0\}$, $g$ is a parabolic transformation on $Mink(x)$.
\end{lemma}
\begin{proof}
Assume that $g$ is a spacelike transformation on $Mink(x_0)$. Then the set of points fixed by $g$ in $\Ein^{1,2}_\phi$ is a unique lightlike geodesic $\gamma\subset \mathcal{F}_\phi(x_0)=L(x_0)\setminus \phi$. 
For an arbitrary point $x\in \phi\setminus \{x_0\}$, $\gamma$ is a lightlike geodesic in the Minkowski patch $Mink(x)$. Hence, $g$ is a parabolic transformation of $Mink(x)$ for all $x\in \phi\setminus \{x_0\}$.

Conversely, assume that $g$ is a parabolic transformation in a Minkowski patch $Mink(x)$  (for some $x\in \phi$). Then the set of points fixed by $g$ in $\Ein^{1,2}_\phi$ is a unique lightlike geodesic $\gamma\subset Mink(x)$. Let $x_0$ denotes the limit point of $\gamma$ in the lightcone $L(x)$. Obviously, $x_0\in \phi$. Observe that $\gamma$ is a vertex-less photon in $\mathcal{F}_\phi(x_0)=L(x_0)\setminus \phi$. Hence, $g$ is a spacelike transformation on $Mink(x_0)$. This completes the proof.
\end{proof}
\begin{corollary}\label{cor.5.1.5}
Let $g\in H(3)$ be a non-trivial element. Then, either $g\in \el$ (hence, for all $x\in \phi$, it is a lightlike transformation on the Minkowski patch $Mink(x)$), or there exists a unique point $x_0\in \phi$ such that $g$ is a spacelike transformation on $Mink(x_0)$ and for all $x\in \phi\setminus \{x_0\}$ it is a parabolic transformation on $Mink(x)$. 
\end{corollary}
\begin{proof}
Considering Eq. \ref{ker}, every non-trivial element in $H(3)$ with $s=u=0$ (resp. $s=v=0$, resp $s\neq 0$) is a lightlike (resp. spacelike, parabolic) transformation of Minkowski patch $Mink(x_0)$. Now, the corollary follows from Lemma \ref{lem.5.1.4}.
\end{proof}

\begin{proposition}\label{pro.5.1.6}
Let $\phi$ be a photon in $\Ein^{1,2}$ and $G$ a connected Lie subgroup of $\Conf(\Ein^{1,2}_\phi)$ acting transitively on $\phi$. Then, for $x\in \phi$ and $p\in \mathcal{F}_{\phi}(x)$ the orbit induced by $G$ at $p$ is $k$-dimensional if and only if the orbit induced by $Stab_{G}(x)$ at $p$ is $(k-1)$-dimensional. 
\end{proposition}
\begin{proof}
In the one hand, since $G$ acts on $\phi$ transitively, $\dim G=\dim Stab_{G}(x)+1$, for all $x\in \phi$. On the other hand, for $p\in \mathcal{F}_\phi(x)$, $Stab_G(p)$ is a subgroup of $Stab_G(x)$, since the action of $G$ preserves the foliation $\mathcal{F}_\phi$. More precisely, $Stab_{G}(p)=Stab_{Stab_G(x)}(p)$. Hence, 
\begin{align*}
\dim G(p)&=\dim G-\dim Stab_G(p)\\
&=\dim Stab_G(x)+1-\dim Stab_{Stab_G(x)}(p)\\
&=\dim \big(Stab_G(x)\big)(p)+1.
\end{align*}
\end{proof}

\textbf{Proof of Proposition \ref{pro.5.1.2}.} Assume the contrary, where $g$ is an element in the connected component of $G\cap H(3)$ and $g\notin \el$. There exists a point $x_0\in \phi$ such that $p\in \mathcal{F}_\phi(x_0)$. By Corollary \ref{cor.5.1.5}, $g$ is either a spacelike or parabolic transformation on the Minkowski patch $Mink(x_0)$. Denote by $O_p$ and $g^t$, the orbit induced by $Stab_G(x_0)$ at $p$ and  a $1$-parameter subgroup of $G\cap H(3)$ containing $g$, respectively. By Proposition \ref{pro.5.1.6}, the orbit $O_p$ is $1$-dimensional.
\begin{itemize}
\item If $g$ is a spacelike transformation on $Mink(x_0)$, then $\gamma\subset \mathcal{F}_\phi(x_0)=L(x_0)\setminus \phi$. 
For an arbitrary point $x\in \phi\setminus \{x_0\}$, there exists $h\in G$ such that $hx_0=x$, since $G$ acts on $\phi$ transitively. By Lemma \ref{lem.5.1.4} $g$ is a parabolic transformation on $Mink(x)$. Hence, $g^t$ acts on $\mathcal{F}_\phi(x)$ freely. Thus, the orbit $C=g^t(hp)$ is an open subset of the orbit induced by $Stab_G(x)$ at $hp$ and it is spacelike (i.e. of signature $(0,1)$). Hence, the orbit induced by $Stab_G(x)$ at $hp$ is spacelike. Obviously, $h^{-1}(C)\subset \mathcal{F}_\phi(x_0)$ is an open subset of $O_{p}$. This is a contradiction, since $g$ preserves no spacelike curve in $\mathcal{F}_\phi(x_0)$.
\item If $g$ is a parabolic transformation of $Mink(x_0)$, then the orbit $C=g^t(p)$ is a spacelike (i.e. of signature $(0,1)$) curve in $\mathcal{F}_\phi(x_0)$. By Lemma \ref{lem.5.1.4}, there exists a unique point $x\in \phi\setminus \{x_0\}$ such that $g$ is a spacelike transformation on the Minkowski patch $Mink(x)$. There exists $h\in G$ such that  $hx_0=x$. The same argument as the previous case shows that $h(C)$ is an open subset of the orbit induced by $Stab_G(x)$ at $hp$. This contradicts the fact that $g$ preserves no spacelike curve in $\mathcal{F}_\phi(x)$.
\end{itemize}
This completes the proof.\hfill $\square$

\begin{definition}\label{def.5.1.7}
A non-trivial element $g\in \mathcal{K}$ is called a hyperbolic-homothety (abbreviation $\mathcal{H}$) -transformation if the set of its fixed points in $\Ein^{1,2}_\phi$ is a photon.
\end{definition}
Considering Eq. \ref{ker}, elements with $t\neq 0$ are $\mathcal{H}$-transformations.
Let $g\in \mathcal{K}$ be an $\mathcal{H}$-transformation and $\psi\subset \Ein^{1,2}_\phi$ be the unique photon fixed pointwisely by $g$. The photon $\psi$ intersects every leaf of $\mathcal{F}_\phi$ in a unique point. More precisely, for an arbitrary point $x\in \phi$, the intersection of $\psi$ with the Minkowski patch $Mink(x)$ is a lightlike geodesic $\gamma$. The limit point of $\gamma$ is contained in $\mathcal{F}_{\phi}(x)= L(x)\setminus \phi$. Hence, $\gamma$ intersects every affine degenerate plane in $Mink(x)$ with limit point in $\phi$. If $g\in \mathcal{K}_\circ$, denote by $g^t$ an arbitrary $1$-parameter subgroup of $\mathcal{K}_\circ$ containing $g$. Observe that $g^t$ preserves $\psi$, since it is abelian. On the other hand $g^t$ preserves the leaves of the foliation $\mathcal{F}_\phi$. Hence, $g^t$ acts on $\gamma$ trivially. Consequently, $\gamma$ is the unique photon fixed pointwisely by every element in $g^t$.

\begin{corollary}\label{cor.5.1.8}
Let $g$ be a non-trivial element in $\mathcal{K}$. Then either $g\in H(3)$ or it is an $\mathcal{H}$-transformation.
\end{corollary}
\begin{lemma}\label{lem.5.1.9}
Let $x_0\in \phi$ be an arbitrary point. Then every $\mathcal{H}$-transformation in $\mathcal{K}$ preserves a unique $\el$-invariant affine Lorentzian $2$-plane in $Mink(x_0)$.
\end{lemma}
\begin{proof}
Let $g\in \mathcal{K}_\circ$ be an $\mathcal{H}$-transformation. Denote by $\psi$ the unique photon in $\Ein^{1,2}_\phi$ fixed pointwisely by $g$. Assume that $p$ is an arbitrary point of the lightlike geodesic $\gamma=Mink(x_0)\cap \psi$. The orbit induced by the $1$-parameter subgroup $\el$ at $p$ is a lightlike geodesic $\eta$ with limit point in $\phi$. The lightlike geodesics $\gamma$ and $\eta$ generate a unique affine Lorentzian $2$-plan $T\subset Mink(x_0)$. Obviously, $T$ is invariant by $\el$.
\end{proof}
Observe that, a $1$-parameter subgroup $g^t$ generated by an $\mathcal{H}$-transformation $g\in \mathcal{K_\circ}$ preserves the unique $g$-invariant affine Lorentzian $2$-plane in $Mink(x_0)$, since $g^t$ is abelian.
\begin{notation}\label{not.5.1.10}
For a Lie subgroup $G\subset \Conf(\Ein^{1,2}_\phi)$, we denote by $K^G$ and $K^G_\circ$ the kernel $\ker\pi|_G$ and its identity component, respectively.
\end{notation}
\begin{lemma}\label{lem.5.1.11}
Let $x_0\in \phi$ and $G$ be a $1$-parameter subgroup of $\Conf(\Ein^{1,2}_\phi)$ with $\dim \pi(G)=1$ which acts on $\phi$ transitively. 
Then the kernel $K^G=\ker\pi|_G$ contains neither a parabolic nor a spacelike transformation of $Mink(x_0)$.
\end{lemma}
\begin{proof}


Assume the contrary that $g\in K^G$ is a parabolic or spacelike transformation on $Mink(x_0)$. The element $g$ fixes pointwisely a unique lightlike geodesic $\gamma\subset \Ein^{1,2}_\phi$ contained in a leaf of $\mathcal{F}_\phi$. Note that, $G$ preserves $\gamma$, since it is abelain. But, this contradicts the fact that every $G$-orbit in $\Ein^{1,2}_\phi$ intersects every leaf of $\mathcal{F}_\phi$. This completes the proof.
\end{proof}
\begin{lemma}\label{lem.5.1.12}
Let $G$ be a connected Lie subgroup of $\Conf(\Ein^{1,2}_\phi)$ with $\dim \pi(G)=1$ which acts on $\phi$ transitively and admits a $2$-dimensional orbit in $\Ein^{1,2}$. Then, there exists a $1$-parameter subgroup $L\subset G$ transversal to $K^G_\circ$ such that either the kernel $K^L=\ker\pi|_L$ is trivial (hence $L\simeq SO(2)$), or every non-trivial element of $K^L$ is an $\mathcal{H}$-transformation.
\end{lemma}
\begin{proof}
Let $L\subset G$ be an arbitrary $1$-parameter subgroup transversal to $K^G_\circ$. Obviously, $L$ acts on $\phi$ transitively. If $K^L$ is trivial, then lemma follows evidently. Otherwise, $L$ is isomorphic to $\R$ and $K^L\simeq \mathbb{Z}$. Assume that $g\in K^L$ is a generator. By Lemma \ref{lem.5.1.11}, $g$ is either an $\mathcal{H}$-transformation or a lightlike transformation. In the first, lemma follows easily. If $g$ is a lightlike transformation, by Proposition \ref{pro.5.1.2}, there are two possibilities:
\begin{itemize}
\item \textit{$\el$ is a subgroup of $K^G_\circ$:} There exists a $1$-parameter subgroup $L'\subset G$ transversal to $K^G_\circ$ such that it intersects $\el$ only in the identity element, since $\el\subset G$.

\item \textit{$\el$ is not a subgroup of $K^G_\circ$}: Observe that in this case, $K^G_\circ$ is a $1$-parameter subgroup consisting of the identity element and $\mathcal{H}$-transformations. Also, $G$ is a $2$-dimensional connected Lie group and so, by \cite[p.p. 212]{Onish} it is isomorphic to the $2$-torus $\mathbb{T}^2$, $\R\times SO(2)$, $\R^2$, or $\Aff$.  Let $h\in K^G_\circ$ be an arbitrary non-trivial element. Then $hg$ is an $\mathcal{H}$-transformation by Corollary \ref{cor.5.1.8}. Since the exponential map $\exp:Lie(G)\rightarrow G$ is surjective, there exists a $1$-parameter subgroup $L'$ through $hg$. In the one hand, $hg\in K^{L'}$, so all the non-trivial elements in $K^{L'}$ are $\mathcal{H}$-transformations. On the other hand,  $L'$ is transversal to $K^G_\circ$, since $hg\notin K^G_\circ$. 
\end{itemize}
\end{proof}

Let $P$ be the totally isotropic $2$-plane in $\R^{2,3}$ corresponding to the photon $\phi\subset \Ein^{1,2}$, and $Q$ be a subspace of $\R^{2,3}$ supplementary to $P^\perp$. There is a canonical identification between $Q$ and the dual space $P^*$. Let $\langle.,.\rangle$ denote the bilinear form on $\R^{2,3}$. The map $\vartheta$ sending a vector $v\in Q$ to the functional $\langle v,.\rangle:P\rightarrow\R$ is linear. Also, $\vartheta$ is injective: $\vartheta(v)=\vartheta(w)$ implies that $\langle v-w,.\rangle\equiv 0$, since $\langle.,.\rangle$ is non-degenerate, hence we get $v-w=0$. Thus, $\vartheta$ is an isomorphism.

 Furthermore, let $H\subset SO_\circ(2,3)$ preserves $P$, and $Q$ be an $H$-invariant complement for $P^\perp$ in $\R^{2,3}$. In the one hand, the action of $H$ on $P$ induces a representation (not unique) from $H$ to $GL(P^*)$ by duality. On the other hand, the action of $H$ on $Q$ induces a unique representation from $H$ to $GL(P^*)$ on which the isomorphism $\vartheta$ is $H$-equivariant. It is not hard to see that the representation induced via $P$ is conjugate to the one induced via $Q$.
 
 Let $\psi\subset \Ein^{1,2}_\phi$ be a photon and denote by $P_\psi$ its corresponding totally isotropic $2$-plane in $\R^{2,3}$. The linear subspace $P\cup P_\psi\leq \R^{2,3}$ is of signature $(2,2)$. Hence, the union of $\phi$ and $\psi$ determines a unique Einstein hypersphere $\Ein^{1,1}\subset \Ein^{1,2}$. Let $x_0\in \phi$ be an arbitrary point. Then by Remark \ref{rem.ein}, the intersection of $\Ein^{1,1}$ with the Minkowski patch $Mink(x_0)$ is a Lorentzian affine $2$-plane $T$. The intersection of $\Ein^{1,1}$ with the lightcone $L(x_0)$ is the lightcone of $x_0$ in $\Ein^{1,1}$ consisting of $\phi$ and another photon $\xi$ which contain the limit points of lightlike geodesics in $T$. Indeed, $\phi\cup \xi$ is the set of the limit points of lightlike geodesics of every Lorentzian affine $2$-plane in $Mink(x_0)$ parallel to $T$. But, $T$ is the unique such affine plane which contains the lightlike geodesic $\psi\cap Mink(x_0)$.
 

\textbf{Proof of Theorem \ref{thm.photon}.}  First, consider the action of $G$ on $\phi$. If $G$ admits a fixed point $x\in \phi$, then $x$ is the desired fixed point. 

Now, assume that $G$ acts on $\phi$ transitively. Obviously, $\pi(G)$ is either $\PSL(2,\R)$ or it is conjugate to $SO(2)\simeq Y_E$. We show that in the both cases, $G$ preserves a line in $\R^{2,3}$. 

\textbf{Case I:} $\boldsymbol{\pi(G)=\PSL(2,\R)}$. Denote by $\g$ and $\ki$ the Lie algebras correspond to $G$ and $K^G$, respectively. The following short sequence of Lie algebras and Lie algebra morphisms is exact.
\begin{align*}
1\longrightarrow \ki\hookrightarrow \g\overset{d\pi}{\longrightarrow} \s(2,\R)\longrightarrow 1.
\end{align*}
One can see $ \s(2,\R)$ as the Levi factor of $\g$, since $\ki$ is the radical solvable ideal of $\g$. Henceforth, $\PSL(2,\R)$ is a subgroup of $G$, up to finite cover, and $G=K^G.\PSL(2,\R)$. It is clear that $G$ acts on the totally isotropic plane $P$ irreducibly and preserves the orthogonal space $P^\perp$. By Proposition \ref{pro.5.1.2}, the connected component of the  intersection of $G$ with Heisenberg group $H(3)$ is either trivial or it is $\el\simeq \R$. Therefore, $K^G_\circ$ is either trivial or it is isomorphic to $\R$ or $\Aff$. The subgroup $\PSL(2,\R)\subset G$ acts on $K^G$ by conjugacy, since $K^G$ is a normal subgroup of $G$. The simplicity of $\PSL(2,\R)$, and Lemma \ref{lem.1.4.10} imply that this action is trivial. Moreover, using  the simplicity of $\PSL(2,\R)$ again, its action on $\R^{2,3}$ splits as the sum of irreducible actions (see \cite[p.p. 28]{Hu}).

Suppose that $\R^{2,3}=P\oplus \ell\oplus Q$ is a $\PSL(2,\R)$-invariant splitting, where $\ell\leq P^\perp$ is a line supplementary to $P$ in $P^\perp$ and $Q$ is a $2$-plane supplementary to $P^\perp$ in $\R^{2,3}$. The canonical identification between $Q$ and $P^*$ shows that $\PSL(2,\R)$ acts on $Q$ irreducibly. It is not hard to see that $\ell$ is the only $\PSL(2,\R)$-invariant line in $\R^{2,3}$. Since the conjugacy action of $\PSL(2,\R)$ on $K^G$ is trivial, every element of $\PSL(2,\R)$ commutes with all the elements of $K^G$. This implies that $K^G$ preserves $\ell$ as well. Thus, any element of $G$ preserves $\ell$. Henceforth $\Pn(\ell)\in \R\Pn^4$ is the desired fixed point.

\textbf{Case II: $\boldsymbol{\pi(G)=SO(2)}$ up to conjugacy.} In this case $\dim K^G\geq 1$, since $\dim G\geq 2$. Proposition \ref{pro.5.1.2} implies that $K^G_\circ$ is isomorphic to either $\R$ or $\Aff$. We split this case to two subcases: $G$ contains a $1$-dimensional compact subgrgoup (a copy of $SO(2)$), and there is no $1$-dimensional compact subgroup in $G$:
\begin{itemize}
\item \textit{$G$ contains a $1$-dimensional compact subgroup.} In this case we have $G=K^G_\circ.SO(2)$. The group $SO(2)\subset G$ acts on $K^G_\circ$ by conjugacy, since $K^G_\circ$ is a normal subgroup of $G$. This action is trivial, since both $\R\simeq Aut_\circ(\R)$ and $\Aff\simeq Aut_\circ(\Aff)$ contain no $1$-dimensional compact Lie subgroup (see Lemma {\ref{lem.1.4.10}}). Furthermore, the action of $SO(2)$ on $\R^{2,3}$ splits as the sum of irreducible actions, since it is compact (see \cite[Proposition 4.36]{Hall}). Using the same symbols as we used for the previous case, suppose that $\R^{2,3}=P\oplus \ell\oplus Q$ is a $SO(2)$-invariant splitting. It is easy to see that $SO(2)$ acts irreducibly on $Q$. Also, $\ell$ is the only $SO(2)$-invariant line in $\R^{2,3}$. Since the conjugacy action of $SO(2)$ on $K^G_\circ$ is trivial, every element of $SO(2)$ commutes with all the elements of $K^G_\circ$. This implies that $K^G_\circ$ preserves $\ell$ as well. Consequently, all the elements in $G$ preserve $\ell$. Henceforth $\Pn(\ell)\in \R\Pn^4$ is the desired fixed point.

\item \textit{$G$ contains no $1$-dimensional compact subgroup.} In this case, every $1$-parameter subgroup  transversal to $K^G_\circ$ is isomorphic to $\R$. By Lemma \ref{lem.5.1.12}, there exists a $1$-parameter subgroup $L\subset G$ transversal to $K^G_\circ$ such that the kernel $K^L=\ker\pi|_L$ consists of the identity element and $\mathcal{H}$-transformations.
Let $g\in K^L$ be a non-trivial element. Then $g$ fixes a unique photon $\psi\subset \Ein^{1,2}_\phi$ pointwisely. Observe that $L$ preserves $\psi$ since it is abelian. Therefore, $L$ preserves the Einstein hypersphere $\Ein^{1,1}\subset \Ein^{1,2}$ containing $\phi$ and $\psi$. Also, by Lemma \ref{lem.5.1.9}, $g$ preserves a unique $\el$-invariant affine Lorentzian $2$-plane $T_g$ in the Minkowski patch $Mink(x_0)$. In fact, $T_g$ coincides with the intersection of $\Ein^{1,1}$ with $Mink(x_0)$. It is suffix to show that $K_\circ^G$ preserves $T_g$. If $K_\circ^G=\el$, then obviously $G=K^G_\circ.L$ preserves the Einstein hypersphere $\Ein^{1,1}$, since $T_g$ is $\el$-invariant. If $K_\circ^G$ contains an $\mathcal{H}$-transformation $h$, then by Lemma \ref{lem.5.1.9}, $h$
preserves a unique $\el$-invariant affine Lorentzian $2$-plane $T_h$. In fact, $T_h=T_g$, since $K_\circ^G$ is normal in $K^G$. Now, it is easy to see that $T_g$ is $K^G_\circ$-invariant. Hence, once more, $G=K^G_\circ.L$ preserves $\Ein^{1,1}$. Therefore, $G$ preserves the spacelike direction in $\R^{2,3}$ corresponding to $\Ein^{1,1}$, and so, it admits a fixed point in the projective space $\R\Pn^4$.
\end{itemize}
\hfill$\square$
\section{Two-dimensional orbits induced in Minkowski patch and lightcone }\label{orbit}
According to Theorem \ref{thm.fix}, every connected Lie subgroup of $\Conf(\En^{1,2})$ with $\dim \geq 2$ acts on Einstein universe $\Ein^{1,2}$ with cohomogeneity one, except $\R^{1,2}$ and $\R_+^*\ltimes \R^{1,2}$. The connected Lie subgroups of $\Conf(\En^{1,2})$ with dimension greater that or equal to $2$ are classified in Theorem \ref{thm.3.2.1} up to conjugacy. In this section, we describe the topology and causal character of the codimension one orbits in $\Ein^{1,2}=Mink(p)\cup L(p)$ induced by the cohomogeneity one action of the Lie subgroups indicated in Theorem \ref{thm.3.2.1}. For more details about the other orbits see \cite{Has0}.

For an arbitrary point $q\in \En^{1,2}\approx Mink(p)$, there is a natural identification between the tangent space $T_q\En^{1,2}$ and the underlying scalar product space $\R^{1,2}$. Therefore, by the action of a Lie subgroup $G\subset \Conf(\En^{1,2})$, we may always consider the translation part $T(G)$ as a linear subspace of $T_qG(q)$, since $\R^{1,2}$ acts on $\En^{1,2}$ freely.
 
We fix some notations here. Denote by $\Pi_\phi$ and $\phi$ the unique degenerate plane $\R(e_1+e_2)\oplus\R e_3\leq \R^{1,2}$ invariant by the $1$-parameter parabolic subgroup $Y_P\subset SO_\circ(1,2)$ and its corresponding photon in $L(p)$, respectively. Furthermore, denote by $\Pi_\psi$ and $\psi$ the degenerate plane $\R(e_1-e_2)\oplus \R e_3\leq \R^{1,2}$ and its corresponding photon in $L(p)$, respectively, which both are invariant by the $1$-parameter hyperbolic subgroup $Y_H\subset SO_\circ(1,2)$. Also, for a linear subspace $V\leq \R^{1,2}$, denote by $\mathcal{F}_V$ the foliation induced by $V$ in Minkowski space $\En^{1,2}$.

\begin{remark}
Let $\lambda +X+v$ be an arbitrary element in $(\R\oplus\so(1,2))\oplus_\theta\R^{1,2}$ and $q\in\En^{1,2}$ be an arbitrary point. There is an easy way to determine the tangent vector in $T_q\En^{1,2}$ induced by the action of $\exp(\R(\lambda+X+v))$ on $\En^{1,2}$: the vector $\frac{d}{dt}|_{t=0}\big(\exp(t(\lambda+X+v))(q)\big)$ coincides with $\lambda X(q)+v$ where $\lambda$ and $X$ act on $\En^{1,2}$ (with origin $o$) as linear maps.
\end{remark}

\
Now, we are ready to describe the $2$-dimensional orbits induced by the subgroups indicated in Tables \ref{table1}-\ref{table8}.
\subsubsection*{\underline{Orbits induced by subgroups with full translation part}}\label{subsec.3.2.1}
Here, we consider the codimension one orbits of a Lie subgroup $G\subset \Conf_\circ(\En^{1,2})$ with $T(G)=\R^{1,2}$. These groups have been listed in Table \ref{table1}. Obviously, $G$ acts on Minkowski space $\En^{1,2}$ transitively, since it  contains $\R^{1,2}$ as a subgroup. Therefore, the $2$-dimensional $G$-orbit is contained in the lightcone $L(p)$. The translation part $T(G)=\R^{1,2}$ acts on each vertex-less photon in $L(p)$ transitively. 

Note that, the groups $\R^{1,2}$ and $\R_+^*\times \R^{1,2}$ admit $2$-dimensional orbit neither in $\En^{1,2}$ nor in the lightcone $L(p)$.

\begin{itemize}
\item If the linear isometry projection $P_{li}(G)$ contains an elliptic element, then $G$ acts on the vertex-less lightcone $L(\hat{p})$ transitively. Hence, the following Lie subgroups admit a unique $2$-dimensional orbit in $\Ein^{1,2}$ i.e., the vertex-less lightcone $L(p)$ which is a degenerate surface (i.e., of signature $(0,1,1)$).
\begin{align*}
&(\R_+^*\times SO_\circ(1,2))\ltimes \R^{1,2}, \;\;\;\;\;\;\;\;SO_\circ(1,2)\ltimes \R^{1,2},\;\;\;\;\;\;\;\;(\R_+\times Y_E)\ltimes \R^{1,2},\\
&Y_E\ltimes \R^{1,2},  \;\;\;\;\;\;\;\;\;\;\;\;\;\;\;\;\;\;\;\;\;\;\;\;\;\;\;\;\;\exp\big(\R(a+\Y_E))\ltimes \R^{1,2}, \; where\; a\in \R^*.
\end{align*}

\item If the linear isometry projection $P_{li}(G)$ is a proper subgroup of $SO_\circ(1,2)$ and it contains a parabolic element, then $G$ preserves a unique photon $\phi$, and acts on its complement in the lightcone $L(p)$ transitively. Thus all the following groups admit a unique $2$-dimensional orbit in $\Ein^{1,2}$ which is the degenerate surface $L(p)\setminus \phi$.
\begin{align*}
&(\R_+^*\times \Aff)\ltimes \R^{1,2}, \;\;\;\;\;\Aff\ltimes \R^{1,2}, \;\;\;\;\;\;\exp\big(\R(a+\Y_P))\ltimes \R^{1,2},\\
 &Y_P\ltimes \R^{1,2}\;\;\;\;\;\;\;\;\;\;\;(\R_+\times Y_P)\ltimes \R^{1,2},\;\;\;\;\;\;\;  \exp\big(\R(a+\Y_H)+\R\Y_P\big)\ltimes \R^{1,2},\;where \; a\in \R^*
\end{align*}

\item If the linear isometry projection $P_{li}(G)$ is a $1$-parameter hyperbolic subgroup of $SO_\circ(1,2)$, then $G$ preserves two distinct photons $\phi$ and $\psi$, and acts on the two connected components of $L(p)\setminus (\phi\cup \psi)$ transitively. This implies that the following groups admit in $\Ein^{1,2}$ two $2$-dimensional orbits which are the connected components of $L(p)\setminus\{\phi\cup \psi\}$.
\begin{align*}
(\R_+\times Y_H)\ltimes \R^{1,2},\;\;\;\;\;\;\;\; Y_H\ltimes \R^{1,2},\;\;\;\;\;\;\;\;\exp\big(\R(a+\Y_H))\ltimes \R^{1,2}, \;where \;a\in \R^*.
\end{align*}
\end{itemize}
\subsubsection*{\underline{Orbits induced by subgroups with a Lorentzian plane as the translation part}}\label{subsec.3.2.2}
Now, we describe the $2$-dimensional orbits induced by a connected Lie subgroup $G\subset \Conf_\circ(\En^{1,2})$ which its translation part is a Lorentzian plane. These groups have been listed in Table \ref{table2}.

 Observe that, the translation part $T(G)$ acts on each vertex-less photon in the lightcone $L(p)$ transitively, since the action of a timelike plane does not preserve any degenerate affine plane in $\En^{1,2}$. Also, $T(G)$ induces a codimension $1$ foliation $\mathcal{F}_{T(G)}$ in $\En^{1,2}$ on which the leaves are Lorentzian affine planes. Also, the linear isometry projection $P_{li}(G)$ is either trivial or it is a $1$-parameter hyperbolic subgroup of $SO_\circ(1,2)$. In the first case, $G$ preserves every photon in the lightcone $L(p)$, and so, it admits no $2$-dimensional orbit in $L(p)$. In the later, $G$ preserves two distinct photons $\phi,\psi\subset L(p)$ and acts on the both connected components of $L(p)\setminus (\phi\cup \psi)$ transitively. In the following we discuss on the $2$-dimensional $G$-orbits in the Minkowski patch $Mink(p)\approx\En^{1,2}$.
 
\begin{itemize}
 \item $G=T(G)=\R e_1\oplus \R e_2$. The $2$-dimensional orbits induced by $G$ in $\Ein^{1,2}$ are the leaves of the leaves of the foliation $\mathcal{F}_{T(G)}$.

\item $G=\R^*_+\ltimes (\R e_1\oplus \R e_2)$. It can be easily seen that the leaf $\mathcal{F}_{T(G)}(o)\subset \En^{1,2}$ is the only $2$-dimensional $G$-orbit in $\Ein^{1,2}$.
\item $G=Y_H\ltimes (\R e_1\oplus \R e_2)$. This group preserves the leaves of the foliation $\mathcal{F}_{T(G)}$. Henceforth, the $2$-dimensional orbits induced by $G$ in $\En^{1,2}$ are the leaves of $\mathcal{F}_{T(G)}$.

\item $G=(\R_+^*\times Y_H)\ltimes (\R e_1\oplus \R e_2)$ or $G=\exp\big(\R(a+\Y_H)\big)\ltimes (\R e_1\oplus \R e_2)$, $a\in \R^*$. Simple computations show that the leaf $\mathcal{F}_{T(G)}(o)$ is the only $2$-dimensional $G$-orbit in the Minkowski patch $ \En^{1,2}$.
\item $G=\exp\big(\R(\Y_H+e_3)\big)\ltimes (\R e_1\oplus \R e_2\big)$. For an arbitrary point $q=(x,y,z)\in \En^{1,2}$, the vector tangent to the orbit $G(q)$ at $q$ induced by the $1$-parameter subgroup $\exp\big(\R(\Y_H+e_3)\big)$ is $v=(y,x,1)$. The set $\{e_1,e_2,v\}$ is a basis for the tangent space $T_qG(q)$. This implies that $G$ acts on $\En^{1,2}$ transitively. Thus, the connected components of $L(p)\setminus (\phi\cup \psi)$ are the only $2$-dimensional orbits induced by $G$ in Einstein universe $\Ein^{1,2}$.
 \end{itemize}
 \begin{remark}\label{Time}
Observe that by the action of $G=(\R_+^*\times Y_H)\ltimes (\R e_1\oplus \R e_2)$ the union of the leaf $\mathcal{F}_{T(G)}(o)$ (which is a Lorentzian affine $2$-plane) and the photons $\phi$ and $\psi$ is $G$-invariant. Actually, $\mathcal{F}_{T(G)}(o)\cup \phi \cup \psi$ is an Einstein hypersphere (see Definition \ref{def.1.7.17} and Remark \ref{rem.ein}). Moreover, $G$ is the unique (up to conjugacy) maximal connected Lie subgroup in $SO_\circ(2,3)$ which preserves an Einstein hypersphere and admits a fixed point on it.
 \end{remark}
 \subsubsection*{\underline{Orbits induced by the subgroups with a spacelike plane as the translation part}}\label{subsec.3.2.3}
Suppose that $G\subset \Conf_\circ(\En^{1,2})$ is a connected Lie subgroup which its translation part $T(G)$ is a spacelike plane. These groups have been listed in Table \ref{table3}.

The translation part $T(G)$ acts on each vertex-less photon in the lightcone $L(p)$ transitively, since the action of a spacelike plane preserves no degenerate affine plane in $\En^{1,2}$. Also, $T(G)$ induces a codimension $1$ foliation $\mathcal{F}_{T(G)}$ in $\En^{1,2}$ on which the leaves are affine spacelike planes. Moreover, the linear isometry projection $P_{li}(G)$ is either trivial or it is a $1$-parameter elliptic subgroup of $SO_\circ(1,2)$. In the first case, $G$ preserves every vertex-less photon in $L(p)$. In the later, $G$ acts on the vertex-less lightcone $L(\hat{p})$ transitively. Now, we describe the $2$-dimensional $G$-orbits in the Minkowski patch $Mink(p)\approx\En^{1,2}$.
 \begin{itemize}
\item $G=T(G)=\R e_2\oplus \R e_3$. The $2$-dimensional orbits induced by $G$ in $\Ein^{1,2}$ are the leaves of the foliation $\mathcal{F}_{T(G)}$ in the Minkowski patch $\En^{1,2}$.
 
\item $G=\R_+^*\ltimes (\R e_2\oplus \R e_3)$. One can see, $G$ admits a unique $2$-dimensional orbit in $\Ein^{1,2}$, namely, the leaf $\mathcal{F}_{T(G)}(o)$.

\item $G=Y_E\ltimes (\R e_2\oplus \R e_3)$. This group preserves the leaves of the foliation $\mathcal{F}_{T(G)}$. Therefore, the leaves of $\mathcal{F}_{T(G)}$ are the $2$-dimensional $G$-orbits in $\En^{1,2}$.

\item $G=(\R_+^*\times Y_E)\ltimes (\R e_2\oplus \R e_3)$ or $G=\exp\big(\R(a+\Y_E)\big)\ltimes (\R e_2\oplus \R e_3)$, $a\in \R^*$. It is not hard to see that the leaf $\mathcal{F}_{T(G)}(o)$ is the unique $2$-dimensional $G$-orbit in the Minkowski space $\En^{1,2}$.

\item $G=\exp\big(\R(\Y_E+e_1)\big)\ltimes (\R e_2\oplus \R e_3)$. For an arbitrary point $q=(x,y,z)\in \En^{1,2}$, the vector tangent to the orbit $G(q)$ at $q$ induced by the $1$-parameter subgroup $\exp\big(\R(\Y_E+e_1)\big)$ is $v=(1,z,-y)$. The set $\{e_2,e_3,v\}$ is a basis for the tangent space $T_qG(q)$. This implies that $G$ acts on $\En^{1,2}$, transitively. Hence, the connected components of $L(p)\setminus (\phi\cup \psi)$ are the only $2$-dimensional orbits induced by $G$ in Einstein universe $\Ein^{1,2}$.
 \end{itemize}
  \begin{remark}\label{Space}
 Observe that by the action of $G=(\R_+^*\times Y_E)\ltimes (\R e_2\oplus \R e_3)$ the union of the leaf $\mathcal{F}_{T(G)}(o)$ (which is a spacelike affine $2$-plane) and the vertex $\{p\}$ is $G$-invariant. Actually, $\mathcal{F}_{T(G)}(o)\cup \{p\}$ is a spacelike hypersphere (Definition \ref{def.1.7.14}). Furthermore, $G$ is the unique (up to conjugacy) maximal connected Lie subgroup in $SO_\circ(2,3)$ which preserves a spacelike hypersphere and admits a fixed point on it.
 \end{remark}
\subsubsection*{\underline{Orbits induced by the subgroups with a degenerate plane as the translation part}}\label{subsec.3.2.4}
Assume that $G\subset \Conf_\circ(\En^{1,2})$ is a connected Lie subgroup which its translation part $T(G)$ is a degenerate plane, i.e., $G$ belongs to Table \ref{table4}.

The translation part $T(G)=\Pi_\phi$ preserves the leaves of the foliation $\mathcal{F}_{\Pi_\phi}$. Hence, it acts on the corresponding photon $\phi$ in $L(p)$ trivially. Also, it acts on all the vertex-less photons in $L(p)$ different from $\hat{\phi}$ transitively. Furthermore, the linear isometry projection $P_{li}(G)$ is either trivial or it is a subgroup of $\Aff\subset SO_\circ(1,2)$, up to conjugacy.

\textbf{Case I:} Assume that the linear isometry projection $P_{li}(G)$ is trivial. Then $G$ admits no $2$-dimensional orbit in the lightcone.
\begin{itemize}
\item $G=\Pi_\phi=\R(e_1+e_2)\oplus \R e_3$. The codimension one orbits induced by $G$ in $\Ein^{1,2}$ are the leaves of the foliation $\mathcal{F}_{\Pi_\phi}$
\item $G=\R_+^*\ltimes \Pi_\phi$. This group admits a unique $2$-dimensional orbit in Einstein universe $\Ein^{1,2}$, namely, the leaf $\mathcal{F}_{\Pi_\phi}(o)$.
\end{itemize}

\textbf{Case II:} Assume that the linear isometry projection $P_{li}(G)$ contains a parabolic element. Then $G$ preserves a unique photon $\phi\subset L(p)$ and acts on the degenerate surface $L(p)\setminus \phi$ transitively. So, in the following, we only describe the $2$-dimensional orbits in the Minkowski patch $Mink(p)\approx\En^{1,2}$.
\begin{itemize}
\item $G=\Aff\ltimes \Pi_\phi$ or $G=(\R_+^*\times \Aff)\ltimes \Pi_\phi$ or $G=\exp\big((\R(a+\Y_H)+\R\Y_P)\big)\ltimes \Pi_\phi$, $a\in \R^*\setminus \{1\}$ or $G=(\R_+^*\times Y_P)\ltimes \Pi_\phi$ or $G=\exp\big(\R(a+\Y_P)\big)\ltimes \Pi_\phi$, $a\in \R^*$. It is easily seen that $G$ admits a unique $2$-dimensional orbit in $\En^{1,2}$, namely, the degenerate leaf $\mathcal{F}_{\Pi_\phi}(o)$.

\item $G=\exp\big((\R(1+\Y_H)+\R\Y_P)\big)\ltimes \Pi_\phi$. For an arbitrary point $q=(x,y,z)\in \En^{1,2}$, the vectors tangent to the orbit $G(q)$ at $q$ induced by the $1$-parameter subgroups $\exp\big(\R(1+\Y_H)\big)$ and $Y_P$ are $v=(x+y,x+y,z)$ and $w=(z,z,x-y)$, respectively. Obviously, $v,w\in \Pi_\phi\leq T_qG(q)$. Therefore, $G$ preserves the leaves of the foliation $\mathcal{F}_{\Pi_\phi}$. Henceforth, the $2$-dimensional orbits induced by $G$ in $\En^{1,2}$ are the leaves of $\mathcal{F}_{\Pi_\phi}$.
\end{itemize}
\begin{remark}\label{rem.ker}
The subgroup $\exp\big((\R(1+\Y_H)+\R\Y_P)\big)\ltimes \Pi_\phi$ is the identity component of the intersection of the stabilizers of the points in the photon $\phi$. In other words, up to conjugacy, it is the maximal connected Lie subgroup of $SO_\circ(2,3)$ which acts on a photon in Einstein universe $\Ein^{1,2}$ trivially. We denote this group by the special symbol $\mathsf{K}$.
\end{remark}
\begin{itemize}
\item $G=\exp\big(\R(2+\Y_H)+\R(\Y_P+e_1)\big)\ltimes \Pi_\phi$ or $G=\exp\big(\R(1+\Y_H +e_1)+\R\Y_P\big)\ltimes \Pi_\phi$ or $G=\exp\big(\R(\Y_P+e_1)\big)\ltimes  \Pi_\phi$. By some computations, one can see that $G$ acts on the Minkowski patch $\En^{1,2}$ transitively. Hence, the degenerate surface $L(p)\setminus \phi$ is the only $2$-dimensional $G$-orbit in $\Ein^{1,2}$.
\item $G=Y_P\ltimes\Pi_\phi$. This group preserves the leaves of the foliation $\mathcal{F}_{\Pi_\phi}$, since it is a subgroup of $\mathsf{K}$. Therefore, the leaves of the of $\mathcal{F}_{\Pi_\phi}$ are the $2$-dimensional $G$-orbits in $\En^{1,2}$.
\end{itemize}

\textbf{Case III:} Assume that the linear isometry projection is a $1$-parameter hyperbolic subgroup of $SO_\circ(1,2)$. Then $G$ preserves two distinct photons $\phi$ and $\psi$ in $L(p)$, and acts on the both connected components of $L(p)\setminus (\phi\cup \psi)$ transitively. In the following we discuss on the $2$-dimensional $G$-orbits in the Minkowski patch $Mink(p)\setminus \En^{1,2}$.
\begin{itemize}
\item $G=Y_H\ltimes \Pi_\phi$. The $1$-parameter subgroup $Y_H$ preserves the leaf $\mathcal{F}_{\Pi_\phi}(o)$, hence, it is a $G$-orbit. For an arbitrary point $q=(x,y,z)\in \En^{1,2}$, the vector tangent to the orbit $G(q)$ at $q$ induced by $Y_H$ is $v=(y,x,z)$. The set $\{(e_1+e_2),e_3,v\}\subset T_qG(q)$ is a basis if and only if $x\neq y$ if and only if $q\notin \mathcal{F}_{\Pi_\phi}(o)$. Therefore, the leaf $\mathcal{F}_{\Pi_\phi}(o)$ is the only $2$-dimensional $G$-orbit in $\En^{1,2}$.

\item $G=(\R_+^*\times Y_H)\ltimes \Pi_\phi$ or $G=\exp\big(\R(a+\Y_H)\big)\ltimes \Pi_\phi$, $a\in \R^*\setminus \{1\}$. It can be easily seen that $G$ admits a unique $2$-dimensional orbit in $\En^{1,2}$, namely, the degenerate leaf $\mathcal{F}_{\Pi_\phi}(o)$.

\item $G=\exp\big(\R(1+\Y_H)\big)\ltimes \Pi_\phi$. This group preserves the leaves of the foliation $\mathcal{F}_{\Pi_\phi}$, since it is a subgroup of $\mathsf{K}$ (Remark \ref{rem.ker}). Hence, the leaves of $\mathcal{F}_{\Pi_\phi}$ are the $2$-dimensional $G$-orbits in $\En^{1,2}$.

\item $G=\exp\big(\R(1+\Y_H+e_1)\big)\ltimes \Pi_\phi$. For an arbitrary point $q=(x,y,z)\in \En^{1,2}$, the vector tangent to the orbit $G(q)$ at $q$ induced by $\exp\big(\R(1+\Y_H+e_1)\big)$ is $v=(x+y+1,x+y,z)$. The set  $\{(e_1+e_2),e_3,v\}\subset T_qG(q)$ is a basis, and so, $G$ acts on $\En^{1,2}$ transitively. Hence the connected components of $L(p)\setminus (\phi\cup \psi)$ are the only $2$-dimensional $G$-orbits in $\Ein^{1,2}$.
\end{itemize}

\subsubsection*{\underline{Orbits induced by the subgroups with a timelike line as the translation part}}\label{subsec.3.2.5} 
Let $G\subset \Conf_\circ(\En^{1,2})$ be a connected Lie subgroup on which its translation part $T(G)$ is a  timelike line, i.e., $G$ belongs to Table \ref{table5}.

The translation part $T(G)=\R e_1$ admits a $1$-dimensional foliation $\mathcal{F}_{\R e_1}$ in $\En^{1,2}$ on which the leaves are affine timelike lines. On the other hand, $\R e_1$ preserves no degenerate affine plane in $\En^{1,2}$. Hence, $\R e_1$ acts on each vertex-less photon in the lightcone $L(p)$ transitively. In this case, the linear isometry projection $P_{li}(G)$ is either trivial or it is a $1$-parameter elliptic subgroup of $SO_\circ(1,2)$. In the first case, the orbits induced by $G$ in the lightcone $L(p)$ are the vertex-less photons. In the later, the vertex-less lightcone $L(\hat{p})$ is a $G$-orbit. Observe that every $2$-dimensional $G$-orbit in the Minkowski patch $Mink(p)\approx\En^{1,2}$ is Lorentzian, since the timelike vector $e_1$ is tangent to it.

\begin{itemize}
\item $G=\R_+^*\ltimes \R e_1$. The homothety factor $\R_+^*$ preserves the leaf $\mathcal{F}_{\R e_1}(o)$. Hence, $\mathcal{F}_{\R e_1}(o)$ is a $G$-orbit. Also, $G$ preserves every affine Lorentzian plane in $\En^{1,2}$ containing $\mathcal{F}_{\R e_1}(o)$. On the other hand, $G$ acts on $\En^{1,2}\setminus \mathcal{F}_{\R e_1}(o)$ freely. It follows that, the $2$-dimensional orbits induced by $G$ in $\Ein^{1,2}$ are the affine Lorentzian half-plane in $\En^{1,2}$.

\item $G=Y_E\times \R e_1$. The timelike affine line $\mathcal{F}_{\R e_1}(o)$ is a $G$-orbit. Furthermore, $G$ acts on $\En^{1,2}\setminus \mathcal{F}_{\R e_1}(o)$ freely. Hence, the $2$-dimensional orbits induced by $G$ in Minkowski patch $\Ein^{1,2}$ are Lorentzian cylinders.

\item $G=(\R_+^*\times Y_E)\ltimes \R e_1$. For an arbitrary point $q=(x,y,z)\in \En^{1,2}$, the vectors tangent to the orbit $G(q)$ at $q$ induced by the $1$-parameter subgroups $\R_+^*$ and $Y_E$ are $v=(x,y,z)$ and $w=(0,z,-y)$, respectively. The set $\{e_1,v,w\}\subset T_qG(q)$ is a basis if and only if $y,z\neq 0$ if and only if $q\notin \mathcal{F}_{\R e_1}(o)$. This implies that $G$ admits a unique $2$-dimensional orbit in $\Ein^{1,2}$ which is the vertex-less lightcone $L(\hat{p})$.

\item $G=\exp\big(\R(a+ \Y_E)\big)\ltimes\R e_1$, $a\in \R^*$. This group preserves the leaf $\mathcal{F}_{\R e_1}(o)$, and acts on $\En^{1,2}\setminus \mathcal{F}_{\R e_1}(o)$ freely. Hence, $G$ admits a codimension $1$ foliation on $\En^{1,2}\setminus \mathcal{F}_{\R e_1}(o)$ on which every leaf is a Lorentzian surface.
\end{itemize}
\subsubsection*{\underline{Orbits induced by the subgroups with a spacelike line as the translation part}}\label{subsec.3.2.6} 
Suppose that $G\subset \Conf_\circ(\En^{1,2})$ is a connected Lie subgroup with a spacelike line as the translation part. These groups have been listed in Table \ref{table6}.

The translation part $T(G)=\R e_3$ admits a $1$-dimensional foliation $\mathcal{F}_{\R e_3}$ in $\En^{1,2}$ on which the leaves are affine spacelike lines parallel to $\R e_3$. In the one hand, the translation part fixes two photons $\phi,\psi\subset L(p)$, pointwisely, since $\R e_3$ is contained in the both lightlike planes $\Pi_\phi=\R(e_1+e_2)\oplus \R e_3$ and $\Pi_\psi=\R(e_1-e_2)\oplus \R e_3$. On the other hand, the linear projection $P_l(G)$ preserves both foliations $\mathcal{F}_{\Pi_\phi}$ and $\mathcal{F}_{\Pi_\psi}$. Hence, the photons $\phi$ and $\psi$ are invariant by $G$.  Also, $\R e_3$ acts on all the vertex-less  photons in the lightcone $L(p)$ different form $\hat{\phi}$ and $\hat{\psi}$ transitively. In addition, the linear isometry projection is either trivial or it is a $1$-parameter hyperbolic subgroup of $SO_\circ(1,2)$. In the first case, every photon in $L(p)$ is  $G$-invariant. In the later, $G$ acts on the both connected components of $L(p)\setminus (\phi\cup \psi)$ transitively. Therefore, in the following, we only consider the $2$-dimensional $G$-orbits in the Minkowski patch $Mink(p)\approx\En^{1,2}$.

\begin{itemize}
\item $G=\R_+^*\ltimes \R e_3$. The homothety factor $\R_+^*$ preserves the leaf $\mathcal{F}_{\R e_3}(o)$. Hence, $\mathcal{F}_{\R e_3}(o)$ is a $G$-orbit. 
For an arbitrary point $q=(x,y,z)\in \Ein^{1,2}$, the vector tangent to the orbit $G(q)$ at $q$ induced by homothety factor $\R_+^*$ is $v=(x,y,z)$. Observe that the set $\{v,e_3\}\subset T_qG(q)$ is a basis if and only if $x\neq 0$ or $y\neq 0$ if and only if $q\notin \mathcal{F}_{\R e_3}(o)$. The group $G$ induces spacelike, Lorentzian, and degenerate $2$-dimensional orbits in $\En^{1,2}$, since the orthogonal space of the spacelike tangent vector $e_3$ in the tangent space $T_qG(q)$ can be a spacelike, timelike, or lightlike line.
\item $G=Y_H\times \R e_3$. For an arbitrary point $q=(x,y,z)\in \En^{1,2}$, the vector tangent to the orbit $G(q)$ at $q$ induced by the $1$-parameter subgroup $Y_H$ is $v=(y,x,0)$. The set $\{e_3,v\}\subset T_qG(q)$ is a basis if and only if $x\neq 0$ or $y\neq 0$ if and only if $q\notin \mathcal{F}_{\R e_3}(o)$. The tangent vector $v$ is orthogonal to the spacelike vector $e_3$. Hence, $G$ admits spacelike, Lorentzian, and degenerate $2$-dimensional orbits in $\En^{1,2}$, since the vector $v$ can be spacelike, timelike, or lightlike.

\item $G=(\R_+^*\times Y_H)\ltimes \R e_3$. This group preserves the leaves $\mathcal{F}_{\R e_3}(o)$, $\mathcal{F}_{\Pi_\phi}(o)$, and $\mathcal{F}_{\Pi_\psi}(o)$. Also, $G$ acts on the four connected components of $(\mathcal{F}_{\Pi_\phi}(o)\cup\mathcal{F}_{\Pi_\psi}(o))\setminus \mathcal{F}_{\R e_3}(o)$ transitively, since its subgroup $Y_H\times \R e_3$ does. For an arbitrary point $q=(x,y,z)\in \En^{1,2}$, the vectors tangent to the orbit $G(q)$ at $q$ induced by the $1$-parameter subgroups $Y_H$ and $\R_+^*$ are $v=(y,x,0)$ and $w=(x,y,z)$. The set $\{e_3,v,w\}\subset T_qG(q)$ is a basis if and only if $x\neq \pm y$ if and only if $q\notin (\mathcal{F}_{\Pi_\phi}(o)\cup\mathcal{F}_{\Pi_\psi}(o))$. This implies that $G$ acts on the four connected components of $\En^{1,2}\setminus (\mathcal{F}_{\Pi_\phi}(o)\cup\mathcal{F}_{\Pi_\psi}(o))$ transitively. It follows that the only $2$-dimensional $G$-orbits in the Minkowski patch $Mink(p)\approx\En^{1,2}$ are the connected components of $(\mathcal{F}_{\Pi_\phi}(o)\cup\mathcal{F}_{\Pi_\psi}(o))\setminus \mathcal{F}_{\R e_3}(o)$ which are affine degenerate half-planes. 

\item $G=\exp\big(\R(a+\Y_H)\big)\ltimes\R e_3$, $a\in \R\setminus \{-1,0,1\}$. For an arbitrary point $q=(x,y,z)\in \En^{1,2}$, the vector tangent to the orbit $G(q)$ at $q$ induced by the $1$-parameter subgroup $\exp\big(\R(a+\Y_H)\big)$ is $v=(ax+y,x+ay,az)$. The set $\{e_3,v\}\subset T_qG_a(q)$ is a basis if and only if $x\neq 0$ or $y\neq 0$ if and only if $q\notin \mathcal{F}_{\R e_3}(o)$. The tangent vector $v-aze_3$ is orthogonal to the spacelike vector $e_3$. It follows that $G$ admits spacelike, Lorentzian, and degenerate $2$-dimensional orbits in $\En^{1,2}$, since the vector $v-aze_3$ can be spacelike, timelike, or lightlike.

\item $G=\exp\big(\R(1+\Y_H)\big)\ltimes \R e_3$. This group preserves the leaves of the foliation $\mathcal{F}_{\Pi_\phi}$, since it is a subgroup of $\mathsf{K}$ (Remark \ref{rem.ker}). Also, the spacelike affine line $\mathcal{F}_{\R e_3}(o)$ is a $G$-orbit. 
 For an arbitrary point $q=(x,y,z)\in \En^{1,2}$, the vector tangent to the orbit $G(q)$ at $q$ induced by the $1$-parameter subgroup $\exp\big(\R(1+\Y_H)\big)$ is $v=(x+y,x+y,z)$. The set $\{e_3,v\}\subset T_qG(q)$ is a basis if and only if $x\neq -y$ if and only if $q\notin \mathcal{F}_{\Pi_\psi}(o)$. It follows that, every $2$-dimensional $G$-orbit in $\En^{1,2}$ is a degenerate surface included in a leaf of $\mathcal{F}_{\Pi_\phi}$.
 
\item $G=\exp\big(\R(-1+\Y_H)\big)\ltimes \R e_3$. For an arbitrary point $q=(x,y,z)\in \En^{1,2}$, the vectors tangent to the orbit $G(q)$ at $q$ induced by the $1$-parameter subgroups $\exp\big(\R(-1+\Y_H)\big)$ is $v=(-x+y,x-y,z)$. The group $G$ preserves the leaves of the foliation $\mathcal{F}_{\Pi_\psi}$, since $e_3,v\in \Pi_\psi$. This implies that $G$ acts on the photon $\psi$ trivially. Therefore $G$ is a subgroup of $\mathsf{K}$ (Remark \ref{rem.ker}) up to conjugacy. 

 \item $G=\exp\big(\R(1+\Y_H+e_1)\big)\ltimes\R e_3$. For an arbitrary point $q=(x,y,z)\in \En^{1,2}$, the vector tangent to the orbit $G(q)$ induced by the $1$-parameter subgroup $\exp\big(\R(1+\Y_H+e_1)\big)$ is $v=(x+y+1,x+y,z)$. The set $\{e_3,v\}\subset T_qG(q)$ is a basis, hence, all the orbits induced by $G$ in $\En^{1,2}$ are $2$-dimensional. On the other hand, $G$ admits spacelike, Lorentzian, and degenerate orbits in $\En^{1,2}$, since the tangent vector $v-ze_3$, which is orthogonal to the spacelike vector $e_3$, can be spacelike, timelike, or lightlike.
 
   \item $G=\exp\big(\R(-1+\Y_H+e_1)\big)\ltimes\R e_3$. For an arbitrary point $q=(x,y,z)\in \En^{1,2}$, the vector tangent to the orbit $G(q)$ at $q$ induced by the $1$-parameter subgroup $\exp\big(\R(-1+\Y_H+e_1)\big)$ is $v=(-x+y+1,x-y,z)$. By Lemma \ref{lem.1.3.4}, the action of $G$ on $\Ein^{1,2}$ is orbitally-equivalent to the action of $\exp\big(\R(1+\Y_H+e_1)\big)\ltimes\R e_3$ via the element in $O(1.2)$ which maps $(x,y,z)\in \En^{1,2}$ to $(-x,y,z)$.
\end{itemize}
\subsubsection*{\underline{Orbits induced by the subgroups with a lightlike line as the translation part}}\label{subsec.3.2.7} 
Assume that $G\subset \Conf_\circ(\En^{1,2})$ is a connected Lie subgroup with a lightlike line as the translation part, i.e., $G$ is an element of Table \ref{table7}. 
%

The translation part $T(G)=\el$ induces a $1$-dimensional foliation $\mathcal{F}_\el$ in $\En^{1,2}$ on which the leaves are affine lightlike lines parallel to $\el$. Observe that, the translation part preserves the leaves of the foliation $\mathcal{F}_{\Pi_\phi}$, where $\Pi_\phi=\el^\perp=\R(e_1+e_2)\oplus \R e_3$. Hence, $T(G)$ acts on the photon $\phi$ trivially, and acts on all the vertex-less photons in $L(p)$ different form $\hat{\phi}$ transitively. On the other hand, the photon $\phi$ is invariant by $G$. In this case, the linear isometry projection $P_{li}(G)$ is either trivial or it is a subgroup of the affine group $\Aff\subset SO_\circ(1,2)$.

\textbf{Case I:} \textit{The linear isometry projection is trivial}. In this case $G=\R_+^*\ltimes \el$. Observe that $G$ admits no $2$-dimensional orbit in the lightcone $L(p)$. The homothety factor $\R_+^*$ preserves the leaf $\mathcal{F}_{\el}(o)$. Hence, $\mathcal{F}_{\el}(o)$ is a $G$-orbit.  For an arbitrary point $q=(x,y,z)\in \En^{1,2}$ the vector tangent to the orbit $G(q)$ at $q$ induced by $\R_+^*$ is $v=(x,y,z)$. The set $\{(e_1+e_2).v\}\subset T_qG(q)$ is a basis if and only if $z\neq 0$ or $x\neq y$ if and only if $q\notin \mathcal{F}_\el(o)$. It can be easily seen that a $2$-dimensional $G$-orbit in $\En^{1,2}$ is either Lorentzian or degenerate, since the orthogonal space of the lightlike vector $e_1+e_2$ in the tangent space is either $\R(e_1+e_2)$ or entire $T_qG(q)$.
 
\textbf{Case II:} \textit{The linear isometry projection $P_{li}(G)$ contains a parabolic element.} In this case, $G$ preserves a unique photon $\phi\subset L(p)$, and acts on the degenerate surface $L(p)\setminus \phi$ transitively. In the following we discuss on the $2$-dimensional orbits in the Minkowski patch $Mink(p)\approx\En^{1,2}$.
\begin{itemize}
\item $G=(\R^*_+\times \Aff)\ltimes \el$ or $G= \exp\big((\R(a+\Y_H)+\R\Y_P)\big)\ltimes \el$, $a\in \R^*\setminus \{1\}$ or $G=(\R_+^*\times Y_P)\ltimes \el$. This group preserves the leaves $\mathcal{F}_\el(o)$ and $\mathcal{F}_{\Pi_\phi}(o)$. For an arbitrary point $q=(x,y,z)\in \En^{1,2}$ the vectors tangent to the orbit $G(q)$ at $q$ induced by the $1$-parameter subgroups $\R^*_+$ and $Y_P$ are $v=(x,y,z)$ and $w=(z,z,x-y)$, respectively. Observe that $G$ acts on the both connected components of $\En^{1,2}\setminus \mathcal{F}_{\Pi_\phi}(o)$ transitively, since the set $\{e_1+e_2,v,w\}\subset T_qG(q)$ is a basis. Also, $G$ acts on the both connected components of $\mathcal{F}_{\Pi_\phi}(o)\setminus \mathcal{F}_\el(o)$ transitively, since $\{e_1+e_2,v\}$ is a basis for the tangent space $TqG(q)$.
.
\item $G=\Aff\ltimes \el$. For an arbitrary point $q=(x,y,z)\in \En^{1,2}$ the vectors tangent to the orbit $G(q)$ at $q$ induced by the $1$-parameter subgroups $Y_P$ and $Y_H$ are $v=(z,z,x-y)$ and $w=(y,x,0)$, respectively. Observe that for all $q\in \mathcal{F}_{\Pi_\phi}(o)$, the leaf $\mathcal{F}_{\el}(q)$ is invariant by $G$. On the other hand, the set $\{e_1+e_2,v,w\}\subset T_qG(q)$ is a basis for all $q\in \En^{1,2}\setminus \mathcal{F}_{\Pi_\phi}(o)$. This implies that $G$ admits no $2$-dimensional orbit in $\En^{1,2}$, and so, the connected components of $L(p)\setminus (\phi\cup \psi)$ are the only $2$-dimensional orbits in Einstein universe $\Ein^{1,2}$.

\item $G= \exp\big((\R(1+\Y_H)+\R\Y_P)\big)\ltimes \el$. Since $G$ is a subgroup of $\mathsf{K}$ (Remark \ref{rem.ker}), it preserves the leaves of the foliation $\mathcal{F}_{\Pi_\phi}$. 
It is clear that $G$ preserves the leaf $\mathcal{F}_\el(o)$. For an arbitrary point $q=(x,y,z)\in \En^{1,2}$ the vectors tangent to the orbit $G(q)$ at $q$ induced by the $1$-parameter subgroups $\exp\big(\R(1+Y_H)\big)$ and $Y_P=\exp\big(\R\Y_P\big)$ are $v=(x+y,x+y,z)$ and $w=(z,z,x-y)$, respectively. The set $\{e_1+e_2,v,w\}\subset T_qG(q)$ generates a $2$-dimensional vector space if and only if $z\neq 0$ or $x\neq y$ if and only if $q\notin\mathcal{F}_\el(o)$. Hence, $G$ admits $2$-dimensional orbits in $\En^{1,2}$ which are degenerate surfaces included in a leaf of $\mathcal{F}_{\Pi_\phi}$.

\item $G=\exp\big(\R(2+\Y_H)+\R(\Y_P+e_1)\big)\ltimes \el$. For an arbitrary point $q=(x,y,z)\in \En^{1,2}$ the vectors tangent to orbit $G(q)$ at $q$ induced by subgroups $\exp\big(\R(2+\Y_H)\big)$ and $\exp\big(\R(\Y_P+e_1)\big)$ are $v=(2x+y,x+2y,2z)$ and $w=(z+1,z,x-y)$, respectively. The set $\{e_1+e_2,v,w\}\subset T_qG(q)$ is a basis if and only if $z\neq (x-y)^2/2$. The set of points in $\En^{1,2}$ with $z=(x-y)^2/2$ is a connected Lorentzian surface $S$, and one can see that $G$ acts on it transitively (in fact it is the orbit induced by the subgroup $\exp\big(\R(\Y_P+e_1)\big)\ltimes\el$ at $o$). Hence, the surface $S$ is the only $2$-dimensional $G$-orbit in $\En^{1,2}$.

\item $G=\exp\big((\R(\Y_H+e_3)+\R\Y_P)\big)\ltimes\el$. Since $G$ is a subgroup of $\Aff\ltimes \Pi_\phi$, it preserves the leaf $\mathcal{F}_{\Pi_\phi}(o)$. 
For an arbitrary point $q=(x,y,z)\in \En^{1,2}$ the vectors tangent to orbit $G(q)$ at $q$ induced by subgroups $H$ and $Y_P$ are $v=(y,x,1)$ and $w=(z,z,x-y)$, respectively. The set $\{e_1+e_2,v,w\}\subset T_qG(q)$ is a basis if and only if $x\neq y$ if and only if $q\notin \mathcal{F}_{\Pi_\phi}(o)$.  On the other hand, for a point $q\in \mathcal{F}_{\Pi_\phi}(o)$ the set $\{e_1+e_2,v\}\subset T_qG(q)$ is a basis, and so, the leaf $\mathcal{F}_{\Pi_\phi}(o)$ is the only $2$-dimensional $G$-orbit in $\En^{1,2}$.

\item $G=Y_P\times \el$. Since $G$ is a subgroup of $\mathsf{K}$ (Remark \ref{rem.ker}), it  preserves the leaves of the foliation $\mathcal{F}_{\Pi_\phi}$. For an arbitrary point $q=(x,y,z)\in \En^{1,2}$, the vector tangent to the orbit $G(q)$ at $q$ induced by $Y_P$ is $v=(z,z,x-y)$. The set $\{e_1+e_2,v\}$ is a basis if and only if $x\neq y$, which describes the leaf $\mathcal{F}_{\Pi_\phi}(o)$. Therefore, $G$ acts on all the leaves of $\mathcal{F}_{\Pi_\phi}$ different form $\mathcal{F}_{\Pi_\phi}(o)$ transitively. 

\item $G=\exp\big(\R(a+\Y_P)\big)\ltimes \el$, $a\in \R^*$. The leaf $\mathcal{F}_{\el}(o)$ is a $G$-orbit. 
For an arbitrary point $q=(x,y,z)\in \En^{1,2}$, the vector tangent to the orbit $G(q)$ at $q$ induced by the $\exp\big(\R(a+\Y_P)\big)$ is $v=(ax+z,ay+z,x-y+az)$. The set $\{e_1+e_2,v\}\subset T_qG(q)$ is a basis if and only if $z\neq 0$ or $x\neq y$ if and only if $q\notin \mathcal{F}_{\el}(o)$. The $2$-dimensional $G$-orbits in $\En^{1,2}$ are either Lorentzian or degenerate, since the orthogonal space of the lightlike vector $e_1+e_2$ in the tangent space $T_qG(q)$ is either $\R(e_1+e_2)$ or entire $T_qG(q)$. 

  \item If $G=\exp\big(\R(\Y_P+e_1)\big)\ltimes \el$. For an arbitrary point $q=(x,y,z)\in \En^{1,2}$, the vector tangent to the orbit $G(p)$ at $q$ induced by the $1$-parameter subgroup $\exp\big(\R(\Y_P+e_1)\big)$ is $v=(z+1,z,x-y)$. It is clear that the set $\{e_1+e_2,v\}\subset T_qG(q)$ is a basis. Thus all the orbits induced by $G$ in $\En^{1,2}$ are $2$-dimensional. On the other hand, all the $G$-orbits in $\En^{1,2}$ are Lorentzian, since the orthogonal space of the lightlike vector $e_1+e_2$ in the tangent space $T_qG(q)$ is $\R(e_1+e_2)$.
\end{itemize}

\textbf{Case III:} \textit{The linear isometry projection $P_{li}(G)$ is a $1$-parameter hyperbolic subgroup of $SO_\circ(1,2)$.} In this case, $G$ preserves both the foliations $\mathcal{F}_{\Pi_\phi}$ and $\mathcal{F}_{\Pi_\psi}$, where $\Pi_\psi=\R(e_1-e_2)\oplus \R e_3$, and so, it preserves both the photons $\phi$ and $\psi$ in the lightcone $L(p)$. Also, $G$ acts on the both connected components of $L(p)\setminus (\phi\cup \psi)$ transitively. Now, we describe the $2$-dimensional orbits induced by $G$ in the Minkowski patch $Mink(p)\approx\En^{1,2}$.
\begin{itemize}
\item If $G=(\R_+^*\times Y_H)\ltimes \el$. The leaf $\mathcal{F}_{\el}(o)$ is a $G$-orbit. Also, $G$-preserves the leaf $\mathcal{F}_{\Pi_\phi}(o)$ and the Lorentzian affine plane $L_0=\{z=0\}\subset \En^{1,2}$. For an arbitrary point $q=(x,y,z)\in \En^{1,2}$, the vectors tangent to the orbit $G(q)$ at $q$ induced by the $1$-parameter subgroups $\R_+^*$ and $Y_H$ are $v=(x,y,z)$ and $w=(y,x,0)$, respectively. The set $\{e_1+e_2,v,w\}\subset T_qG(q)$ is a basis if and only if $z\neq 0$ and $x\neq y$ if and only if $q\notin (L_0\cup \mathcal{F}_{\Pi_\phi}(o))$. Thus, $G$ acts on the four connected components of $\En^{1,2}\setminus (L_0\cup \mathcal{F}_{\Pi_\phi}(o))$ transitively. Furthermore,  $G$ acts on each component of the four connected components of $(L_0\cup \mathcal{F}_{\Pi_\phi}(o))\setminus \mathcal{F}_\el(o)$ transitively, which are the only $2$-dimensional $G$-orbits in $\En^{1,2}$. 
 
\item $G=Y_H\ltimes \el$. This group preserves the leaf $\mathcal{F}_{\Pi_\phi}(o)$. For an arbitrary point $q=(x,y,z)\in \En^{1,2}$, the vector tangent to the orbit $G(q)$ at $q$ induced by $Y_H$ is $v=(y,x,0)$. Observe that, the set $\{(e_1+e_2,v\})\subset T_qG(q)$ is a basis if and only if $x\neq y$ if and only if $q\notin \mathcal{F}_{\Pi_\phi}(o)$. On the other hand, every $2$-dimensional $G$-orbit in $\En^{1,2}$ is Lorentzian, since the orthogonal space of the lightlike vector $e_1+e_2$ is $\R(e_1+e_2)$. 

\item $G=\exp\big(\R(a+\Y_H)\big)\ltimes \el$,  $a\in \R^*\setminus \{1\}$. This group preserves the leaf $\mathcal{F}_{\el}(o)$. For an arbitrary point $q=(x,y,z)\in \En^{1,2}$, the vector tangent to the orbit $G(q)$ at $q$ induced by the $1$-parameter subgroup $\exp\big(\R(a+\Y_H)\big)$ is $v=(ax+y,x+ay,az)$. The set $\{e_1+e_2,v\}\subset T_qG(q)$ is a basis if and only if $z\neq 0$ or $x\neq y$ if and only if $p\notin \mathcal{F}_{\el}(o)$. On the other hand, a $2$-dimensional $G$-orbit in $\En^{1,2}$ is either Lorentzian or degenerate, since the orthogonal space of the lightlike tangent vector $e_1+e_2$ is either $\R(e_1+e_2)$ or entire $T_qG(q)$.

\item If $G=\exp\big(\R(1+\Y_H)\big)\ltimes \el$. Since, $G$ is a subgroup of $\mathsf{K}$ (Remark \ref{rem.ker}), it preserves the leaves of the foliation $\mathcal{F}_{\Pi_\phi}$. Also, $G$ preserves the Lorentzian affine plane $L_0=\{z=0\}\subset \En^{1,2}$. For an arbitrary point $q=(x,y,z)\in \En^{1,2}$ the vector tangent to the orbit $G(q)$ at $q$ induced by the $1$-parameter subgroup $\exp\big(\R(1+\Y_H)\big)$ is $v=(x+y,x+y,z)$. The set $\{e_1+e_2,v\}\subset T_qG(q)$ is a basis if and only if $z\neq 0$ if and only if $q\notin L_0$. Thus every $2$-dimensional $G$-orbit in $\En^{1,2}$ is a degenerate surface included in a leaf of $\mathcal{F}_{\Pi_\phi}$.

  \item $G=\exp\big(\R(\Y_H+e_3)\big)\ltimes \el$. For an arbitrary point $q=(x,y,z)\in \En^{1,2}$, the vector tangent to the orbit $G(p)$ induced by the $1$-parameter subgroup $\exp\big(\R(\Y_H+e_3)\big)$ is $v=(y,x,1)$. The set $\{e_1+e_2,v\}\subset T_qG(q)$ is a basis. Thus, all the orbits induced by $G$ in $\En^{1,2}$ are $2$-dimensional. The group $G$ admits Lorentzian and degenerate orbits in $\En^{1,2}$, since the orthogonal space of the lightlike vector $e_1+e_2$ in the tangent space $T_qG(q)$ is either $\R(e_1+e_2)$ or entire $T_qG(q)$.
  
    \item $G=\exp\big(\R(1+\Y_H+e_1)\big)\ltimes \el$. For an arbitrary point $q=(x,y,z)\in \En^{1,2}$, the vector tangent to the orbit $G(q)$ at $q$ induced by the $1$-parameter subgroup $\exp\big(\R(1+\Y_H+e_1)\big)$ is $v=(x+y+1,x+y,z)$. The set $\{e_1+e_2,v\}\subset T_qG(q)$ is a basis. Thus, all the orbits induced by $G$ in $\En^{1,2}$ are $2$-dimensional. Also, all the orbits induced by $G$ in $\En^{1,2}$ are Lorentzian, since the orthogonal space of the lightlike vector $e_1+e_2$ in $T_qG(q)$ is $\el=\R(e_1+e_2)$. 
    
   \item $G=\exp\big(\R(-1+\Y_H+e_1)\oplus_\theta\el\big)$. For an arbitrary point $q=(x,y,z)\in \R^{1,2}$, the vector tangent to the orbit $G(q)$ at $q$ induced by the $1$-parameter subgroup $H=\exp\big(\R(-1+\Y_H+e_1)\big)$ is $v=(-x+y+1,x-y,-z)$. The set $\{e_1+e_2,v\}\subset T_qG(q)$ is a basis if and only if $z\neq 0$ and $x-y\neq 1/2$ if and only if $q\notin \mathcal{F}_\el(1/2,0,0)$. Obviously, the leaf $\mathcal{F}_\el(1/2,0,0)$ is a $G$-orbit. Furthermore, $G$ admits in $\En^{1,2}$ Lorentzian and degenerate orbits, since the orthogonal space of the lightlike tangent vector $e_1+e_2$ in the tangent space $T_qG(q)$ is either $\R(e_1+e_2)$ or entire $T_qG(q)$.
\end{itemize}
\subsubsection*{\underline{Orbits induced by the subgroups with trivial translation part}}\label{subsec.3.2.8}
Finally, assume that $G\subset \Conf_\circ(\En^{1,2})$ is a connected Lie subgroup with trivial translation part $T(G)=\{0\}$. These groups have been listed in Table \ref{table8}.

First, assume that $G$ is a subgroup of $\R_+^*\times SO_\circ(1,2)$. This group fixes the origin $o=(0,0,0)\in \En^{1,2}$ and  preserves the nullcone centered at $o$
$$\mathfrak{N}(o)=\{q=(x,y,z)\in \En^{1,2}\setminus \{o\}:\mathfrak{q}(q)=-x^2+y^2+z^2=0\}.$$
Also, it preserves the three connected components of the complement of the nullcone $\Nu(o)$ in $\En^{1,2}$ which are: the domain $\{\q>0\}$ and the two connected components of the domain $\{\q<0\}$. Also, this group preserves the ideal circle $S_\infty=L(p)\cap L(o)$. Observe that, if $G$ is a subgroup of $SO_\circ(1,2)$, it also preserves the de-Sitter spaces $\dS^{1,1}(r)=\mathfrak{q}^{-1}(r)$, and the hyperbolic planes $\Hn^2(r)$ (the connected components of $\mathfrak{q}^{-1}(-r)$) in $\En^{1,2}$ centered at $o$ with radius $r\in \R^*_+$. Moreover, the group $\R^*_+\times \Aff$ preserves the foliation $\mathcal{F}_{\Pi_\phi}$, and also, preserves the leaves $\mathcal{F}_{\Pi_\phi}(o)$ and $\mathcal{F}_\el(o)$.
\begin{itemize}
\item $G=\R_+^*\times SO_\circ(1,2)$. This group acts on the both components of the nullcone $\Nu(o)$ transitively. One can see, $G$ acts on the connected complements of the complement of $\Nu(o)\cup \{o\}$ in $\En^{1,2}$ (which are open subsets) transitively. Also, $G$ acts on the both connected components $L(\hat{p})\setminus S_\infty$ transitively. Therefore, $G$ admits exactly four $2$-dimensional orbits in $\Ein^{1,2}$ which all of them are degenerate surfaces.
\item $G=SO_\circ(1,2)$. This group acts on the de-Sitter spaces $\dS^{1,1}(r)$, the hyperbolic planes $\Hn^{2}(r)$, and the connected components of the nullcone $\Nu(o)$ transitively. Furthermore, it acts on the both connected components of $L(\hat{p})\setminus S_\infty$ transitively. Hence, $G$ admits spacelike, Lorentzian, and degenerate $2$-dimensional orbits in $\Ein^{1,2}$.

\item $G=\Aff$. The leaves $\mathcal{F}_\el(o)$ and $\mathcal{F}_{\Pi_\phi}(o)$ are invariant by $G$. Also, $G$ preserves the photon $\phi$ in the lightcone $L(p)$, and acts on the connected components of $L(p)\setminus (\phi\cup S_\infty)$ transitively. Furthermore, the hyperbolic planes and the both connected components of $\Nu(o)\setminus \mathcal{F}_\el(o)$ are $G$-orbits. Moreover, for all $r\in \R_+^*$, $G$ acts on the both connected components of $\dS^{1,2}(r)\setminus \mathcal{F}_{\Pi_\phi}(o)$ transitively. Therefore, $G$ admits spacelike, Lorentzian, and degenerate $2$-dimensional orbits in $\Ein^{1,2}$.
\item $G=\R_+^*\times \Aff$. Observe that the connected components of $L(p)\setminus (\phi\cup S_\infty)$ are degenerate $2$-dimensional $G$-orbits. On the other hand, $G$ acts on each of the four connected components of $(\Nu(o)\cup \mathcal{F}_{\Pi_\phi}(o))\setminus \mathcal{F}_\el(o)$ transitively. In fact, $G$ admits exactly six $2$-dimensional orbits in $\Ein^{1,2}$, since it acts on the connected components of $\En^{1,2}\setminus (\Nu(o)\cup \mathcal{F}_{\Pi_\phi}(o))$ transitively. 

\item $G_a=\exp\big(\R(a+\Y_H)+\R\Y_P\big)$, $0<|a|<1$. One can see, $G$ acts on the both connected components of $L(p)\setminus (\phi\cup S_\infty)$ transitively. For an arbitrary point $q=(x,y,z)\in \En^{1,2}$, the vectors tangent to the orbit $G_a(q)$ at $q$ induced by the $1$-parameter subgroups $\exp\big(\R(a+Y_H)\big)$ and $Y_P$ are $v=(ax+y,x+ay,az)$ and $w=(z,z,x-y)$, respectively. The set $\{v,w\}\subset T_qG(q)$ is a basis if and only if $z\neq 0$ and $x\neq y$ if and only if $q\notin \mathcal{F}_{\el}(o)$. Thus, $G$ acts on the both connected components of $\mathcal{F}_{\Pi_\phi}(o)\setminus \mathcal{F}_\el(o)$, and on the both connected components of $\mathfrak{N}(o)\setminus \mathcal{F}_\el(o)$ transitively. Also, $G$ admits spacelike and Lorentzian $2$-dimensional orbits in $\En^{1,2}$, since for a point $q\in \En^{1,2}\setminus (\Nu(o)\cup \mathcal{F}_{\Pi_\phi(o)})$, the orthogonal space of the spacelike vector $w$ in the tangent space $T_qG(q)$ can be a spacelike or timelike line. 

\item $G=\exp\big(\R(1+\Y_H)+\R\Y_P\big)$. This group preserves the leaves of the foliation $\mathcal{F}_{\Pi_\phi}$, since it is a subgroup of $\mathsf{K}$ (Remark \ref{rem.ker}). It is not hard to see that $G$ acts on the both connected components of $L(p)\setminus (\phi\cup S_\infty)$ transitively. For an arbitrary point $q=(x,y,z)\in \En^{1,2}$, the vectors tangent to the orbit $G(q)$ at $q$ induced by the $1$-parameter subgroups $\exp\big(\R(1+Y_H)\big)$ and $Y_P$ are $v=(x+y,x+y,z)$ and $w=(z,z,x-y)$ respectively. The set $\{u,v\}\subset T_qG(q)$ is a basis if and only if $q\notin \mathfrak{N}(o)$. Hence, every $2$-dimensional $G$-orbit in $\En^{1,2}$ is a degenerate surface contained in a leaf of $\mathcal{F}_{\Pi_\phi}$.
\item $G=\exp\big(\R(-1+\Y_H)+\R\Y_P\big)$. In the one hand, $G$ acts on the lightlike affine line $\mathcal{F}_\el(o)\subset \mathcal{F}_{\Pi_\phi}(o)$ trivially. On the other hand, $G$ fixes the corresponding limit point of $\mathcal{F}_{\Pi_\phi}(o)$, $d\in \hat{\phi}$. It follows that, $G$ acts on the photon $\mathcal{F}_\el(o)\cup \{d\}$ trivially, and so, $G$ is a subgroup of $\mathsf{K}$ (Remark \ref{rem.ker}), up to conjugacy.

\item $G=\R_+^*\times Y_P$. One can see that $G$ admits the same orbits in the lightcone $L(p)$ as $\Aff$. For an arbitrary point $q=(x,y,z)\in \En^{1,2}$, the vectors tangent to the orbit $G(q)$ at $q$ induced by the $1$-parameter subgroups $\R_+^*$ and $Y_P$ are $w=(x,y,z)$ and $v=(z,z,x-y)$ respectively. The set $\{w,v\}\subset T_qG(q)$ is a basis if and only if $q\notin \mathcal{F}_\el(o)$. A $2$-dimensional orbit $G(q)$ in $\En^{1,2}$ is degenerate (resp. Lorentzian, spacelike) if and only if it lies in the nullcone $\Nu(o)$ or the degenerate affine plane $\mathcal{F}_{\Pi_\phi}(o)$ (resp. $\{\q<0\}$, $\{\q>0\}\setminus \mathcal{F}_{\Pi_\phi}(o)$).
\item $G=\R_+^*\times Y_H$. This group preserves both the foliations $\mathcal{F}_{\Pi_\phi}$ and $\mathcal{F}_{\Pi_\psi}$ in the Minkowski space $\En^{1,2}$, and so, it preserves the corresponding photons $\phi,\psi\subset L(p)$. One can see that, $G$ acts on the four connected components of $L(p)\setminus(\phi\cup \psi\cup S_\infty)$ transitively. Furthermore, $G$ preserves three affine lines in the Minkowski space $\En^{1,2}$, namely, $\mathcal{F}_{\el}(o)$, $\ell_0=\{(t,-t,0):t\in \R\}$, and $\ell_1=\{(0,0,t):t\in \R\}$. For an arbitrary point $q=(x,y,z)\in \En^{1,2}$, the vectors tangent to the orbit $G(q)$ at $q$ induced by the $1$-parameter subgroups $\R_+^*$ and $Y_H$ are $v=(x,y,z)$ and $w=(y,x,0)$, respectively. The set $\{w,v\}\subset T_qG(q)$ is a basis if and only if $q\notin (\mathcal{F}_\el(o)\cup\ell_0\cup \ell_1)$. Actually, $\{v,w\}\subset T_qG(q)$ is an orthogonal basis. It follows that, $G$ admits spacelike, Lorentzian, and degenerate orbits in $\En^{1,2}$, precisely, if the $2$-dimensional orbit lies in the region $\{q>0\}$ (resp. $\{q<0\}$, $\Nu(o)$), then it is spacelike (resp. Lorentzian, degenerate).

\item $G=\R^*_+\times Y_E$. It is clear that $G$ acts on the both connected components of $L(\hat{p})\setminus S_\infty$ transitively. For an arbitrary point $q=(x,y,z)\in \En^{1,2}$, the vectors tangent to the orbit $G(q)$ at $q$ induced by the $1$-parameter subgroups $\R_+^*$ and $Y_E$ are $w=(x,y,z)$ and $v=(0,z,-y)$ respectively. The set $\{w,v\}$ is a basis if and only if $x\neq 0$ or $z\neq 0$. Furthermore, $G$ admits spacelike, Lorentzian, and degenerate $2$-dimensional orbits in $\En^{1,2}$, since the orthogonal space of the spacelike vector $v$ in the tangent space $T_qG(q)$ can be a spacelike, timelike, or lightlike line
\end{itemize}

\textit{Special subgroups with trivial translation part}
Here, we describe the orbits of the two subgroups in Table \ref{table8} which have trivial translation part, but they are not subgroups of the linear group $\R^*_+\times SO_\circ(1,2)$, i.e., these groups fix no point in the Minkowski space $\En^{1,2}$.

 $\blacktriangleright\;G= \exp\big(\R(-1+\Y_H+e_1+e_2)+\R\Y_P\big)$. This group preserves the foliation $\mathcal{F}_{\Pi_\phi}$, and its corresponding photon $\phi\subset L(p)$, since it is a subgroup of $(\R^*_+\times\Aff)\ltimes \R(e_1+e_2)$. 
 For an arbitrary point $q=(x,y,z)\in \En^{1,2}$, the vectors tangent to the orbit $G(q)$ at $q$ induced by the $1$-parameter subgroups $H= \exp\big(\R(-1+\Y_H+e_1+e_2)\big)$ and $Y_P$ are $v=(-x+y+1,x-y+1,-z)$ and $w=(z,z,x-y)$, respectively. The subgroup $H$ also preserves the foliation $\mathcal{F}_{\Pi_\psi}$ and its corresponding photon $\psi\subset L(p)$. Observe that, for all $q\in \En^{1,2}$, the tangent vector $\in T_qG(q)$ does not belong to the degenerate plane $\Pi_\psi$. This implies that $H$ acts on the vertex-less photon $\hat{\psi}$ transitively, and consequently, the degenerate surface $L(p)\setminus \phi$ is a $G$-orbit. On the other hand, the set $\{u,v\}\subset T_qG(q)$  is a basis if and only if $x\neq y$ and $z\neq 0$ if and only if $q\notin \mathcal{F}_\el(o)$. It follows that a $2$-dimensional orbit $G(q)$ (where $q=(x,y,z)$) in the Minkowski space $\En^{1,2}$ is spacelike (resp. Lorentzian, degenerate) if and only if $x<y$ (resp. $x>y$, $x=y$).

 $\blacktriangleright\;G=\exp\big(\R(2+\Y_H)+\R(\Y_P+e_1-e_2)\big)$. Observe that $G(o)$ is a $1$-dimensional lightlike orbit, but it is not a geodesic. Actually, $G$ is the only (up to conjugacy) connected Lie subgroup of $\Conf(\En^{1,2})$ inducing a $1$-dimensional lightlike orbit in $\Ein^{1,2}$ which is not a lightlike geodesic. We conclude that, the action of $G$ on $\Ein^{1,2}$ is conjugate to the action of affine group $\Aff\subset \PSL(2,\R)$ within the irreducible action of $\PSL(2,\R)$ on Einstein universe described in \cite{Has}. Indeed, up to conjugacy, this is the action of the stabilizer (by the irreducible action of $\PSL(2,\R)$ on Einstein universe $\En^{1,2}$) of a projective homogeneous polynomial of degree $4$ in two variables with a real root of multiplicity $4$. So, we study the orbits induced by this action in the setting of \cite{Has}.
 
Consider $(XY^3,X^3Y,X^2Y^2)$ as a coordinate for the Minkowski patch $Min(Y^4)\subset \Ein^{1,2}$. The restriction of the quadratic form $\q$ on $Mink(Y^4)\approx\En^{1,2}$ with origin $o=(0,0,0)$ is 
\begin{align*}
\q(aXY^3+bX^3Y+cX^2Y^2)=-\frac{1}{2}ab+\frac{1}{6}c^2.
\end{align*}
 For a point $q=(x,y,z)\in Mink(Y^4)$, the orbit induced by $\Aff\subset \PSL(2,\R)$ at $q$ is

 \begin{align*}
 \Aff(q)=\left\lbrace(xe^{6t}+3ye^{4t}s^2-2ze^{5t}s-4e^{3t}s^3,ye^{2t}-4e^{t}s,ze^{4t}-3ye^{3t}s+6e^{2t}s^2):t,s\in \R\right\rbrace.
 \end{align*}
 The affine group $\Aff$ preserves the foliation $\mathcal{F}_{\Pi_\phi}$ on $Mink(Y^4)$ induced by the degenerate plane $\Pi_\phi=(XY^3)^\perp\leq \R^{1,2}$. Hence, it preserves the corresponding photon $\phi\subset L(Y^4)$. 
 The $1$-parameter hyperbolic subgroup $Y_H\subset \Aff$ preserves the foliation $\mathcal{F}_{\Pi_\psi}$ where $\Pi_\psi$ is the degenerate plane $(X^3Y)^\perp\leq \R^{1,2}$. Also, $Y_H$ preserves the leaf $\mathcal{F}_{\Pi_\psi}(o)$, and so, it fixes the corresponding limit point $d\in \hat{\psi}\subset L(Y^4)$. Therefore, $\Aff$ induces a $1$-dimensional spacelike orbit in $L(Y^4)$ at $d$. It is not hard to see that $\Aff$ acts on the both connected components of $L(Y^4)\setminus (\phi\cup \Aff(d))$ transitively. 
 
 Now, we describe the orbits induced by $\Aff$ in the Minkowski patch $Mink(Y^4)$. For an arbitrary point $q=(x,y,z)\in Mink(Y^4)$, the vectors tangent to the orbit $G(q)$ at $q$ induced by the $1$-parameter subgroups $Y_P$ and $Y_H$ are $v=(-2z,-4,-3y)$ and $w=(6x,2y,4z)$, respectively. Observe that $v$ is a non-vanishing vector. Hence, $\Aff$ fixes no point in $Mink(Y^4)$. One can see, $\Aff(o)$ is the only $1$-dimensional $\Aff$-orbit in Minkowski patch $Mink(Y^4)$, indeed $\Aff(o)=\mathcal{N}\setminus \{Y^4\}$ where $\mathcal{N}$ is the $1$-dimensional orbit induced by the irreducible action of $\PSL(2,\R)$ on $\Ein^{1,2}$ (see \cite[Theorem 1]{Has}). The affine group $\Aff$ admits spacelike, Lorentzian and degenerate $2$-dimensional orbits in $Mink(Y^4)$, for instance consider the orbits with a representative in the affine line $\ell=\{(t,-t,0):t\in \R\}$.

\begin{remark}
It is well-known that, by the proper action of a Lie group on a manifold, the quotient space is Hausdorff. However, according to Section \ref{orbit}, Section \ref{subsec.4.1.2}, and \cite[Theorem 1]{Has}, if $G$ is a connected Lie group which acts on Einstein universe $\Ein^{1,2}$ non-properly and with cohomogeneity one, then there are two distinct orbits $G(p)$ and $G(q)$ such that $G(p)$ accumulates to $G(q)$, i.e., $G(q)$ is in the closure of $G(p)$ in $\Ein^{1,2}$. Hence, every $G$-invariant open neighborhood in $\Ein^{1,2}$ around $G(q)$ contains $G(p)$ as well. This implies that the orbit space $\Ein^{1,2}/G$ is not Hausdorff. Therefore, the properness of a cohomogeneity one action on $\Ein^{1,2}$ is equivalent to the Hausdorff condition on the orbit space.
\end{remark}
\end{document}